\documentclass{article}
\usepackage{amsmath,amsthm,amssymb,cases,amscd}
\usepackage{indentfirst}
\usepackage{ascmac}
\usepackage{mathtools}
\usepackage[all]{xy}
\usepackage{yhmath}
\usepackage{mathrsfs}
\usepackage{textcomp}
\usepackage[toc,page]{appendix}
\usepackage{braket}
\usepackage{enumerate}
\usepackage{multienum}
\usepackage{graphicx}
\usepackage{mathtools}
\usepackage{calc}

\newcommand{\bigboxplus}{
  \mathop{
    \vphantom{\bigoplus} 
    \mathchoice
      {\vcenter{\hbox{\resizebox{\widthof{$\displaystyle\bigoplus$}}{!}{$\boxplus$}}}}
      {\vcenter{\hbox{\resizebox{\widthof{$\bigoplus$}}{!}{$\boxplus$}}}}
      {\vcenter{\hbox{\resizebox{\widthof{$\scriptstyle\oplus$}}{!}{$\boxplus$}}}}
      {\vcenter{\hbox{\resizebox{\widthof{$\scriptscriptstyle\oplus$}}{!}{$\boxplus$}}}}
  }\displaylimits 
}

\makeatletter
\def\widebreve{\mathpalette\wide@breve}
\def\wide@breve#1#2{\sbox\z@{$#1#2$}%
     \mathop{\vbox{\m@th\ialign{##\crcr
\kern0.08em\brevefill#1{0.8\wd\z@}\crcr\noalign{\nointerlineskip}%
                    $\hss#1#2\hss$\crcr}}}\limits}
\def\brevefill#1#2{$\m@th\sbox\tw@{$#1($}%
  \hss\resizebox{#2}{\wd\tw@}{\rotatebox[origin=c]{90}{\upshape(}}\hss$}
\makeatletter

\makeatletter
    
    \@addtoreset{equation}{section}
\makeatother

\usepackage[top=25truemm,bottom=25truemm,left=18truemm,right=18truemm]{geometry} 
\allowdisplaybreaks
\theoremstyle{plain}
\newtheorem{definition}{Definition}[section]
\newtheorem{lemma}{Lemma}[section]
\newtheorem{proposition}[lemma]{Proposition}
\newtheorem{theorem}[lemma]{Theorem}
\newtheorem{corollary}[lemma]{Corollary}

\newtheorem{conjecture}[lemma]{Conjecture}
\newtheorem{hypothesis}[lemma]{Hypothesis}
\newtheorem{theorem*}[lemma]{Theorem*}
\newtheorem{lemma*}[lemma]{Lemma*}
\newtheorem{proposition*}[lemma]{Proposition*}

\theoremstyle{definition}
\newtheorem*{remark}{Remark}

\DeclarePairedDelimiter{\abs}{\lvert}{\rvert}

\newcommand{\lra}{\longrightarrow}
\newcommand{\llra}{\longleftrightarrow}
\newcommand{\into}{\hookrightarrow}
\newcommand{\onto}{\twoheadrightarrow}
\newcommand{\id}{\operatorname{id}}
\newcommand{\myim}{\operatorname{Im}}
\newcommand{\myre}{\operatorname{Re}}
\newcommand{\mmatrix}[4]{\begin{pmatrix} #1 & #2 \\ #3 & #4 \end{pmatrix}}

\newcommand{\tp}[1]{\prescript{\mathrm t}{}{#1}}
\newcommand{\diag}{\operatorname{diag}}
\newcommand{\Lgp}[1]{\prescript{L}{}{#1}}

\newcommand{\tr}{\operatorname{tr}}

\newcommand{\cent}{\operatorname{Cent}}
\newcommand{\inv}{^{-1}}
\newcommand{\an}[1]{\langle #1 \rangle}

\newcommand{\der}{\operatorname{der}}

\newcommand{\sgn}{\operatorname{sgn}}

\newcommand{\A}{\mathbb{A}}
\newcommand{\C}{\mathbb{C}}

\newcommand{\Q}{\mathbb{Q}}
\newcommand{\R}{\mathbb{R}}
\newcommand{\Z}{\mathbb{Z}}

\newcommand{\bfJ}{\mathbf{J}}

\newcommand{\calA}{\mathcal{A}}
\newcommand{\calC}{\mathcal{C}}
\newcommand{\calE}{\mathcal{E}}
\newcommand{\calG}{\mathcal{G}}
\newcommand{\calH}{\mathcal{H}}
\newcommand{\calI}{\mathcal{I}}
\newcommand{\calL}{\mathcal{L}}
\newcommand{\calM}{\mathcal{M}}

\newcommand{\calR}{\mathcal{R}}
\newcommand{\calS}{\mathcal{S}}

\newcommand{\calX}{\mathcal{X}}
\newcommand{\calY}{\mathcal{Y}}

\newcommand{\fraka}{\mathfrak{a}}
\newcommand{\frake}{\mathfrak{e}}
\newcommand{\frakg}{\mathfrak{g}}

\newcommand{\frakn}{\mathfrak{n}}
\newcommand{\frakp}{\mathfrak{p}}

\newcommand{\frakt}{\mathfrak{t}}

\newcommand{\frakN}{\mathfrak{N}}
\newcommand{\frakS}{\mathfrak{S}}

\newcommand{\Frob}{\operatorname{Frob}}

\newcommand{\GL}{\operatorname{GL}}
\newcommand{\PGL}{\operatorname{PGL}}
\newcommand{\Sp}{\operatorname{Sp}}
\newcommand{\Mp}{\operatorname{Mp}}

\newcommand{\SO}{\operatorname{SO}}
\newcommand{\Or}{\operatorname{O}}

\newcommand{\SU}{\operatorname{SU}}
\newcommand{\SL}{\operatorname{SL}}

\newcommand{\Mat}{\operatorname{M}}

\newcommand{\so}{\mathfrak{so}}

\newcommand{\disc}{\mathrm{disc}}
\newcommand{\temp}{\mathrm{temp}}
\newcommand{\Irr}{\operatorname{Irr}}
\newcommand{\cusp}{\mathrm{cusp}}
\newcommand{\scusp}{\mathrm{scusp}}
\newcommand{\simple}{\mathrm{sim}}
\newcommand{\bdd}{\mathrm{bdd}}
\newcommand{\el}{\mathrm{ell}}
\newcommand{\unit}{\mathrm{unit}}
\newcommand{\exc}{\mathrm{exc}}
\newcommand{\reg}{\mathrm{reg}}
\newcommand{\semisimple}{\mathrm{ss}}
\newcommand{\srss}{\mathrm{srss}}
\newcommand{\aut}{\mathrm{aut}}
\newcommand{\iso}{\mathrm{iso}}

\newcommand{\weylbc}[1]{\mathfrak{S}_{#1} \ltimes (\Z/2\Z)^{#1}}

\newcommand{\Hom}{\operatorname{Hom}}
\newcommand{\End}{\operatorname{End}}
\newcommand{\Aut}{\operatorname{Aut}}

\newcommand{\Sym}{\operatorname{Sym}}

\newcommand{\Out}{\operatorname{Out}}
\newcommand{\ad}{\operatorname{ad}}
\newcommand{\Ad}{\operatorname{Ad}}
\newcommand{\Gal}{\operatorname{Gal}}
\newcommand{\Res}{\operatorname{Res}}

\title{The endoscopic classification of representations of non-quasi-split odd special orthogonal groups}
\author{Hiroshi Ishimoto\footnote{ishimoto.hiroshi.55m@gmail.com}}
\date{[\today]}

\begin{document}

\maketitle

\begin{abstract}
  In an earlier book of Arthur, the endoscopic classification of representations of quasi-split orthogonal and symplectic groups was established.
  Later Mok gave that of quasi-split unitary groups.
  After that, Kaletha, Minguez, Shin, and White gave that of non-quasi-split unitary groups for generic parameters.
  In this paper we prove the endoscopic classification of representations of non-quasi-split odd special orthogonal groups for generic parameters, following Kaletha, Minguez, Shin, and White.
\end{abstract}

\tableofcontents

\section{Introduction}\label{intr}
\subsection*{Background \& Main result}
In a 2013 book \cite{art13}, Arthur established the endoscopic classification of irreducible representations of quasi-split special orthogonal and symplectic groups over local fields of characteristic zero and of automorphic representations of those groups over global fields of characteristic zero.
To local $A$-parameters he attached local $A$-packets characterized by the endoscopic character relation (which we shall call ECR for short), and proved that the local $A$-packets for generic parameters give the local Langlands correspondence (which we shall call LLC for short), and that automorphic representations are classified in terms of automorphic cuspidal representations of general linear groups and local $A$-packets.

Let us roughly recall his result.
Let first $F$ be a local field of characteristic zero, and $G$ a split odd special orthogonal group over $F$ for simplicity.
A local $A$-parameter for $G$ is a homomorphism
\begin{equation*}
    \psi:L_F\times\SU(2,\R) \to \Lgp{G},
\end{equation*}
with some conditions, where $L_F$ and $\Lgp{G}$ denote the Langlands group of $F$ and the $L$-group of $G$, respectively.
For every $A$-parameter $\psi$, Arthur first constructed a multiset $\Pi_\psi(G)$ over the set $\Pi_\unit(G)$ of equivalence classes of irreducible unitary representations of $G(F)$, and a mapping $\Pi_\psi(G)\to\Irr(\pi_0(S_\psi))$, where $S_\psi$ denotes the centralizer of $\psi$ in the Langlands dual group $\widehat{G}$ of $G$, and then he proved that they satisfy ECR.
An $A$-parameter is called a bounded $L$-parameter if its restriction to $\SU(2,\R)$ is trivial.
He also proved LLC by showing that if $\psi$ is a bounded $L$-parameter then the packet $\Pi_\psi(G)$ gives the $L$-packet, which has been expected to exist.
Let us recall here the basic form of LLC, which is a conjecture in general.
\begin{conjecture}[LLC]
    Let $G$ be a connected reductive algebraic group over a local field $F$.
    Let $\Pi_\temp(G)$ be the set of equivalence classes of irreducible tempered representations of $G(F)$, and $\Phi_\bdd(G)$ the set of equivalence classes of bounded $L$-parameters.
    There exists a canonical map
    \begin{align*}
        LL : \Pi_\temp(G) \longrightarrow \Phi_\bdd(G),
    \end{align*}
    such that for each $\phi \in \Phi_\bdd(G)$, the fiber $\Pi_\phi(G) = LL\inv(\phi)$ is a finite set and there is an injective map $\iota_\phi:\Pi_\phi(G)\to\Irr(\pi_0(S_\phi))$.
    These two correspondences satisfy some interesting properties.
    The finite set $\Pi_\phi(G)$ is called the $L$-packet of $\phi$.
\end{conjecture}
We refer the reader to \cite{llcnqs} for details.

On the other hand let next $F$ be a number field, and $G$ a split odd special orthogonal group over $F$.
A global $A$-parameter for $G$ is a formal finite unordered sum
\begin{equation*}
    \psi=\bigboxplus_i \pi_i \boxtimes \nu_i,
\end{equation*}
with some conditions, where $\pi_i$ is an irreducible cuspidal automorphic representation of a general linear group, and $\nu_i$ is an irreducible finite dimensional representation of $\SU(2,\R)$.
For all place $v$ of $F$, we have the localization $\psi_v=\bigoplus_i \phi_{i,v}\boxtimes\nu_i$, where $\phi_{i,v}$ is the unique $L$-parameter for a general linear group over $F_v$ corresponding to $\pi_{i,v}$ via LLC.
Then $\psi_v$ is a local $A$-parameter for $G_v$ over $F_v$. (Strictly speaking, it may not be an $A$-parameter, but a packet for it is defined.)
Arthur determined an appropriate subset $\Pi_\psi(\varepsilon_\psi)$ of $\{\bigotimes_v \pi_v \mid \pi_v \in \Pi_{\psi_v}(G_v)\}$ and showed the decomposition
\begin{equation}\label{ys}
    \begin{gathered}
        L^2_{\disc}(G(F)\backslash G(\A_F))=\bigoplus_\psi L^2_{\disc,\psi}(G(F)\backslash G(\A_F)),\\
        L^2_{\disc,\psi}(G(F)\backslash G(\A_F))=\bigoplus_{\pi\in\Pi_\psi(\varepsilon_\psi)} \pi,
    \end{gathered}
\end{equation}
of the discrete spectrum into near equivalence classes, and then into irreducible automorphic representations.
We shall call a decomposition like \eqref{ys} "Arthur's multiplicity formula", and abbreviate it as AMF.

Later, Mok \cite{mok} proved ECR, LLC, and AMF for quasi-split unitary groups by the similar argument, and Kaletha-Minguez-Shin-White \cite{kmsw} partially proved those for non-quasi-split unitary groups.
In this paper, following \cite{kmsw}, we shall partially prove the analogous classification (ECR, LLC, and AMF) for non-quasi-split odd special orthogonal groups by the similar argument.
In other words, our main theorem (Theorem \refeq{main}) is the following:
\begin{theorem}
    Over a local field of characteristic zero, LLC holds for any non-quasi-split odd special orthogonal group, and the $L$-packets equipped with mappings $\iota_\phi$ satisfy ECR.
    Over a global field of characteristic zero, AMF holds for any non-quasi-split odd special orthogonal group, except the irreducible decompositions of $L^2_{\disc,\psi}(G(F)\backslash G(\A_F))$ for non-generic parameters $\psi$.
\end{theorem}

The most important difference between this paper and \cite{kmsw} is the proof of the local intertwining relation.
The local intertwining relation (\cite[Theorem* 2.6.2]{kmsw}, Theorem \ref{2.6.2} in this paper), for which we shall write LIR for short, is a key theorem in the proof of the local main theorems ECR and LLC.
By the similar idea to \cite{kmsw}, we can reduce the proof of LIR to that for two special cases: the real special orthogonal groups $\SO(1,4)$ and $\SO(2,5)$ relative to Levi subgroups isomorphic to $\GL_{1}\times\SO(0,3)$ and $\GL_{2}\times\SO(0,3)$ respectively, and the special parameters.
In the case of $\SO(1,4)$, the special representation of the Levi subgroup under consideration is the trivial representation, and hence we can prove LIR by the similar argument to \S2.9 in loc.cit.
However, in the case of $\SO(2,5)$, the situation is too complicated to calculate similarly to loc. cit., since the representation of the Levi subgroup is infinite dimensional.
For this reason in this paper, we shall prove them by a completely different argument.
We choose a test function using the Iwasawa decomposition, while the relative Bruhat decomposition was used in loc. cit.
The argument will appear in \S\ref{2.9}.

We remark that LLC of non-quasi-split odd special orthogonal groups has already studied by M{\oe}glin-Renard \cite{mr}, but their result does not contain LIR.
Thus our local theorem is differentiated from their work.

\subsection*{Application}
In \cite[Chapter 9]{art13}, Arthur formulated the classification for non-quasi-split symmetric and orthogonal groups, and he has the intention of proving it.
The results of this paper are included in his project, but we have a different motivation, the representation theory of the metaplectic groups.
It is one of the most important application of this paper.
Let us recall the results on the metaplectic groups by Adams, Barbasch, Gan, Savin, and Ichino.

Let $F$ be a local field.
The metaplectic group, denoted by $\Mp_{2n}(F)$, is a unique nonlinear two-fold cover of $\Sp_{2n}(F)$ except $F=\C$, in which case $\Mp_{2n}(\C)=\Sp_{2n}(\C)\times\{\pm1\}$.
We identify $\Mp_{2n}(F)$ with $\Sp_{2n}(F) \times \{\pm1\}$ as sets, and a representation $\pi$ of $\Mp_{2n}(F)$ is said to be genuine if $\pi((1,-1))$ is not trivial.
Let $\Pi_\temp(\Mp_{2n})$ denote the set of equivalence classes of genuine tempered irreducible representations of $\Mp_{2n}(F)$.
Adams-Barbasch \cite{ab95,ab98}, Adams \cite{ad} (archimedean case), and Gan-Savin \cite{gs} (non-archimedean case) showed that the local theta correspondence gives a canonical bijection
\begin{equation*}
    \Pi_\temp(\Mp_{2n}) \llra \bigsqcup_V \Pi_\temp(\SO(V)),
\end{equation*}
where $V$ runs over all $(2n+1)$-dimensional quadratic space of discriminant 1 over $F$.
Thanks to their results, LLC for $\Mp_{2n}$ is implied by that for all $\SO(V)$.
In addition, there is an article \cite{ishi} which proved LIR for $\Mp_{2n}$ assuming that for all $\SO(V)$ over a $p$-adic field.

Let next $F$ be a number field.
The metaplectic group $\Mp_{2n}(\A_F)$ is a nontrivial two-fold cover of $\Sp_{2n}(\A_F)$, and there is a canonical injective homomorphism $\Sp_{2n}(F) \into \Mp_{2n}(\A_F)$.
Hence the notions of "automorphic representations" and "discrete spectrum" are defined in a canonical way.
AMF for the metaplectic group was studied by Gan-Ichino \cite{gi18}.
They proved the decomposition of the discrete spectrum of $\Mp_{2n}$ into near equivalence classes without any assumption, and proved the decomposition into irreducible automorphic representations of the generic part assuming that for all non-quasi-split odd special orthogonal groups.

Those theories will be completed by this paper.
Namely, LLC and LIR for the metaplectic groups will hold true, and the result of Gan-Ichino \cite{gi18} on the generic part of AMF will be unconditional.

\subsection*{Organization}
\S\ref{0} is the preliminary section, where some notations for odd special orthogonal groups are established.

In \S\ref{1} we first recall the notions of endoscopic triple, transfer factor, local and global parameters, and the canonical correspondence $(\frake,\psi^\frake)\leftrightarrow(\psi,s)$.
Next we shall state LLC, ECR, and AMF more precisely.
We will recall the result of Arthur \cite{art13} and state our main theorem (Theorem \ref{main}).

\S\ref{2} is the most important section in this paper.
We define the local intertwining operator, state LIR, and reduce its proof to the case when the parameter is discrete for $M^*$ and elliptic or exceptional for $G^*$, following \cite{kmsw, art13, mok}.
Then in \S\ref{2.9} we give a proof of LIR for the special case explained above.

In \S\ref{3}, we will recall the global theory on the trace formula, and obtain some lemmas.
Proofs of some lemmas are omitted since they are quite similar to those in \cite[\S3]{kmsw}.

In \S\ref{4}, we complete the proof of the local main theorem by the argument involving globalizations and trace Paley-Wiener theorem.

In \S\ref{5}, we complete the proof of the global main theorem.

\subsection*{Convention \& Notation}
We marked some theorems, lemmas, and propositions with symbol * to indicate that they are only proved for generic parameters in this paper.
(We omit * when we refer to them.)
The author expect that their proofs for non-generic parameters will be completed by an analogous argument to a sequel of \cite{kmsw}.

We do not use "$\calS$" in this paper.
Instead, we shall use "$\frakS$" to denote the component groups.\\

In this paper every field is assumed to be of characteristic zero.
In particular, a local field is either $\R$, $\C$, or a finite extension of $\Q_p$ for some prime number $p\in\Z$, and a global field is a number field, i.e., a finite extension of $\Q$.
For a field $F$, we write $\overline{F}$ for its algebraic closure, and $\Gamma=\Gamma_F$ for its absolute Galois group $\Gal(\overline{F}/F)$.
For a connected reductive algebraic group $G$ over $F$, we write $e(G)$ for the Kottwitz sign (\cite{kot83}) of $G$, and $\widehat{G}$ for the dual group over $\C$.
If moreover $F$ is a local or global field, we write $W_F$ for the absolute Weil group of $F$, and the Weil form of the $L$-group is defined by $\Lgp{G}=\widehat{G}\rtimes W_F$.

If $F$ is a number field, we write $\A_F$ for the ring of adeles of $F$.
We often fix a nontrivial additive character of $F\backslash \A_F$, which always denoted by $\psi_F$.
For each place $v$ of $F$, we write $\psi_{F,v}$ for the local component of $\psi_F$ at $v$.
We abbreviate $\Gamma_{F_v}$ as $\Gamma_v$.
Following \cite{art13} and \cite{kmsw}, we do not use a symbol $\bigotimes_v'$ for a restricted tensor product. We simply write $\bigotimes_v$.
Similarly, we write $\prod_v$ in place of $\prod_v'$.
Unless otherwise specified, $\prod_v$, $\sum_v$, $\bigoplus_v$, and $\bigotimes_v$ denote the certain operations taken over all places $v$ of $F$, respectively.

If $F$ is a local field, the Langlands group is defied as follows.
\begin{align*}
    L_F=
    \begin{dcases*}
        W_F \times \SU(2,\R), & if $F$ is non-archimedean, \\
        W_F,                  & if $F$ is archimedean.
    \end{dcases*}
\end{align*}
As in the global case, we often fix a nontrivial additive character of $F$, which always denoted by $\psi_F$.

For an algebraic or abstract group $G$, its center is denoted by $Z(G)$.
In addition, when $X$ is a subgroup or an element of $G$, we write $\cent(X,G)$, $Z_X(G)$, or $Z(X,G)$ (resp. $N_G(X)$ or $N(X,G)$) for the centralizer (resp. normalizer) of $X$ in $G$.

For a topological group $G$, its connected component of the identity element is denoted by $G^\circ$, and we put $\pi_0(G)=G/G^\circ$.
For a finite group $A$, we shall write $\Irr(A)$ for the set of equivalence classes of irreducible $\C$-representations of $A$.

For an algebraic group $G$ over a field $F$, put $X^*(G)=\Hom(G,\GL_1)$ and $X_*(G)=\Hom(\GL_1,G)$, which are equipped with the $F$-structure.
Put moreover
\begin{align*}
    \fraka_G      & =\Hom(X^*(G)_F, \R), & \fraka_G^*      & =X^*(G)_F\otimes_\Z \R, \\
    \fraka_{G,\C} & =\Hom(X^*(G)_F, \C), & \fraka_{G,\C}^* & =X^*(G)_F\otimes_\Z \C.
\end{align*}

For any representation $\pi$, let $\pi^\vee$ denote its contragredient representation.
If $\pi$ is a representation of a topological group $G$ with a Haar measure $dg$, for a function $f$ on $G$, let $f_G(\pi)$ denote its character i.e.,
\begin{align*}
    f_G(\pi) & =\tr(\pi(f)), &  & \text{where} & \pi(f) & =\int_G f(g)\pi(g)dg.
\end{align*}
We also write $f(\pi)$ if there is no danger of confusion.

In this paper we shall write $E_{i,j}$ for the $(i,j)$-th matrix unit.
For a positive integer $N$, we put
\begin{align*}
    J=J_N=\left( \begin{array}{cccccc}
                      &  &  &  &  & 1 \\ &&&&-1&\\ &&&1&&\\ &&\adots&&&\\ &(-1)^{N-2}&&&&\\ (-1)^{N-1}&&&&&
                 \end{array} \right)
    \in \GL_N,
\end{align*}
and define an automorphism $\theta_N$ of $\GL_{N}$ by $\theta_N(g)=J\tp{g}\inv J\inv$.
Then the standard pinning $(T_N,B_N,\{E_{i,i+1}\}_{i=1}^{N-1})$, where $T_N$ is the maximal torus consisting of diagonal matrices and $B_N$ is the Borel subgroup consisting of upper triangular matrices, is $\theta_N$-stable.
The dual group for $\GL_N$ is $\GL(N,\C)$, and the automorphism $\widehat{\theta}_N$ of $\GL(N,\C)$, which is dual to $\theta_N$, is given by $\widehat{\theta}_N(g)=J\tp{g}\inv J\inv$.
As $\GL_N$ is split, the Galois action on $\GL(N,\C)$ is trivial.

\subsection*{Acknowledgment}
  The author would like to thank his doctoral advisor Atsushi Ichino for his helpful advice.
  He also thanks the co-advisor Wen-Wei Li for his helpful comments.
  In addition, he also thanks Masao Oi and Hirotaka Kakuhama for sincere and useful comments.
  This work was partially supported by JSPS Research Fellowships for Young Scientists KAKENHI Grant Number 20J10875 and JSPS KAKENHI Grant Number 22K20333.
  The author also would like to appreciate Naoki Imai for his great support by JSPS KAKENHI Grant Number 22H00093.

\section{Odd special orthogonal groups}\label{0}
In this section, we establish some notations for the odd special orthogonal groups, and recall the Kottwitz map.
In the third subsection, we shall describe the real case as a preparation for the proof of Lemma \ref{2.2.1}.
\subsection{Split odd special orthogonal groups $\SO_{2n+1}$}\label{0.1}
Let $F$ be any field of characteristic zero, and $n$ a non-negative integer.
We shall write $\SO_{2n+1}$ for the split odd special orthogonal group over $F$ of size $(2n+1)$, which is defined by
\begin{equation*}
    \SO_{2n+1}=\Set{g\in\GL_{2n+1} | \tp{g}\begin{pmatrix}&&1_n\\&2&\\1_n&&\end{pmatrix}g=\begin{pmatrix}&&1_n\\&2&\\1_n&&\end{pmatrix}}.
\end{equation*}
We put 2 at the center to make root vectors simple.
Its Lie algebra is realized as
\begin{equation*}
    \so_{2n+1}=\Set{X \in \Mat_{2n+1} | \tp{X}\begin{pmatrix}&&1_n\\&2&\\1_n&&\end{pmatrix}+\begin{pmatrix}&&1_n\\&2&\\1_n&&\end{pmatrix}X = 0 },
\end{equation*}
over $F$.
Let us fix the standard Borel subgroup and the standard maximal torus
\begin{align*}
    B^* & = \Set{\begin{pmatrix}a&*&* \\ &1&* \\ &&\tp a\inv \end{pmatrix}\in \SO_{2n+1} | a\in \GL_n, \text{upper triangular} }, \\
    T^* & = \Set{ t=\operatorname{diag} (t_1,\ldots,t_n,1,t_1\inv ,\ldots,t_n\inv ) | t_i \in \GL_1},
\end{align*}
and let $\chi_i$ denote the element of $X^*(T^*)$ such that $\chi_i(t)=t_i$, for $i=1,\ldots,n$.
We shall define simple roots $\alpha_i$ and simple root vectors $X_{\alpha_i}$ so that
\begin{align*}
    R(T^*,\SO_{2n+1}) & = \set{\pm(\chi_i-\chi_j), \pm(\chi_i+\chi_j) }_{1\leq i<j \leq n} \cup \set{\pm\chi_i }_{1\leq i\leq n}, \\
    \Delta(B^*)       & = \Set{\alpha_1,\dots,\alpha_n }                                                                          \\
                      & =\set{\chi_1-\chi_2,\chi_2-\chi_3,\dots,\chi_{n-1}-\chi_n, \chi_n},                                       \\
    X_{\alpha_i}      & =E_{i,i+1}-E_{n+1+i+1,n+1+i}, \qquad 1\leq i\leq n-1,                                                     \\
    X_{\alpha_n}      & =2E_{n,n+1}-E_{n+1,2n+1},
\end{align*}
where $R(T^*,\SO_{2n+1})$ and $\Delta(B^*)$ denote the root system of $T^*$ in $\SO_{2n+1}$ and the set of simple roots with respect to $B^*$, respectively.
Here we write the group law on $X^*(T^*)$ additively.
Put
\begin{align*}
    H_i & = E_{i,i}-E_{i+1,i+1}-E_{n+1+i,n+1+i}+E_{n+1+i+1,n+1+i+1}, & \text{for } & 1\leq i\leq n-1, \\
    H_n & = 2E_{n,n}-2E_{2n+1,2n+1},                                 & \text{for } & i=n,
\end{align*}
and
\begin{align*}
    X_{-\alpha_i} & = E_{i+1,i}-E_{n+1+i,n+1+i+1}, &  & \text{for } 1\leq i\leq n-1, & \\
    X_{-\alpha_n} & = E_{n+1,n}-2E_{2n+1,n+1},     &  & \text{for } i=n.             &
\end{align*}
It can be seen that $X_{-\alpha_i}$ is a root vector corresponding to a negative root $-\alpha_i$.
Also one has $d\alpha_i(H_i) = 2$ and $[X_{\alpha_i}, X_{-\alpha_i}] = H_i$, so $H_i$ is the coroot for $\alpha_i$ and $X_{-\alpha_i}$ is the opposite root vector to $X_{\alpha_i}$.
In this paper, by the standard pinning of $\SO_{2n+1}$, we mean the pinning $(T^*,B^*,\{X_\alpha\}_{\alpha\in\Delta(B^*)})$.

The dual group of $\SO_{2n+1}$ is the symplectic group $\Sp(2n,\C)$.
We shall choose a standard pinning of $\Sp(2n,\C)$ dual to the standard pinning of $\SO_{2n+1}$, to be the same as in \cite[\S7.1]{ishi}.

\subsection{Inner forms of odd special orthogonal groups}\label{0.3}
Let $F$ be any field of characteristic zero, and $G^*=\SO_{2n+1}$ the split special orthogonal group over $F$.
Since any automorphism of $G^*$ is an inner automorphism, the isomorphism classes of inner forms of $G^*$ and those of inner twists of $G^*$ are bijectively corresponding to each other in the natural way.
Moreover, since $G^*_{\ad}=G^*$ and $Z(G^*)=\{1\}$, the isomorphism classes of inner twists, of pure inner twists, and of extended pure inner twists of $G^*$ are also bijectively corresponding to each other in the natural way.
We refer the reader to \cite[\S0.3.1]{kmsw} for the definitions and basic properties of inner twists and so on.
Hence there are bijective correspondences between the sets of isomorphism classes of inner forms, inner twists, and pure inner twists of $G^*$, and the Galois cohomology set $H^1(F,G^*)$:
\begin{align*}
    \begin{array}{ccccccc}
        \left\{ \text{inner forms} \right\}/\simeq & \llra   & \left\{ \text{inner twists} \right\}/\simeq & \llra   & \left\{ \text{pure inner twists} \right\}/\simeq & \llra   & H^1(F,G^*), \\
        G                                          & \mapsto & \xi:G^*\to G                                & \mapsto & (\xi,z):G^*\to G                                 & \mapsto & z.
    \end{array}
\end{align*}
By abuse of terminology we often identify them.
We remark that for any $(2n+1)$-dimensional quadratic space $V$ over $F$, its special orthogonal group $\SO(V)$ is an inner form of $G^*$, and every inner form of $G^*$ is isomorphic to $\SO(V)$ for some quadratic space $V$ over $F$.

Let $G$ be an inner form of $G^*$, and $M^*\subset G^*$ a Levi subgroup.
We say that $M^*$ transfers to $G$ if there exists an inner twist $\xi:G^*\to G$ such that $M:=\xi(M^*)\subset G$ is a Levi subgroup over $F$ and the restriction $\xi|_{M^*} : M^*\to M$ is an inner twist over $F$.
When $M^*$ transfers to $G$, we will always choose an inner twist $\xi$ as above.\\

Let $F$ be a local field.
Kottwitz \cite[Theorem 1.2]{kot86} gave a canonical map of pointed sets
\begin{equation}\label{kot}
    \alpha:H^1(F,G^*)\to \pi_0(Z(\widehat{G^*})^{\Gamma})^D,
\end{equation}
extending the Tate-Nakayama isomorphism, where $D$ denotes the Pontryagin dual.
For any pure inner twist $(\xi,z);G^*\to G$, which we regard as an element of $H^1(F,G^*)$, put $\chi_G=\alpha(G)$.
We shall also write $\an{-,z}$ or $\an{-,\xi}$ for it.
Note that in our case we have $\pi_0(Z(\widehat{G^*})^{\Gamma})^D=Z(\Sp(2n,\C))^D\simeq\Z/2\Z$.

If $F$ is non-archimedean, then there are only two inner forms of $G^*$.
One is non-quasi-split, and the other is split.
In this case the map $\alpha$ is the unique isomorphism of pointed sets.
If $F=\C$, then $H^1(\C,G^*)$ is singleton, and hence $\alpha$ is trivial.
Suppose that $F=\R$.
For non-negative integers $p$ and $q$ with $p+q=2n+1$, put
\begin{equation*}
    \SO(p,q)=\Set{g\in\GL_{2n+1}(\R) | \tp{g}\begin{pmatrix}1_p&\\&-1_q\end{pmatrix}g=\begin{pmatrix}1_p&\\&-1_q\end{pmatrix}}.
\end{equation*}
Then $\SO(p,q)$ is an inner form of $G^*$, and every inner form is isomorphic to $\SO(p,q)$ for some $(p,q)$.
Since $\SO(p_1,q_1)$ and $\SO(p_2,q_2)$ are isomorphic if and only if $(p_1,q_1)=(p_2,q_2)$ or $(q_2,p_2)$, we can naturally identify $H^1(\R,\SO_{2n+1})$ with $\{(p,q) \mid p,q\in\Z_{\geq0},\ p+q=2n+1\}/{(p,q)\sim(q,p)}$.
Moreover, It is known that the Kottwitz map is characterized by
\begin{equation*}
    \ker \alpha = \Set{\SO(p,q) | p-q\equiv1,7 \mod{8}}.
\end{equation*}

Let $F$ be a global field.
For each place $v$ of $F$ there is a localization map $H^1(F,G^*)\to H^1(F_v,G^*)$.
Let $\alpha_v$ denote the Kottwitz map \eqref{kot} from $H^1(F_v,G^*)$ to $\pi_0(Z(\widehat{G^*})^{\Gamma_v})^D=\pi_0(Z(\widehat{G^*})^{\Gamma})^D$.
By \cite[Proposition 2.6]{kot86} we have an exact sequence
\begin{equation}\label{033}
    H^1(F,G^*) \lra \bigoplus_v H^1(F_v,G^*) \overset{\sum_v\alpha_v}{\lra} \pi_0(Z(\widehat{G^*})^{\Gamma})^D.
\end{equation}
\begin{remark}
    Such an exact sequence holds for general $G^*$.
\end{remark}

\subsection{On the real odd special orthogonal groups $\SO(p,q)$}\label{reallie}
Let us describe the real special orthogonal groups $\SO(p,q)$ and their Lie algebras $\so(p,q)$ more explicitly.
Let $p>q$ be non-negative integers with $p+q=2n+1$, and put $r=p-n-1=n-q$.
Put
\begin{align*}
    S_0=\left(\begin{array}{ccc} 1_n&&1_n\\ &2&\\ 1_n&&-1_n \end{array}\right).
\end{align*}
Then an assignment $g \mapsto S_0gS_0\inv$ determines an isomorphism from $\SO_{2n+1}$ to $\SO(n+1,n)$ over $\R$.
Put next
\begin{align*}
    S_{p,q}'=\left(\begin{array}{ccc}1_{n+1}&&\\ &-i1_r&\\ &&1_q\end{array}\right).
\end{align*}
Then an assignment $g \mapsto S_{p,q}'g{S_{p,q}'}\inv$ determines an isomorphism from $\SO(n+1,n)$ to $\SO(p,q)$ over $\C$.
We thus obtain an inner twist $\xi=\xi_{p,q} : G^*=\SO_{2n+1}\to\SO(p,q)$ given by $g \mapsto S_{p,q}gS_{p,q}\inv$, where $S_{p,q}=S_{p,q}'S_0$.

The standard pinning $(T^*, B^*, \{X_\alpha\}_\alpha)$ of $G^*$ gives a Chevalley basis of $\frakg^*=\so_{2n+1}$:
\begin{align*}
    [X_\beta, X_\gamma]=\pm(b+1)X_{\beta+\gamma},
\end{align*}
where $b$ is the greatest positive integer such that $\gamma-b\beta$ is a root.
Explicitly,
\begin{align*}
    X_{\chi_i-\chi_j}
     & =E_{i,j}-E_{n+1+j,n+1+i},                          & \text{for } & 1\leq i<j\leq n, \\
    X_{\chi_i+\chi_j}
     & =E_{i,n+1+j}-E_{j,n+1+i},                          & \text{for } & 1\leq i<j\leq n, \\
    X_{\chi_i}
     & =2E_{i,n+1}-E_{n+1,n+1+i},                         & \text{for } & 1\leq i\leq n,   \\
    \\
    H_{\chi_i-\chi_j}
     & = E_{i,i}-E_{j,j}-E_{n+1+i,n+1+i}+E_{n+1+j,n+1+j}, & \text{for } & 1\leq i<j\leq n, \\
    H_{\chi_i+\chi_j}
     & = E_{i,i}+E_{j,j}-E_{n+1+i,n+1+i}-E_{n+1+j,n+1+j}, & \text{for } & 1\leq i<j\leq n, \\
    H_{\chi_i}
     & = 2E_{i,i}-2E_{n+1+i,n+1+i},                       & \text{for } & 1\leq i\leq n,   \\
    \\
    X_{-(\chi_i-\chi_j)}
     & = E_{j,i}-E_{n+1+i,n+1+j},                         & \text{for } & 1\leq i<j\leq n, \\
    X_{-(\chi_i+\chi_j)}
     & = E_{n+1+j,i}-E_{n+1+i,j},                         & \text{for } & 1\leq i<j\leq n, \\
    X_{-\chi_i}
     & = E_{n+1,i}-2E_{n+1+i,n+1},                        & \text{for } & 1\leq i\leq n.
\end{align*}

As usual the Lie algebra of $\SO(p,q)$ has the real structure
\begin{align*}
    \so(p,q)_\R=\Set{X \in \Mat_{2n+1}(\R) | \tp{X}\begin{pmatrix}1_p&\\&-1_q\end{pmatrix} +\begin{pmatrix}1_p&\\&-1_q\end{pmatrix}X = 0 },
\end{align*}
and the inner twist $\xi=\xi_{p,q} : G^*=\SO_{2n+1} \to \SO(p,q)$ gives an isomorphism $\xi : \frakg^* \ni X\mapsto S_{p,q}XS_{p,q}\inv \in \so(p,q)$ of the complex Lie algebras.

The inner twist $\xi_{p,q}$ sends an element
\begin{equation*}
    \operatorname{diag}(\ldots,1,\overset{i}{t},1,\ldots, 1, \overset{n+1+i}{t\inv},1,\ldots) \in T^*,
\end{equation*}
to
\begin{align*}
    \begin{dcases*}
        \left(\begin{array}{ccccc}
                  1_{i-1} &        &     &         &         \\
                          & \cos z &     & -\sin z &         \\
                          &        & 1_n &         &         \\
                          & \sin z &     & \cos z  &         \\
                          &        &     &         & 1_{n-i}
              \end{array}\right),                      & for $i\leq r$,       \\
        \left(\begin{array}{ccccc}
                  1_{i-1} &                   &     &                   &         \\
                          & \frac{t+t\inv}{2} &     & \frac{t-t\inv}{2} &         \\
                          &                   & 1_n &                   &         \\
                          & \frac{t-t\inv}{2} &     & \frac{t+t\inv}{2} &         \\
                          &                   &     &                   & 1_{n-i}
              \end{array}\right), & for $r<i$,
    \end{dcases*}
\end{align*}
where $z$ is a complex number such that $t=e^{\sqrt{-1}z}$. ($i\leq r$ is non-split part, $i>r$ is split part.)

The images of the root vectors and the coroots are given as follows:
\begin{align*}
     & \xi(X_{\chi_i-\chi_j})=\frac{1}{2}\times          \\
     & \quad\begin{dcases*}
                \begin{multlined}
            \left( E_{i,j}+\sqrt{-1}E_{i,n+1+j}-E_{j,i}+\sqrt{-1}E_{j,n+1+i}\right.\\
            \left.-\sqrt{-1}E_{n+1+i,j}+E_{n+1+i,n+1+j}-\sqrt{-1}E_{n+1+j,i}-E_{n+1+j,n+1+i} \right),
        \end{multlined}          & for $i<j\leq r$,          \\
                \begin{multlined}
            \left( E_{i,j}+E_{i,n+1+j}-E_{j,i}+\sqrt{-1}E_{j,n+1+i}\right.\\
            \left.-\sqrt{-1}E_{n+1+i,j}-\sqrt{-1}E_{n+1+i,n+1+j}+E_{n+1+j,i}-\sqrt{-1}E_{n+1+j,n+1+i} \right),
        \end{multlined} & for $i\leq r<j$, \\
                \begin{multlined}
            \left( E_{i,j}+E_{i,n+1+j}-E_{j,i}+E_{j,n+1+i}\right.\\
            \left.+E_{n+1+i,j}+E_{n+1+i,n+1+j}+E_{n+1+j,i}-E_{n+1+j,n+1+i} \right),
        \end{multlined}                            & for $r< i<j$,
            \end{dcases*}
    \\
    \\
     & \xi(H_{\chi_i-\chi_j})=
    \begin{dcases*}
        \sqrt{-1}\left(E_{i,n+1+i}-E_{j,n+1+j}-E_{n+1+i,i}+E_{j,n+1+j}\right), & for $i<j\leq r$, \\
        \sqrt{-1}E_{i,n+1+i}-E_{j,n+1+j}-\sqrt{-1}E_{n+1+i,i}-E_{j,n+1+j},     & for $i\leq r<j$, \\
        E_{i,n+1+i}-E_{j,n+1+j}+E_{n+1+i,i}-E_{j,n+1+j},                       & for $r< i<j$,
    \end{dcases*}
    \\
    \\
     & \xi(X_{-(\chi_i-\chi_j)})=\frac{1}{2}\times       \\
     & \quad\begin{dcases*}
                \begin{multlined}
            \left( -E_{i,j}+\sqrt{-1}E_{i,n+1+j}+E_{j,i}+\sqrt{-1}E_{j,n+1+i}\right.\\
            \left.-\sqrt{-1}E_{n+1+i,j}-E_{n+1+i,n+1+j}-\sqrt{-1}E_{n+1+j,i}+E_{n+1+j,n+1+i} \right),
        \end{multlined}
                                                                                        & for $i<j\leq r$, \\
                \begin{multlined}
            \left( -E_{i,j}+E_{i,n+1+j}+E_{j,i}+\sqrt{-1}E_{j,n+1+i}\right.\\
            \left.-\sqrt{-1}E_{n+1+i,j}+\sqrt{-1}E_{n+1+i,n+1+j}+E_{n+1+j,i}+\sqrt{-1}E_{n+1+j,n+1+i} \right),
        \end{multlined}
                                                                                        & for $i\leq r<j$, \\
                \begin{multlined}
            \left( -E_{i,j}+E_{i,n+1+j}+E_{j,i}+E_{j,n+1+i}\right.\\
            \left.+E_{n+1+i,j}-E_{n+1+i,n+1+j}+E_{n+1+j,i}+E_{n+1+j,n+1+i} \right),
        \end{multlined} & for $r< i<j$,
            \end{dcases*}
    \\
    \\
    \\
     & \xi(X_{\chi_i+\chi_j})=\frac{1}{2}\times          \\
     & \quad\begin{dcases*}
                \begin{multlined}
            \left( E_{i,j}-\sqrt{-1}E_{i,n+1+j}-E_{j,i}+\sqrt{-1}E_{j,n+1+i}\right.\\
            \left.-\sqrt{-1}E_{n+1+i,j}-E_{n+1+i,n+1+j}+\sqrt{-1}E_{n+1+j,i}+E_{n+1+j,n+1+i} \right),
        \end{multlined}          & for $i<j\leq r$,          \\
                \begin{multlined}
            \left( E_{i,j}-E_{i,n+1+j}-E_{j,i}+\sqrt{-1}E_{j,n+1+i}\right.\\
            \left.-\sqrt{-1}E_{n+1+i,j}+\sqrt{-1}E_{n+1+i,n+1+j}-E_{n+1+j,i}+\sqrt{-1}E_{n+1+j,n+1+i} \right),
        \end{multlined} & for $i\leq r<j$, \\
                \begin{multlined}
            \left( E_{i,j}-E_{i,n+1+j}-E_{j,i}+E_{j,n+1+i}\right.\\
            \left.+E_{n+1+i,j}-E_{n+1+i,n+1+j}-E_{n+1+j,i}+E_{n+1+j,n+1+i} \right),
        \end{multlined}                            & for $r< i<j$,
            \end{dcases*}
    \\
    \\
     & \xi(H_{\chi_i+\chi_j})=
    \begin{dcases*}
        \sqrt{-1}\left( E_{i,n+1+i}+E_{j,n+1+j}-E_{n+1+i,i}-E_{j,n+1+j} \right), & for $i<j\leq r$, \\
        \sqrt{-1}E_{i,n+1+i}+E_{j,n+1+j}-\sqrt{-1}E_{n+1+i,i}+E_{j,n+1+j},       & for $i\leq r<j$, \\
        E_{i,n+1+i}+E_{j,n+1+j}+E_{n+1+i,i}+E_{j,n+1+j},                         & for $r<i<j$,
    \end{dcases*}
    \\
    \\
     & \xi(X_{-(\chi_i+\chi_j)})=\frac{1}{2}\times       \\
     & \quad\begin{dcases*}
                \begin{multlined}
            \left( -E_{i,j}-\sqrt{-1}E_{i,n+1+j}+E_{j,i}+\sqrt{-1}E_{j,n+1+i}\right.\\
            \left.-\sqrt{-1}E_{n+1+i,j}+E_{n+1+i,n+1+j}+\sqrt{-1}E_{n+1+j,i}-E_{n+1+j,n+1+i} \right),
        \end{multlined}
                                                                                        & for $i<j\leq r$, \\
                \begin{multlined}
            \left( -E_{i,j}-E_{i,n+1+j}+E_{j,i}+\sqrt{-1}E_{j,n+1+i}\right.\\
            \left.-\sqrt{-1}E_{n+1+i,j}-\sqrt{-1}E_{n+1+i,n+1+j}-E_{n+1+j,i}-\sqrt{-1}E_{n+1+j,n+1+i} \right),
        \end{multlined}
                                                                                        & for $i\leq r<j$, \\
                \begin{multlined}
            \left( -E_{i,j}-E_{i,n+1+j}+E_{j,i}+E_{j,n+1+i}\right.\\
            \left.+E_{n+1+i,j}+E_{n+1+i,n+1+j}-E_{n+1+j,i}-E_{n+1+j,n+1+i} \right),
        \end{multlined} & for $r< i<j$,
            \end{dcases*}
    \\
    \\
    \\
     & \xi(X_{\chi_i})=
    \begin{dcases*}
        E_{i,n+1}-E_{n+1,i}+\sqrt{-1}E_{n+1,n+1+i}-\sqrt{-1}E_{n+1+i,n+1}, & for $i\leq r$, \\
        E_{i,n+1}-E_{n+1,i}+E_{n+1,n+1+i}+E_{n+1+i,n+1},                   & for $r< i$,
    \end{dcases*}
    \\
    \\
     & \xi(H_{\chi_i})=2\times
    \begin{dcases*}
        \sqrt{-1}\left( E_{i,n+1+i}-E_{n+1+i,i} \right), & for $i\leq r$, \\
        \left( E_{i,n+1+i}+E_{n+1+i,i} \right),          & for $r< i$,
    \end{dcases*}
    \\
    \\
     & \xi(X_{-\chi_i})=
    \begin{dcases*}
        -E_{i,n+1}+E_{n+1,i}+\sqrt{-1}E_{n+1,n+1+i}-\sqrt{-1}E_{n+1+i,n+1}, & for $i\leq r$, \\
        -E_{i,n+1}+E_{n+1,i}+E_{n+1,n+1+i}+E_{n+1+i,n+1},                   & for $r< i$.
    \end{dcases*}
\end{align*}
In particular we have $\tp{\overline{\xi(X_\alpha)}}=\xi(X_{-\alpha})$ for any $\alpha\in R(T^*,G^*)$.

Put $T=\xi(T^*)$.
The real form $T(\R)$ consists of the elements of the form $\xi(\diag(e^{\sqrt{-1}\theta_1},\ldots,e^{\sqrt{-1}\theta_r},t_{r+1},\ldots,t_n,1,\ldots))$, where $t_i\in\R^\times$ and $\theta_i\in\R/2\pi\Z$.
For an element $t\in T(\R)$ of such form and a positive root $\alpha\in R(T,\SO(p,q))$, which is equal to $R(T^*,G^*)$ as sets, we have
\begin{align*}
    \Ad(t)(\xi(X_\alpha))=\alpha(t)\xi(X_\alpha),
\end{align*}
where
\begin{align*}
    \alpha(t)=
    \begin{dcases*}
        e^{\sqrt{-1}(\theta_i\pm\theta_j)}, & for $\alpha=\chi_i\pm\chi_j$, $i<j\leq r$,  \\
        e^{\sqrt{-1}\theta_i} t_j^{\pm1},   & for $\alpha=\chi_i\pm\chi_j$, $i\leq r <j$, \\
        t_it_j^{\pm1},                      & for $\alpha=\chi_i\pm\chi_j$, $r<i<j$,      \\
        e^{\sqrt{-1}\theta_i},              & for $\alpha=\chi_i$, $i\leq r$,             \\
        t_i,                                & for $\alpha=\chi_i$, $r<i$.
    \end{dcases*}
\end{align*}
This means that as roots of $T$ in $G$, $\alpha=\pm(\chi_i\pm\chi_j)$ ($i<j\leq r$), $\pm\chi_i$ ($i\leq r$) are imaginary roots (i.e., $\overline{\alpha}=-\alpha$), $\alpha=\pm(\chi_i\pm\chi_j)$ ($i\leq r<j$) are complex roots (i.e., $\overline{\alpha}\neq\pm\alpha$), and $\alpha=\pm(\chi_i\pm\chi_j)$ ($r<i<j$), $\pm\chi_i$ ($r<i$) are real roots (i.e., $\overline{\alpha}=\alpha$).

\section{Notions around endoscopy and parameters, and Main theorem}\label{1}
In this section, we first establish some notations for the endoscopy groups, the transfer factors, and $L$- or $A$-parameters.
Afterwards, we shall recall Arthur's work and state our main theorem.

\subsection{Endoscopy}\label{1.1}
In this subsection, we recall notion of an endoscopic triple and the transfer factor from \cite[\S1.1]{kmsw}.

\subsubsection{Endoscopic triples}
Let $F$ be a global or local field.
Consider a pair $(G^*,\theta^*)$ consisting of a connected quasi-split reductive group $G^*$ defined over $F$ with a fixed $F$-pinning, and a pinned automorphism $\theta^*$ of $G^*$.
Here, an automorphism $\theta^*$ is said to be pinned if it preserves the fixed pinning (i.e., the maximal torus, the Borel subgroup, and the set of simple root vectors are $\theta^*$-stable).
Then we have an automorphism $\widehat{\theta^*}$ of $\widehat{G^*}$, which preserves a $\Gamma$-invariant pinning for $\widehat{G^*}$.
Put $\Lgp{\theta^*}=\widehat{\theta^*}\rtimes \id_{W_F}$, which is an $L$-automorphism of $\Lgp{G^*}$.
\begin{definition}
    An endoscopic triple for $(G^*, \theta^*)$ is a triple $\frake=(G^\frake,s^\frake,\eta^\frake)$ of a connected quasi-split reductive group $G^\frake$ defined over $F$, a semisimple element $s^\frake\in \widehat{G^*}$, and an $L$-homomorphism $\eta^\frake : \Lgp{G^\frake}\to\Lgp{G^*}$ such that
    \begin{itemize}
        \item $\Ad(s^\frake)\circ \Lgp{\theta^*}\circ \eta^\frake=\eta^\frake$;
        \item $\eta^\frake(\widehat{G^\frake})$ is the connected component of the subgroup of $\Ad(s^\frake)\circ \widehat{\theta^*}$-fixed elements in $\widehat{G^*}$.
    \end{itemize}
    When $\eta(Z(\widehat{G^\frake})^\Gamma)^\circ \subset Z(\widehat{G^*})$, the endoscopic triple $\frake$ is said to be elliptic.
    When $\theta^*$ is trivial (resp. not trivial), every endoscopic triple for $(G^*,\theta^*)$ is said to be ordinary (resp. twisted).
\end{definition}
\begin{definition}
    Two endoscopic triples $\frake_1=(G^\frake_1, s^\frake_1, \eta^\frake_1)$ and $\frake_2=(G^\frake_2, s^\frake_2, \eta^\frake_2)$ for $G^*$ are said to be isomorphic (resp. strictly isomorphic) if there exists an element $g \in \widehat{G^*}$ such that
    \begin{itemize}
        \item $g \eta^\frake_1(\Lgp{G^\frake_1}) g\inv = \eta^\frake_2(\Lgp{G^\frake_2})$;
        \item $zg s^\frake_1 \widehat{\theta^*}(g)\inv = s^\frake_2$ for some $z \in Z(\widehat{G^*})$, (resp. $g s^\frake_1 \widehat{\theta^*}(g)\inv = s^\frake_2$).
    \end{itemize}
    Then the element $g$ is called an isomorphism (resp. strict isomorphism) from $\frake_1$ to $\frake_2$.
    If $g$ is an isomorphism (resp. strict isomorphism) from $\frake$ to $\frake$ itself, then it is said to be an automorphism (resp. strict automorphism) of $\frake$, and we shall write $\overline{\Aut}_{G^*}(\frake)$ (resp. $\Aut_{G^*}(\frake)$) for the set of automorphisms (resp. strict automorphisms) of $\frake$.
\end{definition}

Let $\xi : G^* \to G$ be an inner twist, and put $\theta=\xi\circ\theta^*\circ\xi\inv$.
Suppose that an automorphism $\theta$ of $G$ is defined over $F$.
\begin{definition}
    By an endoscopic triple for $(G,\theta)$, we mean that for $(G^*,\theta^*)$.
    The notion of elliptic endoscopic triple, isomorphism and strict isomorphism of endoscopic triples are also the same as those for $(G^*,\theta^*)$.
\end{definition}
We shall write $\overline{\calE}(G\rtimes\theta)$ (resp. $\calE(G\rtimes\theta)$) for the set of isomorphism (resp. strict isomorphism) classes of endoscopic triples for $(G,\theta)$.
The corresponding subsets of elliptic endoscopic triples will be indicated by the lower right index $\mathrm{ell}$:
\begin{equation*}
    \xymatrix{
    \calE(G\rtimes\theta) \ar@{->>}[r] & \overline{\calE}(G\rtimes\theta)\\
    \calE_\el(G\rtimes\theta) \ar@{->>}[r] \ar@{}[u]|{\bigcup} & \overline{\calE}_\el(G\rtimes\theta) \ar@{}[u]|{\bigcup}.
    }
\end{equation*}
When $\theta$ is trivial, then we may write $\calE(G)$, $\overline{\calE}(G)$, $\calE_\el(G)$, and $\overline{\calE}_\el(G)$.

Put $\widetilde{\calE}(N)=\calE(\GL_N,\theta_N)$ and $\widetilde{\calE}_\el(N)=\calE_\el(\GL_N,\theta_N)$.
If $\frake=(G^\frake, s^\frake, \eta^\frake)\in \widetilde{\calE}(N)$, the group $G^\frake$ is a direct product of finite numbers of symplectic or special orthogonal groups.
An elliptic endoscopic triple $\frake\in\widetilde{\calE}_\el(N)$ is said to be simple if the number of factors is one.
Define $\widetilde{\calE}_\simple(N)$ to be the set of strictly isomorphism classes of simple endoscopic triples of $(\GL_N,\theta_N)$.

\subsubsection{Normalized transfer factors}
Next, we shall recall some properties of transfer factors.
Let $G^*$ be a quasi-split connected reductive group over a local or global field $F$ with a fixed $F$-pinning, and $\theta^*$ an automorphism of $G^*$ preserving the pinning.
We assume that $Z(G^*)$ is connected.
Note that if $G^*$ is a direct product of a finite number of $\SO_{2m+1}$ or $\GL_m$ ($m\geq1$), then $G^*$ is quasi-split connected reductive and its center is connected.
Let $(\xi,z) : G^* \to G$ be a pure inner twist, and put $\theta=\xi\circ\theta^*\circ\xi\inv$.
Then $\theta$ is an automorphism of $G$ defined over $F$.
Not all twisted groups $(G,\theta)$ arise in this way, which is why an element $g_\theta \in G^*$ such that $\theta^*=\Ad(g_\theta)\circ\xi\inv\circ\theta\circ\xi$ was introduced in \cite[\S1.2]{ks99}.
In this paper, however, it suffices to assume $g_\theta$ of \cite{ks99} is $1$.
We assume that $\theta^*(z)=z$, which is enough for our purpose.
Let $\frake=(G^\frake, s^\frake,\eta^\frake)$ be an endoscopic triple for $(G,\theta)$.

Let first $F$ be a local field, and $\psi_F : F\to \C^1$ a nontrivial additive character.
Let further $\delta \in G_{\theta-\srss}(F)$ and $\gamma\in G^\frake(F)$, where $G_{\theta-\srss}(F)$ denotes the set of $\theta$-strongly regular and $\theta$-semisimple elements in $G(F)$.
Then we have the transfer factor:
\begin{align*}
    \Delta[\frake, \xi,z](\gamma,\delta)=\begin{dcases*}
                                             \epsilon(\frac{1}{2},V, \psi_F) \Delta_I^{\mathrm{new}}(\gamma,\delta)\inv \Delta_{II}(\gamma,\delta) \Delta_{III}(\gamma,\delta)\inv \Delta_{IV}(\gamma,\delta), & if $\gamma$ is a norm of $\delta$, \\
                                             0,                                                                                                                                                                & otherwise,
                                         \end{dcases*}
\end{align*}
where $V$, $\Delta_I^{\mathrm{new}}$, $\Delta_{II}$, $\Delta_{III}$, and $\Delta_{IV}$ are defined as in \cite[pp.38-39]{kmsw}.
Although they treated an extended pure inner twist in loc. cit., the similar arguments also works for a pure inner twist.
The main difference is that we utilize an ordinary cocycle, instead of a basic cocycle.
See also \cite[\S4.3-4.5, Appendices A-C]{ks99} and \cite[\S3.4]{ks12}.
\begin{remark}
    In the case $G^*=\SO_{2n+1}$ and $\theta^*=1$, the transfer factor $\Delta[\frake, \xi,z]$ is determined by the isomorphism class of $G$, not by the choice of $(\xi,z)$.
\end{remark}
\begin{proposition}\label{1.1.1}
    Let $\gamma_1, \gamma_2 \in G^\frake(F)$ be norms of $\delta_1, \delta_2\in G_{\theta-\srss}(F)$, respectively.
    Then we have
    \begin{equation*}
        \frac{\Delta[\frake, \xi,z](\gamma_1,\delta_1)}{\Delta[\frake, \xi,z](\gamma_2,\delta_2)}=\Delta'(\gamma_1,\delta_1;\gamma_2,\delta_2),
    \end{equation*}
    where the right hand side is the relative transfer factor of \cite{ks12}.
\end{proposition}
\begin{proof}
    The proof is similar to that of \cite[Proposition 1.1.1]{kmsw}.
\end{proof}
\begin{lemma}\label{1.1.2}
    Let $\overline{s}^\frake$ denote the image of $s^\frake$ in $(\widehat{G^*}/\widehat{G^*}_{\der})_{\widehat{\theta^*}, \mathrm{free}}^\Gamma$, which is the set of $\Gamma$-fixed points in the torsion free quotient of the $\widehat{\theta^*}$-covariants of $\widehat{G^*}/\widehat{G^*}_{\der}$.
    For $x\in Z(\widehat{G^*})^\Gamma$ and $y\in Z^1(F, (Z(G^*)^{\theta^*})^\circ)$, we have
    \begin{equation*}
        \Delta[x\frake, \xi,z]=\an{x,z}\Delta[\frake,\xi,z],
    \end{equation*}
    and
    \begin{equation*}
        \Delta[\frake, \xi,yz]=\an{\overline{s^\frake},y}\Delta[\frake,\xi,z],
    \end{equation*}
    where $x\frake$ denotes the endoscopic triple $(G^\frake, xs^\frake, \eta^\frake)$.
\end{lemma}
\begin{proof}
    The proof is similar to that of \cite[Lemma 1.1.2]{kmsw}.
\end{proof}

We shall now recall the notion of matching functions from \cite[\S5.5]{ks99}.
For $\delta\in G_{\theta-\srss}(F)$ and $f\in \calH(G)$, the orbital integral is defined as
\begin{equation*}
    O_{\delta\theta}(f) = \int_{G_{\delta\theta}(F)\backslash G(F)} f(g\inv \delta \theta(g)) dg,
\end{equation*}
where $G_{\delta\theta}=\cent_\theta(\delta,G)=\{x\in G \mid \delta \theta(x)=x\delta \}$.
If $\theta$ is trivial, it may be simply written $O_\delta(f)$.
For $\gamma\in G^\frake_\srss(F)$ and $f^\frake\in \calH(G^\frake)$, the stable orbital integral is defined as
\begin{equation*}
    \mathit{SO}_\gamma(f) = \sum_{\gamma'} O_{\gamma'}(f^\frake),
\end{equation*}
where the sum is taken over a set of representatives for the conjugacy classes in the stable conjugacy class of $\gamma$.
\begin{definition}\label{matching}
    Two functions $f^\frake\in \calH(G^\frake)$ and $f\in \calH(G)$ are called $\Delta[\frake,\xi,z]$-matching (or  matching when there is no danger of confusion) if for any strongly $G$-regular semisimple element $\gamma\in G^\frake(F)$ we have
    \begin{equation*}
        \mathit{SO}_\gamma(f^\frake) = \sum_{\delta} \Delta[\frake,\xi,z](\gamma,\delta) O_{\delta\theta}(f),
    \end{equation*}
    where the sum is taken over a set of representatives for the $\theta$-conjugacy classes under $G(F)$ of $G_{\theta-\srss}(F)$.
    We also say that $f$ and $f^\frake$ have $\Delta[\frake,\xi,z]$-matching orbital integrals.
\end{definition}
\begin{remark}
    If we do not assume $\theta^*(z)=z$, then we need $\theta^\frake$ on $G^\frake$ as in \cite[\S5.4]{ks99}.
\end{remark}

Next, let $F$ be a global field, and $\psi_F : \A_F/F \to \C^1$ a nontrivial additive character.
For $\delta \in G_{\theta-\srss}(\A_F)$ and $\gamma\in G^\frake(\A_F)$, the global absolute transfer factor is defined by
\begin{equation*}
    \Delta_\A[\frake,\xi](\gamma,\delta)=\begin{dcases*}
        \prod_v \Delta[\frake_v,\xi_v,z_v](\gamma_v,\delta_v), & if $\gamma$ is a norm of $\delta$, \\
        0,                                                     & otherwise.
    \end{dcases*}
\end{equation*}
Note that the product is well-defined because the almost all factors are $1$.
Moreover, Lemma \ref{1.1.2} and the exact sequence \eqref{033} imply that $\Delta_\A[\frake,\xi]$ is independent of $z$.
\begin{proposition}\label{1.1.3}
    The factor $\Delta_\A[\frake,\xi]$ coincides with the inverse of the canonical adelic transfer given in \cite{ks99}.
\end{proposition}
\begin{proof}
    The proof is similar to that of \cite[Proposition 1.1.3]{kmsw}.
\end{proof}

We shall finish this subsubsection stating a lemma regarding local transfer factors and Levi subgroups.
Let again $F$ be a local field, and $\psi_F : F\to\C^1$ a nontrivial additive character.
Let $M^*\subset G^*$ be a standard Levi subgroup such that $M=\xi(M^*)\subset G$ is a Levi subgroup defined over $F$.
Assume that $\theta$ is trivial.
Let $\frake=(G^\frake s^\frake, \eta^\frake)$ and $\frake_M=(M^\frake,s^\frake,\eta^\frake|_{\Lgp{M^\frake}})$ be endoscopic triples for $G$ and $M$ respectively, such that $M^\frake\subset G^\frake$ is a Levi subgroup.
\begin{lemma}\label{1.1.4}
    \begin{enumerate}
        \item For each $\delta \in M(F)$ and $\gamma\in M^\frake(F)$, we have
              \begin{equation*}
                  \Delta[\frake_M,\xi|_{M^*},z](\gamma,\delta)=\Delta[\frake,\xi,z](\gamma,\delta)\left( \frac{|D^{G^\frake}_{M^\frake}(\gamma)|}{|D^G_M(\delta)|}\right)^\frac{1}{2},
              \end{equation*}
              where $D^G_M(\delta)$ and $D^{G^\frake}_{M^\frake}(\gamma)$ are the relative Weyl discriminants.
        \item If $f\in \calH(G)$ and $f^\frake\in \calH(G^\frake)$ are $\Delta[\frake,\xi,z]$-matching, then their constant terms $f_M\in \calH(M)$ and $f^\frake_{M^\frake}\in \calH(M^\frake)$ are $\Delta[\frake_M,\xi|_{M^*},z]$-matching.
    \end{enumerate}
\end{lemma}
\begin{proof}
    The proof is similar to that of \cite[Lemma 1.1.4]{kmsw}.
\end{proof}

\subsubsection{Endoscopic triples for odd special orthogonal groups}
Here we will explicate the set of representatives of isomorphism or strict isomorphism classes of ordinary endoscopic triples for odd special orthogonal groups.
Let $F$ be a local or global field, $G^*=\SO_{2n+1}$ the split special orthogonal group of size $2n+1$ defined over $F$, and $G$ an inner form of $G^*$.
Then the Galois group $\Gamma=\Gamma_F$ acts on $\widehat{G}=\Sp(2n,\C)$ trivially, and thus we have $\Lgp{G}=\Sp(2n,\C)\times W_F$ and $Z(\widehat{G})^\Gamma=\{\pm1\}$.

Every semisimple element $s\in\Sp(2n,\C)$ is $\Sp(2n,\C)$-conjugate to an element of the form
\begin{align}\label{se}
    \left(\begin{array}{cccccccccc}
                  a_1 1_{m_1} &        &             &         &          &                 &        &                 &         &          \\
                              & \ddots &             &         &          &                 &        &                 &         &          \\
                              &        & a_r 1_{m_r} &         &          &                 &        &                 &         &          \\
                              &        &             & 1_{n_1} &          &                 &        &                 &         &          \\
                              &        &             &         & -1_{n_2} &                 &        &                 &         &          \\
                              &        &             &         &          & a_1\inv 1_{m_1} &        &                 &         &          \\
                              &        &             &         &          &                 & \ddots &                 &         &          \\
                              &        &             &         &          &                 &        & a_r\inv 1_{m_r} &         &          \\
                              &        &             &         &          &                 &        &                 & 1_{n_1} &          \\
                              &        &             &         &          &                 &        &                 &         & -1_{n_2} \\
              \end{array}\right),
\end{align}
where $r\geq0$, $m_1,\ldots, m_r\geq1$, and $n_1, n_2\geq0$ are integers such that $m_1+\cdots m_r+n_1+n_2=n$, and $a_1,\ldots,a_r \in \C\setminus\{0,1,-1\}$ are complex numbers other than $0,1,-1$ that are different to each other.
Thus its centralizer $\cent(s,\Sp(2n,\C))$ is isomorphic to
\begin{align}\label{dge}
    \GL(m_1,\C)\times\cdots\times\GL(m_r,\C)\times\Sp(2n_1,\C)\times\Sp(2n_2,\C),
\end{align}
which is the dual group of
\begin{align}\label{ge}
    \GL_{m_1}\times\cdots\times\GL_{m_r}\times\SO_{2n_1+1}\times\SO_{2n_2+1}.
\end{align}
Therefore, each endoscopic triple is isomorphic to a triple of a group \eqref{ge}, a semisimple element \eqref{se}, and a natural inclusion map.
Any description of complete systems of representatives of isomorphism or strict isomorphism classes of them is so complicated that we do not write down it here.
Elliptic ones are simpler.

For two nonnegative integers $n_1$ and $n_2$ such that $n_1+n_2=n$, put $\frake_{n_1,n_2}=(G^\frake_{n_1,n_2}, s_{n_1,n_2}, \eta_{n_1,n_2})$, where
\begin{itemize}
    \item $G^\frake_{n_1,n_2}=\SO_{2n_1+1}\times\SO_{2n_2+1}$,
    \item $
              s_{n_1,n_2}=\left(\begin{array}{cccc}
                      1_{n_1} &          &         &          \\
                              & -1_{n_2} &         &          \\
                              &          & 1_{n_1} &          \\
                              &          &         & -1_{n_2}
                  \end{array}\right),
          $
    \item $\eta_{n_1,n_2}$ denotes the direct product of the inclusion map $\Sp(2n_1,\C)\times\Sp(n_2,\C)\into\Sp(2n,\C)$ given by
          \begin{align*}
              \left(
              \left(\begin{array}{cc}
                        A_1 & B_1 \\
                        C_1 & D_1 \\
                    \end{array}\right),
              \left(\begin{array}{cc}
                        A_2 & B_2 \\
                        C_2 & D_2 \\
                    \end{array}\right)
              \right)
              \mapsto
              \left(\begin{array}{cccc}
                        A_1 &     & B_1 &     \\
                            & A_2 &     & B_2 \\
                        C_1 &     & D_1 &     \\
                            & C_2 &     & D_2
                    \end{array}\right),
          \end{align*}
          and the identity map of $W_F$.
\end{itemize}
Then $\frake_{n_1,n_2}$ is an elliptic endoscopic triple for $G$, and four triples $\pm\frake_{n_1,n_2}, \pm\frake_{n_2,n_1}$ are isomorphic to each other.
Here $-(G^\frake, s^\frake, \eta^\frake)$ denotes $(G^\frake,-s^\frake,\eta^\frake)$.
Therefore, a set
\begin{align*}
    \Set{\frake_{n_1,n_2} | n_1\geq n_2\geq0, \ n_1+n_2=n}
\end{align*}
is a complete system of representatives of isomorphism classes of elliptic endoscopic triples for $G$.
The set $\overline{\calE}_\el(G)$ may be identified with it.
In addition, $\frake_{n_1,n_2}$ is strictly isomorphic to $-\frake_{n_2,n_1}$, but not strictly isomorphic to $\frake_{n_2,n_1}$ unless $n_1=n_2$.
Therefore, a set
\begin{align*}
    \Set{\frake_{n_1,n_2} | n_1, n_2\geq0, \ n_1+n_2=n}
\end{align*}
is a complete system of representatives of strictly isomorphism classes of elliptic endoscopic triples for $G$.
The set $\calE_\el(G)$ may be identified with it.

\subsection{Local parameters}\label{1.2}
In this subsection, we recall the notion of local parameters.
Let $F$ be a local field.
Let $G$ be a quasi-split connected reductive group over $F$.
We say that a continuous homomorphism $\phi : L_F \to \Lgp{G}$ is an $L$-parameter for $G$ if it commutes with the canonical projections $L_F\onto W_F$ and $\Lgp{G}\onto W_F$, and the image $\phi(L_F)$ consists of semisimple elements. Here, an element of $\Lgp{G}$ is said to be semisimple if the first component, which is an element of $\widehat{G}$, is semisimple.
Two $L$-parameters are said to be equivalent if they are conjugate by an element in $\widehat{G}$.
We write $\Phi(G)$ for the set of equivalence classes of $L$-parameters for $G$.
An $L$-parameter $\phi$ is called bounded (or tempered) if the image of the projection of $\phi(L_F)$ to $\widehat{G}$ is bounded.
We write $\Phi_\bdd(G)$ for the subset of $\Phi(G)$ consisting of bounded ones.

We say that a continuous homomorphism $\psi : L_F \times\SU(2,\R) \to \Lgp{G}$ is a local $A$-parameter for $G$ if the restriction $\psi|_{L_F}$ to $L_F$ is a bounded $L$-parameter.
Two local $A$-parameters are said to be equivalent if they are conjugate by an element in $\widehat{G}$.
We write $\Psi(G)$ for the set of equivalence classes of local $A$-parameters for $G$.
Let $\Psi^+(G)$ denote the set of equivalence classes of a continuous homomorphism $\psi : L_F \times\SU(2,\R) \to \Lgp{G}$ whose restriction $\psi|_{L_F}$ to $L_F$ is an $L$-parameter.
A parameter $\psi\in\Psi^+(G)$ is called generic if the restriction $\psi|_{\SU(2,\R)}$ (to the second $\SU(2,\R)$) is trivial.
The subset of $\Psi^+(G)$ (resp. $\Psi(G)$) consisting of generic elements can be identified with $\Phi(G)$ (resp. $\Phi_\bdd(G)$).
We say that $\psi\in\Psi^+(G)$ is discrete if there is no proper parabolic subgroup of $\Lgp{G}$ which contains the image of $\psi$, and write $\Psi^+_2(G)$ for the subset of $\Psi^+(G)$ consisting of discrete elements.
Put $\Psi_2(G)=\Psi(G)\cap\Psi^+_2(G)$, $\Phi_2(G)=\Phi(G)\cap\Psi^+_2(G)$, and $\Phi_{2,\bdd}(G)=\Phi_\bdd(G)\cap\Psi^+_2(G)$.

We shall write $\Pi(G)$ for the set of isomorphism classes of irreducible smooth representations of $G(F)$.
Let $\Pi_\temp(G)$ (resp. $\Pi_2(G)$) to be the subsets of $\Pi(G)$ consisting of the tempered (resp. essentially square integrable) representations.
Put $\Pi_{2,\temp}(G)=\Pi_2(G)\cap\Pi_\temp(G)$.

For $\psi\in\Psi^+(G)$, put $S_\psi=\cent(\myim\psi,\widehat{G})$, $\overline{S}_\psi=S_\psi/Z(\widehat{G})^\Gamma$, $\frakS_\psi=\pi_0(S_\psi)=S_\psi/S_\psi^\circ$, $\overline{\frakS}_\psi=\pi_0(\overline{S}_\psi)=S_\psi/S_\psi^\circ Z(\widehat{G})^\Gamma=\frakS_\psi/Z(\widehat{G})^\Gamma$, $S_\psi^\mathrm{rad}=\cent(\myim\psi, \widehat{G}_\mathrm{der})^\circ$, and $S_\psi^\natural=S_\psi/S_\psi^\mathrm{rad}$.
When we emphasize that the parameter is for $G$, we write $S_\psi(G)$, $\overline{S}_\psi(G)$, $\frakS_\psi(G)$, etc.
The associated $L$-parameter $\phi_\psi$ of $\psi$ is defined by
\begin{align*}
    \phi_\psi(w)=\psi(w,\mmatrix{|w|^\frac{1}{2}}{0}{0}{|w|^{-\frac{1}{2}}}),
\end{align*}
for $w\in L_F$.
Finally, we put
\begin{align*}
    s_\psi=\psi(1,\mmatrix{-1}{0}{0}{-1}).
\end{align*}

Recall that when $G$ is a general linear group $\GL_N$, then the parameters can be regarded as an $N$-dimensional representations of $L_F$ or $L_F\times\SU(2,\R)$, and the equivalence of parameters is isomorphism of representations.
A parameter which is irreducible as a representation, is said to be simple.
We write $\Psi^+_\simple(\GL_N)$ for the subset of $\Psi^+(\GL_N)$ consisting of simple parameters, and put $\Psi_\simple(\GL_N)=\Psi(\GL_N)\cap\Psi^+_\simple(\GL_N)$, $\Phi_\simple(\GL_N)=\Phi(\GL_N)\cap\Psi^+_\simple(\GL_N)$, and $\Phi_{\simple,\heartsuit}(\GL_N)=\Phi_\heartsuit(\GL_N)\cap\Phi_\simple(\GL_N)$, for $\heartsuit=2,\bdd$.
The canonical bijection between $\Pi(\GL_N)$ and $\Phi(\GL_N)$, which is called the local Langlands correspondence for $\GL_N$, is known by Langlands \cite{lan89} in the archimedean case (for any $G$ in fact), by Harris-Taylor \cite{ht} and Henniart \cite{hen} in the non-archimedean case.
As in \cite{art13}, we shall write $\Pi(N)=\Pi(\GL_N)$, $\Phi(N)=\Phi(\GL_N)$, and so on.
For any $\phi\in\Phi(N)$, we write $\pi_\phi$ for the unique corresponding element in $\Pi(N)$.
Moreover, for any $\psi\in\Psi^+(N)$, we define the corresponding (not necessarily irreducible) representation $\pi_\psi$ as \cite[(1.2.4)]{kmsw}.
Put $\Psi^+_\unit(N)=\{ \psi\in\Psi^+(N) \mid \pi_\psi \text{ is irreducible and unitary.} \}$.
Note that $\Psi(N)\subset\Psi^+_\unit(N)$.
When $\psi\in\Psi(N)$, then $\pi_\psi=\pi_{\phi_\psi}$.
A parameter $\psi$ for $\GL_N$ is said to be self-dual if it is self-dual as a representation.
We shall write $\widetilde{\Phi}(N)$ (resp. $\widetilde{\Phi}_\simple(N)$, resp. $\widetilde{\Phi}_{\simple,\bdd}(N)$) for the subset of $\Phi(N)$ (resp. $\Phi_\simple(N)$, resp. $\Phi_{\simple,\bdd}(N)$) consisting of self-dual ones.

Suppose that $G$ is a simple twisted endoscopy group of $(\GL_N,\theta_N)$.
We say a parameter $\psi\in\Psi^+(G)$ is simple if $\eta\circ\psi\in\Psi^+(\GL_N)$ is simple, where $\eta:\Lgp{G}\to\Lgp{\GL_N}$ is the $L$-homomorphism attached to the endoscopy group $G$.
We write $\Psi^+_\simple(G)$ for the subset of $\Psi^+(G)$ consisting of simple parameters, and put $\Psi_\simple(G)$, $\Phi_\simple(G)$, and so on in the similar way.
Moreover, put $\widetilde{\Psi}^+(G)$ to be the subset of $\Psi^+(\GL_N)$ consisting of the parameters of the form $\eta\circ\psi$ for some $\psi\in\Psi^+(G)$.
Note that $\widetilde{\Psi}^+(G)=\Psi^+(G)$ unless $G$ is an even special orthogonal group, in which case $\widetilde{\Psi}^+(G)$ coincides with $\Psi^+(G)/\sim_\epsilon$ in the sense of \cite{ag}.
Put $\widetilde{\Psi}(G)=\widetilde{\Psi}^+(G)\cap\Psi(N)$, $\widetilde{\Phi}(G)=\widetilde{\Psi}^+(G)\cap\Phi(N)$, $\widetilde{\Phi}_\simple(G)=\widetilde{\Phi}(G)\cap\Phi_\simple(N)$, and $\widetilde{\Phi}_{\simple,\bdd}(G)=\widetilde{\Phi}_\simple(G)\cap\Phi_\bdd(N)$.
Let $\widetilde{\Phi}_2(G)$ be the subset of $\widetilde{\Phi}(G)$ consisting of parameters of multiplicity free as representations.
Put $\widetilde{\Phi}_{2,\bdd}(G)=\widetilde{\Phi}_2(G)\cap \Phi_\bdd(N)$.

Now let us consider the case when $G$ is an inner form of the split odd special orthogonal group $G^*=\SO_{2n+1}$, which is not necessarily quasi-split.
Then the parameters for $G^*$ are also called parameters for $G$, since $\widehat{G}=\widehat{G^*}$.
We omit the definition of "$G$-relevant", which is given in \cite[Definition 0.4.14]{kmsw} in a general setting.
In this paper $G$-relevant parameters are simply said to be relevant, when there is no danger of confusion.
We remark that "parameter for $G$" is different from "$G$-relevant parameter".
Although we use a phrase "parameter for $G$", we never write as $\Psi(G)$.
Moreover, we do not fix a symbol for the set of $G$-relevant parameters.
Recall from \cite[\S4]{ggp} that the parameters can be regarded as an $2n$-dimensional symplectic representations of $L_F$ or $L_F\times\SU(2,\R)$, and the equivalence of parameters is isomorphism of representations.
A parameter $\psi \in \Psi^+(G^*)$ can be written as
\begin{align*}
    \psi
     & = \bigoplus_{i \in I_\psi^+} \ell_i \psi_i
    \oplus \bigoplus_{i \in I_\psi^-} \ell_i \psi_i
    \oplus \bigoplus_{j \in J_\psi} \ell_j (\psi_j \oplus \psi_j^\vee),
\end{align*}
where $\ell_i, \ell_j$ are positive integers, $I_\psi^\pm$, $J_\psi$ indexing sets for mutually inequivalent irreducible representations $\psi_i$ of $L_F\times\SU(2,\R)$ such that
\begin{itemize}
    \item $\psi_i$ is symplectic for $i \in I_\psi^+$;
    \item $\psi_i$ is orthogonal and $\ell_i$ is even for $i \in I_\psi^-$;
    \item $\psi_j$ is not self-dual for $j \in J_\psi$.
\end{itemize}
Thus we have
\begin{align*}
    S_\psi
    \cong \prod_{i \in I_\psi^+} \Or(\ell_i, \C)
    \times \prod_{i \in I_\psi^-} \Sp(\ell_i, \C)
    \times \prod_{j \in J_\psi} \GL(\ell_j ,\C),
\end{align*}
and its component group is a free $(\Z/2\Z)$-module
\begin{align*}
    \frakS_\psi
    \cong \prod_{i \in I_\psi^+} (\Z/2\Z)a_i,
\end{align*}
with a formal basis $\{a_i\}$, where $a_i$ corresponds to the nontrivial coset $\Or(\ell_i, \C)\setminus \SO(\ell_i, \C)$.
Note that $\widehat{G}=\Sp(2n,\C)$ is perfect, and thus $S_\psi^\mathrm{rad}=S_\psi^\circ$ and $S_\psi^\natural=\frakS_\psi$.
It can be easily seen that the parameter $\psi$ is discrete if and only if $\ell_i=1$ for all $i\in I_\psi^+$ and $I_\psi^-=J_\psi=\emptyset$.
Hence, we have
\begin{align*}
    \Psi^+_2(G^*)
     & =\set{\psi\in\Psi^+(G^*) | \text{$\overline{S}_\psi$ is finite} }                                                          \\
     & =\set{\psi=\psi_1\oplus\cdots\oplus\psi_r | \text{$\psi_i$ are symplectic, irreducible, and mutually distinct, $r\geq1$}}.
\end{align*}
A parameter $\psi\in \Psi^+(G^*)$ is called to be elliptic if there exists a semisimple element in $\overline{S}_\psi$ whose centralizer in $\overline{S}_\psi(G^*)$ is finite.
We shall write $\Psi^+_\el(G^*)$ for the set of such parameters.
Then we have
\begin{align*}
    \Psi^+_\el(G^*)
     & =\set{\psi\in\Psi^+(G^*) | \text{$\cent(s,\overline{S}_\psi)$ is finite for some $s\in\overline{S}_{\psi,\semisimple}$} } \\
     & =\Set{
        \psi=2\psi_1\oplus\cdots2\psi_q\oplus\psi_{q+1}\oplus\cdots\oplus\psi_r
        | \begin{array}{cc}
              r\geq1,\ r\geq q\geq0, \\
              \text{$\psi_i$ are symplectic, irreducible, and mutually distinct}
          \end{array}
    },
\end{align*}
where $\overline{S}_{\psi,\semisimple}$ denotes the set of semisimple elements in $\overline{S}_\psi$.
We write $\Psi_\el(G^*)$ (resp. $\Psi^\el(G^*)$) for the subset of $\Psi(G^*)$ consisting of elliptic (resp. non-elliptic) ones.
We have $\Psi_2(G^*)\subset\Psi_\el(G^*)$.
Put
\begin{align*}
    \Psi_\el^2(G^*) & =\Psi_\el(G^*)\setminus \Psi_2(G^*) \\
                    & =\Set{
        \psi=2\psi_1\oplus\cdots2\psi_q\oplus\psi_{q+1}\oplus\cdots\oplus\psi_r
        | \begin{array}{cc}
              r\geq q\geq1, \\
              \text{$\psi_i$ are symplectic, irreducible, and mutually distinct}
          \end{array}
    },
\end{align*}
$\Phi_\el(G^*)=\Psi_\el(G^*)\cap\Phi(G^*)$, and $\Phi^2_\el(G^*)=\Psi^2_\el(G^*)\cap\Phi(G^*)$.

Let $M^*$ be a standard Levi subgroup of $G^*=\SO_{2n+1}$ isomorphic to
\begin{align}\label{levi}
    \GL_{n_1}\times\cdots\times\GL_{n_k}\times\SO_{2n_0+1},
\end{align}
where $n_1+\cdots n_k+n_0=n$.
Then its dual group $\widehat{M^*}$ is isomorphic to
\begin{align*}
    \GL(n_1,\C)\times\cdots\times\GL(n_k,\C)\times\Sp(2n_0,\C),
\end{align*}
and we regard it as a Levi subgroup of $\widehat{G^*}$.
We have
\begin{align*}
    \Psi^+(M^*)
    =\set{\psi_1\oplus \cdots\oplus\psi_k\oplus\psi_0 |\psi_1\in\Psi^+(n_1), \ldots, \psi_k\in\Psi^+(n_k),\ \psi_0 \in\Psi^+(\SO_{2n_0+1})}.
\end{align*}
The canonical injection $\widehat{M^*}\into\widehat{G^*}$ induces the canonical mapping
\begin{align*}
    \psi_1\oplus \cdots\oplus\psi_k\oplus\psi_0
    \mapsto
    \psi_1\oplus \cdots\oplus\psi_k\oplus\psi_0\oplus\psi_1^\vee\oplus \cdots\oplus\psi_k^\vee
\end{align*}
from $\Psi^+(M^*)$ to $\Psi^+(G^*)$.

\subsection{Global parameters}\label{1.3}
In this subsection we recall from \cite{art13} the notion of global parameters.
Let $F$ be a global field, $\Gamma=\Gamma_F$ its absolute Galois group, and $W_F$ its absolute Weil group.
For a connected reductive group $G$ over $F$, as in \cite[p.19]{art13} we put
\begin{align*}
    G(\A_F)^1=\Set{ g\in G(\A_F) |\abs*{\chi(g)}=1,\ \forall \chi\in X^*(G)_F },
\end{align*}
where $X^*(G)_F=\Hom_F(G,\GL_1)$.
The quotient set $G(F)\backslash G(\A_F)^1$ has finite volume, and there is a sequence
\begin{align*}
    L^2_\cusp(G(F)\backslash G(\A_F)^1) \subset L^2_\disc(G(F)\backslash G(\A_F)^1) \subset L^2(G(F)\backslash G(\A_F)^1),
\end{align*}
where $L^2_\cusp(G(F)\backslash G(\A_F)^1)$ is the subspace consisting of cuspidal functions, and $L^2_\disc(G(F)\backslash G(\A_F)^1)$ the subspace that can be decomposed into a direct sum of irreducible representations of $G(\A_F)^1$.
We have
\begin{align*}
    \calA_\cusp(G) \subset \calA_2(G) \subset \calA(G),
\end{align*}
where $\calA_\cusp(G)$, $\calA_2(G)$, and $\calA(G)$ denote the subsets of irreducible unitary representations of $G(\A_F)$ whose restrictions to $G(\A_F)^1$ are constituents of the respective spaces $L^2_\cusp(G(F)\backslash G(\A_F)^1)$, $L^2_\disc(G(F)\backslash G(\A_F)^1)$, and $L^2(G(F)\backslash G(\A_F)^1)$.
We also write $\calA_\cusp^+(G)$ and $\calA_2^+(G)$ for the analogues of $\calA_\cusp(G)$ and $\calA_2(G)$ defined without the unitarity.

For a finite set $S$ of places of $F$ such that $v$ is non-archimedean and $G_v$ is unramified over $F_v$ for any place $v\notin S$, put $\calC^S(G)$ to be the set consisting of families $c^S=(c_v)_{v\notin S}$ where each $c_v$ is a $\widehat{G}$-conjugacy class in $\Lgp{G_v}=\widehat{G}\rtimes W_{F_v}$ represented by an element of the form $t_v\rtimes \Frob_v$ with a semisimple element $t_v\in\widehat{G}$.
For $S\subset S'$ there is a natural map $c^S=(c_v)_{v\notin S}\mapsto (c_v)_{v\notin S'}$ from $\calC^S(G)$ to $\calC^{S'}(G)$.
Let us define $\calC(G)$ to be the direct limit of $\{\calC^S(G)\}_S$.

Let $\pi=\otimes_v \pi_v$ be an irreducible automorphic representation of $G(\A_F)$ and $S$ a finite set of places such that $\pi_v$ is unramified for all $v\notin S$.
The Satake isomorphism associates a $\widehat{G}$-conjugacy class $c(\pi_v)\in\Lgp{G_v}$ to $\pi$.
We can therefore associate $c^S(\pi)=(c(\pi_v))_{v\notin S}\in \calC^S(G)$ to $\pi$.
This leads to a mapping $\pi\mapsto c(\pi)$ from a set of equivalence classes of irreducible automorphic representations of $G$ to $\calC(G)$.
We shall write $\calC_\aut(G)$ for its image.

Let us next review some fact on automorphic representations of $\GL_N(\A_F)$.
Let $N$ be a positive integer.
Following \cite{art13}, we write $\calA(N)=\calA(\GL_N)$, and similarly $\calA_\heartsuit(N)=\calA_\heartsuit(\GL_N)$ and $\calA_\heartsuit^+(N)=\calA_\heartsuit^+(\GL_N)$ for $\heartsuit\in\{2, \cusp\}$.
Put
\begin{align*}
    \calA_\iso^+(N)
    =\Set{
        \pi=\pi_1\boxplus\cdots\boxplus\pi_r |
        r\in\Z_{\geq1},\ N_i\in\Z_{\geq1},\ N_1+\cdots+N_r=N,\ \pi_i\in\calA_\cusp^+(N_i),\quad \forall i=1,\ldots,r
    },
\end{align*}
where $\boxplus$ denotes the isobaric sum.
It is known by M\oe glin and Waldspurger \cite{mw} that any $\pi\in\calA_2(N)$ is isomorphic to
\begin{equation*}
    \mu|\det|^{\frac{n-1}{2}} \boxplus \mu|\det|^{\frac{n-3}{2}} \boxplus \cdots \boxplus \mu|\det|^{-\frac{n-1}{2}},
\end{equation*}
for some $m, n\in\Z_{\geq1}$ and $\mu\in\calA_\cusp(m)$ with $N=mn$.
Moreover, $(m,n,\mu)$ is uniquely determined by $\pi$, and any representation of such form is an element of $\calA_2(N)$.
In other words, $\calA_2(N)$ is in bijection with $\{ (\mu,m) \mid 1\leq m|N,\ \mu\in\calA_\cusp(N/m) \}$.
We have
\begin{align}\label{chain1}
    \calA_\cusp(N) \subset \calA_2(N) \subset \calA(N).
\end{align}

Put $\Psi_\cusp(N)=\calA_\cusp(N)$, which is trivially in bijection with $\calA_\cusp(N)$.
Put $\Psi_\simple(N)$ to be the set of $\psi=\mu\boxtimes \nu$ for an irreducible finite dimensional representation $\nu$ of $\SU(2,\R)$ and $\mu\in\calA_2(N/\dim \nu)$.
It is in bijection with $\calA_2(N)$ by the fact above, for which we shall write $\psi\mapsto \pi_\psi$.
An element in $\Psi_\simple(N)$ is said to be simple.
Put $\Psi(N)$ to be the set of $\psi=\ell_1\psi_1\boxplus \cdots\boxplus \ell_r\psi_r$ for some $r\in\Z_{\geq1}$, $\ell_i, N_i\in\Z_{\geq1}$ with $\ell_1 N_1+\cdots+\ell_r N_r=N$, and mutually distinct elements $\psi_i\in\Psi_\simple(N_i)$.
To such $\psi$, we associate $\pi_\psi=\pi_{\psi_1}^{\boxplus\ell_1}\boxplus\cdots\boxplus\pi_{\psi_r}^{\boxplus\ell_r}$, so that a mapping $\psi\mapsto\pi_\psi$ gives a bijection from $\Psi(N)$ to $\calA(N)$.
For any such $\psi$, we set $c(\psi)$ equal to $c(\pi_\psi)$, which is an element in $\calC(\GL_N)$.
We have a chain
\begin{align*}
    \Psi_\cusp(N) \subset \Psi_\simple(N) \subset \Psi(N),
\end{align*}
which is compatible with the chain \eqref{chain1}.
We say that a parameter $\psi=\ell_1 (\mu_1\boxtimes\nu_1)\boxplus \cdots\boxplus \ell_r(\mu_r\boxtimes\nu_r)$ is generic if $\nu_1,\ldots,\nu_r$ are trivial.
Let $\Phi(N)$ be the subset of $\Psi(N)$ consisting of generic parameters.

Let $v$ be a place of $F$.
By LLC for $\GL_N$, for any $\mu\in\Psi_\cusp(N)$, we have an $L$-parameter $\phi_v\in\Phi(\GL_N/F_v)$ corresponding to $\mu_v \in \Pi(\GL_N/F_v)$.
By abuse of notation, let us write $\mu_v$ for $\phi_v$.
Then we obtain the localization map $\Psi_\simple(N)\to\Psi^+_v(N)$, $\psi=\mu\boxtimes\nu\mapsto \psi_v=\mu_v\boxtimes\nu$, where $\Psi^+_v(N)=\Psi^+(\GL_N/F_v)$.
This leads the localization map
\begin{align*}
    \Psi(N) & \to\Psi^+_v(N),                                                      \\
    \psi=\ell_1\psi_1\boxplus \cdots\boxplus \ell_r\psi_r,
            & \mapsto \psi_v=\ell_1\psi_{1,v}\oplus \cdots\oplus \ell_r\psi_{r,v}.
\end{align*}

Let us now recall from \cite{art13} the definition of $A$-parameters for odd special orthogonal groups and other classical groups.
An irreducible self-dual cuspidal automorphic representation $\phi$ of $\GL_N(\A_F)$ is said to be symplectic (resp. orthogonal) if the exterior (resp. symmetric) square $L$-function $L(s,\phi,\wedge^2)$ (resp. $L(s,\phi,\Sym^2)$) has a pole at $s=1$.
By the theory of Rankin-Selberg $L$-function, any irreducible self-dual cuspidal automorphic representation $\phi$ of $\GL_N(\A_F)$ is either symplectic or orthogonal, and this is mutually exclusive.
Theorems 1.4.1 and 1.5.3 of \cite{art13} tell us the following proposition.
\begin{proposition}[1st seed theorem]\label{1stseed}
    For each irreducible self-dual cuspidal automorphic representation $\phi$ of $\GL_N(\A_F)$, there exists a simple twisted endoscopic triple $\frake_\phi=(G_\phi,s_\phi,\eta_\phi)$ of $(\GL_N,\theta_N)$ and a representation $\pi\in\calA_2(G_\phi)$ such that $c(\phi)=\eta_\phi(c(\pi))$.
    The triple $\frake_\phi$ is unique up to isomorphism.
    If $\phi$ is symplectic (hence $N$ is even), then $G_\phi=\SO_{N+1}$ and $\eta_\phi$ is regarded as the unique (conjugacy class of) embedding $\Sp(N,\C)\into\GL(N,\C)$.
    If $\phi$ is orthogonal, then $\Lgp{G_\phi}=\SO(N,\C)\rtimes W_F$.
\end{proposition}
Let $\frake=(G^*,s,\eta)\in\widetilde{\calE}_\simple(N)$ be a simple twisted endoscopic triple of $(\GL_N,\theta_N)$, i.e., $G^*$ is the split symplectic group, the split odd special orthogonal group, or a quasi-split even special orthogonal group.
Set $\widetilde{\Phi}_\simple(G^*)$ to be the set of irreducible self-dual cuspidal automorphic representations $\phi$ of $\GL_N(\A_F)$ such that the corresponding endoscopic triple $\frake_\phi$ is isomorphic to $\frake$.
Here note that we cannot write $\Phi_\simple(G^*)$ because $G^*$ may be an even special orthogonal group.

An element $\psi=\mu\boxtimes\nu$ in $\Psi_\simple(N)$ is said to be self-dual if $\psi=\psi^\vee$, where $\psi^\vee=\mu^\vee\boxtimes\nu$.
We write $\widetilde{\Psi}_\simple(N)$ for the subset of $\Psi_\simple(N)$ consisting of such elements, and put $\widetilde{\Phi}_\simple(N)=\widetilde{\Psi}_\simple(N)\cap\Phi(N)$.
Moreover, an element $\psi=\boxplus_i\ell_i\psi_i$ in $\Psi(N)$ is said to be self-dual if $\psi=\psi^\vee$, where $\psi^\vee=\boxplus_i\ell_i\psi_i^\vee$.
Write $\widetilde{\Psi}(N)$ for the subset of $\Psi(N)$ consisting of such elements, and put $\widetilde{\Phi}(N)=\widetilde{\Psi}(N)\cap\Phi(N)$.

For any $\psi\in\widetilde{\Psi}(N)$, the substitute $\calL_\psi$ for the global Langlands group attached to $\psi$ is defined as follows.
We can write
\begin{equation}\label{glsd}
    \psi=\bigboxplus_{i\in I_\psi} \ell_i\psi_i \boxplus \bigboxplus_{j\in J_\psi} \ell_j (\psi_j\boxplus\psi_j^\vee),
\end{equation}
where $\psi_i=\mu_i\boxtimes\nu_i\in\widetilde{\Psi}_\simple(N)$ and $\psi_j=\mu_j\boxtimes\nu_j\in \Psi_\simple(N) \setminus \widetilde{\Psi}_\simple(N)$ for any $i\in I_\psi$, $j\in J_\psi$, and they are mutually distinct.
Put $K_\psi=I_\psi\sqcup J_\psi$ and let $m_k$ be a positive integer such that $\mu_k\in\Psi_\cusp(m_k)$ for $k\in K_\psi$.
For $i\in I_\psi$, put $H_i=G_{\mu_i}$ of Proposition \ref{1stseed} and let $\widetilde{\mu_i}$ denote for an embedding
\begin{align*}
    \eta_{\mu_i} : \Lgp{H_i}= \Lgp{G_{\mu_i}}\into \Lgp{\GL_{m_i}},
\end{align*}
of Proposition \ref{1stseed}.
For $j\in J_\psi$, put $H_j=\GL_{m_j}$ and let $\widetilde{\mu_j}$ be an embedding $\Lgp{H_j}\into\Lgp{\GL_{2m_j}}$ given by
\begin{align*}
    \GL(m_j,\C)\times W_F & \to \GL(2m_j,\C)\times W_F,                     &
    (h,w)                 & \mapsto(\diag(h,\widehat{\theta}_{m_j}(h)), w),
\end{align*}
where $\widehat{\theta}_{m_j}$ is the twist given in the introduction.
The substitute $\calL_\psi$ is the fiber product of $\{ H_k\}_{k\in K_\psi}$ over $W_F$:
\begin{align*}
    \calL_\psi=\prod_{k\in K_\psi} \left(\Lgp{H_k}\to W_F\right).
\end{align*}
The associated $L$-embedding $\widetilde{\psi}$ from $\calL_\psi\times\SU(2,\R)$ to $\Lgp{\GL_N}=\GL(N,\C)\times W_F$ is defined as
\begin{align*}
    \widetilde{\psi}=\bigoplus_{i\in I_\psi} \ell_i (\widetilde{\mu_i} \otimes \nu_i) \oplus \bigoplus_{j\in J_\psi} \ell_j (\widetilde{\mu_j}\otimes\nu_j).
\end{align*}

Now we shall recall the definition of the $A$-parameters for odd special orthogonal groups.
Let $G^*=\SO_{2n+1}$.
An $A$-parameter for $G^*=\SO_{2n+1}$ is a parameter $\psi\in \widetilde{\Psi}(2n)$ such that $\widetilde{\psi}$ factors through $\Lgp{G^*}=\Sp(2n,\C)\times W_F$.
We shall write $\Psi(G^*)$ for the $A$-parameters for $G^*$.
Although Arthur\cite{art13} defined a set $\widetilde{\Psi}(G^*)$ instead of $\Psi(G^*)$, it is not needed now since $N_{\GL(2n,\C)}(\Sp(2n,\C))=\Sp(2n,\C)$, cf. \cite[p.31]{art13}.
Thus for $\psi\in\Psi(\SO_{2n+1})$, we have the associated $L$-embedding from $\calL_\psi\times\SU(2,\R)$ to $\Lgp{\SO_{2n+1}}=\Sp(2n,\C)\times W_F$, which we shall also write $\widetilde{\psi}$.
We will write $\myim\psi$ for the image of $\widetilde{\psi}$, in order to treat local and global matters uniformly.
Let us write $\Phi(G^*)$ for the set of equivalence classes of generic $A$-parameters for $G$, i.e., $\Phi(G^*)=\Psi(G^*)\cap\widetilde{\Phi}(N)$.

We have another characterization of $\Psi(G^*)$ without $\calL_\psi$.
Recall that an irreducible finite dimensional representation $\nu$ of $\SU(2,\R)$ is symplectic (resp. orthogonal) if $\dim \nu$ is even (resp. odd).
A self-dual simple parameter $\mu\boxtimes\nu\in\widetilde{\Psi}_\simple(N)$ is said to be symplectic if one of $\mu$ and $\nu$ is symplectic and the other is orthogonal, and to be orthogonal otherwise.
Let $\psi$ be a parameter \eqref{glsd} in $\widetilde{\Psi}(N)$.
Then $\psi$ is in $\Psi(G^*)$ if and only if $\ell_i$ is even for $i\in I_\psi$ with orthogonal $\psi_i$.

For $\psi\in\Psi(G^*)$, we put $S_\psi=\cent(\myim\widetilde{\psi}, \widehat{G})$, $\overline{S}_\psi=S_\psi/Z(\widehat{G}
    )^\Gamma$, $\frakS_\psi=\pi_0(S_\psi)$, and $\overline{\frakS}_\psi=\pi_0(\overline{S}_\psi)$.
Note that the Galois action is trivial in this case.
We also write $S_\psi(G)$ and so on if the group $G$ is to be emphasized.
Take a partition $I_\psi=I_\psi^+\sqcup I_\psi^-$ such that $I_\psi^+$ (resp. $I_\psi^-$) is the set of $i\in I_\psi$ such that $\psi_i$ is symplectic (resp. orthogonal).
Then we have
\begin{align*}
    S_\psi
    \cong \prod_{i \in I_\psi^+} \Or(\ell_i, \C)
    \times \prod_{i \in I_\psi^-} \Sp(\ell_i, \C)
    \times \prod_{j \in J_\psi} \GL(\ell_j ,\C),
\end{align*}
and its component group is a free $(\Z/2\Z)$-module
\begin{align*}
    \frakS_\psi
    \cong \prod_{i \in I_\psi^+} (\Z/2\Z)a_i,
\end{align*}
with a formal basis $\{a_i\}$, where $a_i$ corresponds to $\psi_i$.

As in \cite[\S4.1]{art13}, we put
\begin{align*}
    \Psi_\simple(G^*) & =\Set{\psi\in\Psi(G^*) | \abs{\overline{S}_\psi}=1},                                                         \\
    \Psi_2(G^*)       & =\Set{\psi\in\Psi(G^*) | \abs*{\overline{S}_\psi}<\infty},                                                   \\
    \Psi_\el(G^*)     & =\Set{\psi\in\Psi(G^*) | \abs*{\overline{S}_{\psi,s}}<\infty,\ \exists s\in\overline{S}_{\psi,\semisimple}},
\end{align*}
and
\begin{align*}
    \Psi_\disc(G^*) & =\Set{\psi\in\Psi(G^*) | \abs*{Z(\overline{S}_\psi)}<\infty},
\end{align*}
where $\overline{S}_{\psi,s}=\cent(s,\overline{S}_\psi)$.
Put also $\Psi_\el^2(G^*)=\Psi_\el(G^*)\setminus\Psi_2(G^*)$.
Since we have an explicit description of $S_\psi$, we can get a more explicit characterization of these sets.
For instance, $\Psi_2(G^*)$ is a set of parameters \eqref{glsd} in $\widetilde{\Psi}(N)$ such that $\ell_i=1$ for $i\in I_\psi^+$ and $I_\psi^-\sqcup J_\psi=\emptyset$.
Let us put $\Phi_\heartsuit(G^*)=\Phi(G^*)\cap\Psi_\heartsuit(G^*)$, for $\heartsuit\in \{2,\simple,\el,\disc\}$.
Put also $\Phi_\el^2(G^*)=\Phi_\el(G^*)\setminus\Phi_2(G^*)$.

Consider a parameter $\psi\in\Psi(G^*)\subset \widetilde{\Psi}(N)$ of the form \eqref{glsd}.
Let $v$ be a place of $F$.
The localization of $\psi$ at $v$ is defined as
\begin{equation*}
    \psi_v=\bigboxplus_{i\in I_\psi} \ell_i\psi_{i,v} \boxplus \bigboxplus_{j\in J_\psi} \ell_j (\psi_{j,v}\boxplus\psi_{j,v}^\vee),
\end{equation*}
where $\psi_{k,v}=\mu_{k,v}\boxtimes \nu_k$, and $\mu_{k,v}$ stands for the $L$-parameter for the local component $\mu_{k,v}$ of $\mu_k$ at $v$, for $k\in K_\psi$.
Since $\mu_{k,v}$ is a map from $L_{F_v}$ to $\GL(m_k,\C)$, one can see that $\psi_v$ is a map from $L_{F_v}\times \SU(2,\R)$ to $\GL(2n,\C)$.
By Theorem 1.4.2 of \cite{art13} and our 1st seed theorem (Proposition \ref{1stseed}), one can see that the image of $\psi_v$ is contained in $\Sp(2n,\C)$.
Thus we obtain our 2nd seed theorem:
\begin{proposition}[2nd seed theorem]\label{2ndseed}
    For an $A$-parameter $\psi\in\Psi(G^*)$ and a place $v$, the equivalence class of $\psi_v$ is uniquely determined.
    Hence, we have the localization $\psi_v\in\Psi^+(G^*_v)$.
\end{proposition}
Then there is a natural map $S_\psi\to S_{\psi_v}$, which induces natural maps
\begin{align*}
     & \overline{S}_\psi\to \overline{S}_{\psi_v},           &
     & \frakS_\psi\to \frakS_{\psi_v},                       &
     & \text{and}                                            &
     & \overline{\frakS}_\psi\to \overline{\frakS}_{\psi_v}. &
\end{align*}

If $(G^\frake,s^\frake,\eta^\frake)\in\calE(G^*)$ is an endoscopic triple for $G^*=\SO_{2n+1}$, the endoscopic group $G^\frake$ is of the form $\GL_{n_1}\times\cdots\times\GL_{n_r}\times\SO_{2m_1+1}\times\SO_{2m_2+1}$, where $r\in\Z_{\geq0}$, $n_1+\cdots+n_r+m_1+m_2=n$.
Then the set of equivalence classes of $A$-parameters for $G^\frake$ is defined as
\begin{align*}
    \Psi(G^\frake)=\Psi(n_1)\times\cdots\times\Psi(n_r)\times\Psi(\SO_{2m_1+1})\times\Psi(\SO_{2m_2+1}),
\end{align*}
or equivalently the set of pairs $(\psi,\widetilde{\psi}^\frake)$ of a parameter $\psi\in\Psi(G^*)$ and an $L$-embedding $\widetilde{\psi}^\frake : \calL_\psi\times\SU(2,\R) \to \Lgp{G^\frake}$ with $\eta^\frake\circ\widetilde{\psi}^\frake=\widetilde{\psi}$.
We have a canonical mapping $(\psi,\widetilde{\psi}^\frake)\mapsto \psi$ from $\Psi(G^\frake)$ to $\Psi(G^*)$, for which we will write $\psi^\frake \mapsto \eta^\frake\circ\psi^\frake$.
The subset of generic parameters is
\begin{align*}
    \Phi(G^\frake)=\Phi(n_1)\times\cdots\times\Phi(n_r)\times\Phi(\SO_{2m_1+1})\times\Phi(\SO_{2m_2+1}),
\end{align*}
and we define $\Psi_\heartsuit(G^\frake)$ and $\Phi_\heartsuit(G^\frake)$ ($\heartsuit=2,\disc,\ldots$) similarly.
The component groups and localization are defined just like the case of $\GL_N$ or $\SO_{2n+1}$.

Now we consider the notion of parameters for non-quasi-split odd special orthogonal groups.
Let $G$ be an inner form of $G^*=\SO_{2n+1}$.
The notion of parameters for $G$ is the same as that for $G^*$.
In addition, a parameter $\psi\in\Psi(G^*)$ is said to be $G$-relevant if it is relevant locally everywhere in the sense of \cite[\S0.4.4]{kmsw}.
Note that an inner twist $\xi:G^*\to G$ is uniquely determined by $G$ up to isomorphism, in our case.
When there is no danger of confusion, we shall simply call relevant.
In this paper, as in the local case, we never use a symbol $\Psi(G)$.
As in \cite[\S1.3.7]{kmsw}, we have the following lemma.
\begin{lemma}
    Let $\psi\in\Psi(G^*)$ be a $G$-relevant parameter and $M^*\subset G^*$ a Levi subgroup.
    If $\psi$ comes from a parameter for $M^*$, then $M^*$ transfers to $G$.
\end{lemma}


\subsection{The bijective correspondence}\label{1.4}
Here we recall the bijective correspondence $(\frake,\psi^\frake) \llra (\psi,s)$ from discussions in \cite[\S1.4]{art13} and \cite[\S1.4]{kmsw} in the case of ordinary endoscopy for $\SO_{2n+1}$.

Let $G^*$ be a connected reductive split group over a local field $F$.
Then $\Gamma$ acts trivially on $\widehat{G^*}$.
Two ordinary endoscopic triples $\frake_1=(G_1^\frake, s_1^\frake, \eta_1^\frake)$ and $\frake_2=(G_2^\frake, s_2^\frake, \eta_2^\frake)$ for $G^*$ are considered strictly equivalent (resp. equivalent) if $s_1^\frake=s_2^\frake$ (resp. $s_1^\frake=zs_2^\frake$ for some $z\in Z(\widehat{G^*})$) and $\eta_1^\frake(\Lgp{G_1^\frake})=\eta_2^\frake(\Lgp{G_2^\frake})$.
Note that the equivalence here is different from the isomorphism.

Let $E(G^*)$ (resp. $\overline{E}(G^*)$) be a complete system of representatives of strict equivalence (resp. equivalence) classes of endoscopic data for $G^*$.
We may identify $E(G^*)$ (resp. $\overline{E}(G^*)$) with the set of strict equivalence (resp. equivalence) classes of endoscopic triples for $G^*$.
(The set $E(G)$ in \cite[p.246]{art13} does not correspond to $E(G^*)$ here, but to $\overline{E}(G^*)$.)
Define $F(G^*)$ to be the set of the parameters $\psi : L_F \times \SU(2) \to \Lgp{G^*}$.
(Here we mean the actual $L$-homomorphisms, not the equivalence classes.)
For each $\psi \in F(G^*)$, we shall write $S_{\psi, \semisimple}$ (resp. $\overline{S}_{\psi,\semisimple}$) for the set of semisimple elements in $S_\psi$ (resp. $\overline{S}_\psi$).
Define
\begin{align*}
    X(G^*) & = \Set{(\frake, \psi^\frake) | \frake=(G^\frake, s^\frake, \eta^\frake) \in E(G^*),\ \psi^\frake \in F(G^\frake)}, \\
    Y(G^*) & = \Set{(\psi,s) | \psi\in F(G^*),\ s \in S_{\psi,\semisimple}}.
\end{align*}
They are left $\widehat{G^*} \times Z(\widehat{G^*})$-sets by the following actions: for $(g,z) \in \widehat{G^*}\times Z(\widehat{G^*})$,
\begin{align*}
    (g,z)\cdot((G^\frake,s^\frake,\eta^\frake), \psi^\frake) & =((G^\frake_1,s^\frake_1, \eta^\frake_1), (\eta^\frake_1)\inv\circ \Ad(g)\circ \eta^\frake\circ \psi^\frake), \\
    (g,z)\cdot(\psi,s)                                       & =(\Ad(g)\circ\psi, zgsg\inv),
\end{align*}
where $(G^\frake_1,s^\frake_1, \eta^\frake_1)$ is an element of $E(G^*)$ that is strictly equivalent to $(G^\frake,zgs^\frake g\inv,\Ad(g)\circ \eta^\frake)$.

Put
\begin{align*}
    \overline{X}(G^*) & = \Set{(\frake, \psi^\frake) | \frake=(G^\frake, s^\frake, \eta^\frake) \in \overline{E}(G^*),\ \psi^\frake \in F(G^\frake)}, \\
    \overline{Y}(G^*) & = \Set{(\psi,s) | \psi\in F(G^*),\ s \in \overline{S}_{\psi,\semisimple}},
\end{align*}
i.e., $\overline{X}(G^*)$ (resp. $\overline{Y}(G^*)$) is the quotient set $Z(\widehat{G^*})\backslash\backslash X(G^*)$ (resp. $Z(\widehat{G^*})\backslash\backslash Y(G^*)$) of $X(G^*)$ (resp. $Y(G^*)$) by $Z(\widehat{G^*})=\{1\}\times Z(\widehat{G^*})$.
Let us put
\begin{align*}
    \calX(G^*)=\widehat{G^*}\backslash\backslash X(G^*),\quad
    \calY(G^*)=\widehat{G^*}\backslash\backslash Y(G^*),
\end{align*}
standing for the quotient sets of $X(G^*)$ and $Y(G^*)$ by $\widehat{G^*}=\widehat{G^*}\times\{1\}$, respectively.
Likewise $\overline{\calX}(G^*)=\widehat{G^*}\backslash\backslash \overline{X}(G^*)$, $\overline{\calY}(G^*)=\widehat{G^*}\backslash\backslash \overline{Y}(G^*)$.
Note that $\calE(G^*)=\widehat{G^*}\backslash\backslash E(G^*)$ and $\overline{\calE}(G^*)=\widehat{G^*}\backslash\backslash \overline{E}(G^*)=\widehat{G^*}\times Z(\widehat{G^*})\backslash\backslash E(G^*)$.
We can also see that
\begin{align*}
    \calX(G^*)            & = \Set{(\frake, \psi^\frake) | \frake=(G^\frake, s^\frake, \eta^\frake) \in \calE(G^*),\ \psi^\frake \in \widetilde{\Psi}(G^\frake)},            \\
    \calY(G^*)            & = \Set{(\psi,s) | \psi\in \Psi(G^*),\ s \in S_{\psi,\semisimple}/(\widehat{G^*}-\mathrm{conj.})},                                                \\
    \overline{\calX}(G^*) & = \Set{(\frake, \psi^\frake) | \frake=(G^\frake, s^\frake, \eta^\frake) \in \overline{\calE}(G^*),\ \psi^\frake \in \widetilde{\Psi}(G^\frake)}, \\
    \overline{\calY}(G^*) & = \Set{(\psi,s) | \psi\in \Psi(G^*),\ s \in \overline{S}_{\psi,\semisimple}/(\widehat{G^*}-\mathrm{conj.})},
\end{align*}
where $\widetilde{\Psi}(G^\frake)$ is the quotient set of $\Psi(G^\frake)$ modulo the action of $\Aut_{G^*}(\frake)$.
We use the symbol $\widetilde{\Psi}$ here, because the idea is similar to that of $\widetilde{\Psi}(G^*)$. (See \cite[p.31]{art13}.)

\begin{lemma}\label{1.4.1}
    Let $G^*$ be $\SO_{2n+1}$ defined over a local field $F$.
    Then the natural map
    \begin{align*}
        X(G^*)\to Y(G^*), \qquad (\frake,\psi^\frake)\mapsto (\eta^\frake\circ \psi^\frake, s^\frake),
    \end{align*}
    is a $\widehat{G^*}\times Z(G^*)$-equivariant bijection.
    Here $s^\frake$ and $\eta^\frake$ are the second and third elements of $\frake$ respectively,
    Furthermore, this induces a $\widehat{G^*}$-equivariant bijection $\overline{X}(G^*)\simeq \overline{Y}(G^*)$, a $Z(\widehat{G^*})$- equivariant bijection $\calX(G^*)\simeq \calY(G^*)$, and a bijection $\overline{\calX}(G^*)\simeq \overline{\calY}(G^*)$.
\end{lemma}
\begin{proof}
    Recall that $E(G^*)$ is a complete system of representatives, not just a quotient set.
    The map $X(G^*)\to Y(G^*)$ is trivially well-defined.
    The equivariance under the actions of $\widehat{G^*}\times Z(G^*)$ follows immediately from the definition of the actions.
    The injectivity also follows easily from the definition of strict equivalence of endoscopic triples.
    Let us show the surjectivity.
    Let $(\psi,s)\in Y(G^*)$.
    Put $\calG^\frake=\cent(s,\widehat{G^*}) \cdot \psi(L_F\times \SU(2))$.
    We know that $\cent(s,\widehat{G^*})$ is of the form \eqref{dge}.
    Let $G'$ be the group \eqref{ge}.
    Then its dual group $\widehat{G'}$ is isomorphic to $\cent(s,\widehat{G^*})$, with the trivial Galois action.
    Thus we have $\calG^\frake\simeq \Lgp{G'}$.
    Let us write $\eta'$ for an injective map $\Lgp{G'}\simeq \calG^\frake \into \Lgp{G^*}$, to obtain an endoscopic triple $(G',s,\eta')$.
    Let $\frake=(G^\frake,s^\frake,\eta^\frake) \in E(G^*)$ be the unique element that is strictly equivalent to $(G',s,\eta')$.
    Since $\myim\psi \subset \calG^\frake=\eta'(\Lgp{G'})=\eta^\frake(\Lgp{G^\frake})$,  there exists a parameter $\psi^\frake\in F(G^\frake)$ such that $\psi=\eta^\frake \circ \psi^\frake$.
    Then we obtain $(\frake,\psi^\frake)\in X(G^*)$, and get a map from $Y(G^*)$ to $X(G^*)$.
    This shows the surjectivity.

    The latter assertion follows from the former.
\end{proof}

Let $F$ be a number field, and $G^*=\SO_{2n+1}$.
For an endoscopic triple $\frake$ for $G^*$, every element in $\Psi(G^\frake)$ can be regarded as an element in $\Psi(G^*)$ in the natural way.
For two parameters $\psi_1,\psi_2 \in \Psi(G^\frake)$, we write $\psi_1\sim\psi_2$ if $\eta^\frake\circ\psi_1=\eta^\frake\circ\psi_2$, i.e., they are same in $\Psi(G^*)$.
Let $\widetilde{\Psi}(G^\frake)=\Psi(G^\frake)/\sim$ be the quotient set.
The reason why we use the symbol $\widetilde{\Psi}$ is because the idea is similar to that of $\widetilde{\Psi}(G^*)$ (see \cite[p.31]{art13}).
Define
\begin{align*}
    \calX(G^*) & = \Set{(\frake, \psi^\frake) | \frake=(G^\frake, s^\frake, \eta^\frake) \in \calE(G^*),\ \psi^\frake \in \widetilde{\Psi}(G^\frake)}, \\
    \calY(G^*) & = \Set{(\psi,s) | \psi\in \Psi(G^*),\ s \in S_{\psi,\semisimple}/(\widehat{G^*}-\mathrm{conj.})}.
\end{align*}

As in the local case, we have the following lemma.
\begin{lemma}\label{1.4.3}
    Let $G^*=\SO_{2n+1}$ be the global split odd special orthogonal group over a global field $F$.
    Then the natural map
    \begin{align*}
        \calX(G^*)\to \calY(G^*), \qquad (\frake,\psi^\frake)\mapsto (\psi^\frake, s^\frake)
    \end{align*}
    is bijective.
\end{lemma}
\begin{proof}
    The proof is similar to that of Lemma \ref{1.4.1}, with $\widetilde{\psi}(\calL_\psi\times \SU(2))$ in place of $\psi(L_F\times \SU(2))$.
\end{proof}

\subsection{Main theorems}\label{1.6}
In this subsection we recall the statements of the endoscopic classification of representations of the odd special orthogonal groups, which will be proved for the generic part in this paper, and was already proven for the quasi-split case by Arthur \cite{art13}.

Let first $F$ be a local field.
Fix a nontrivial additive character $\psi_F:F\to \C^1$.
Let $G^*=\SO_{2n+1}$ be the split special orthogonal group of size $2n+1$ defined over $F$, equipped with the fixed pinning.
It can be regarded as a twisted endoscopic group of $(\GL_{2n},\theta_{2n})$.
Arthur proved the existence of the stable linear form satisfying the twisted ECR:
\begin{proposition}[{\cite[Theorem 2.2.1.(a)]{art13}}]\label{1.5.1}
    Let $\psi\in\Psi(G^*)$.
    Then there is a unique stable linear form
    \begin{equation*}
        \calH(G^*)\ni f\mapsto f^{G^*}(\psi) \in \C,
    \end{equation*}
    satisfying the twisted endoscopic character relation \cite[(2.2.3)]{art13}.
\end{proposition}
We sometimes write $f(\psi)$ for $f^{G^*}(\psi)$ if there is no danger of confusion.

Let $G$ be an inner form of $G^*$ over $F$.
We now state the main local classification theorem.
\begin{theorem*}\label{1.6.1}
    \begin{enumerate}
        \item Let $\psi\in\Psi(G^*)$. There is a finite multiset $\Pi_\psi(G)$ of irreducible unitary representations of $G(F)$, (i.e., a finite set over $\Pi_\unit(G)$ in Arthur's terminology,) satisfying the local and global theorems below. It depends only on the isomorphism class of $G$. We sometimes write $\Pi_\psi=\Pi_\psi(G)$ if there is no danger of confusion. If $\psi=\phi\in\Phi_\bdd(G^*)$, i.e., $\psi$ is generic, then the set $\Pi_\phi$ is empty if and only if $\phi$ is not $G$-relevant. In general, $\Pi_\psi$ is empty if $\psi$ is not $G$-relevant. The set is called the local $A$-packet or simply the packet for $\psi$. When $\psi=\phi$ is generic, it is also called the $L$-packet of $\phi$.
        \item The local $A$-packet is equipped with a map \begin{equation*}
                  \Pi_\psi(G) \to \Irr(\frakS_\psi,\chi_G),\quad \pi\mapsto \an{-,\pi},
              \end{equation*}
              satisfying the local and global theorems below. Here, we shall write $\Irr(\frakS_\psi,\chi_G)$ for the subset of $\Irr(\frakS_\psi)$ consisting of the elements whose pullbacks via $Z(\widehat{G})\into S_\psi \onto \frakS_\psi$ are $\chi_G$. It depends only on the isomorphism class of $G$. The map $\pi\mapsto\an{-,\pi}$ depends on the choice of the Whittaker datum, in general. However, in the case of odd special orthogonal groups, the choice is unique.
              \item(ECR) Let $\frake=(G^\frake, s^\frake, \eta^\frake)$ be an endoscopic triple for $G$ with $s^\frake\in S_\psi$, and $\psi^\frake\in \Psi(G^\frake)$ a parameter for $G^\frake$ such that $\psi=\eta^\frake\circ\psi^\frake$ is a parameter for $G$. If $f\in\calH(G)$ and $f^\frake\in\calH(G^\frake)$ are matching, then we have \begin{equation}\label{ecr}
                  f^\frake(\psi^\frake)=e(G) \sum_{\pi\in\Pi_\psi(G)} \an{s_\psi s^\frake,\pi} f(\pi),
              \end{equation}
              where the left hand side is the stable linear form given by Proposition \ref{1.5.1}.
        \item Let $\psi=\phi\in\Phi_\bdd(G^*)$ be a generic parameter. Then the packet $\Pi_\phi(G)$ is multiplicity free, and $\Pi_\phi(G)\subset \Pi_\temp(G)$. Moreover, if $F$ is non-archimedean (resp. archimedean), then the map $\Pi_\phi(G) \to \Irr(\frakS_\phi,\chi_G)$ is bijective (resp. injective).
              \item(LLC) As $\phi$ runs over $\Phi_\bdd(G^*)$ the sets $\Pi_\phi(G)$ are disjoint, and we have \begin{align*}
                   & \bigsqcup_{\phi\in\Phi_{2,\bdd}(G^*)} \Pi_\phi(G) = \Pi_{2,\temp}(G), &  & \bigsqcup_{\phi\in\Phi_{\bdd}(G^*)} \Pi_\phi(G) = \Pi_{\temp}(G). &
              \end{align*}
    \end{enumerate}
\end{theorem*}

Let $\psi\in\Psi^+(G^*)\cap\Psi^+_\unit(2n)$.
If it is not $G$-relevant, define $\Pi_\psi(G)$ to be empty.
Suppose it is $G$-relevant.
Then one can choose an inner twist $\xi:G^*\to G$, standard parabolic subgroups $P^*\subset G^*$ and $P=\xi(P^*) \subset G$ over $F$, the Levi subgroups $M^*\subset P^*$ and $M=\xi(M^*)\subset P$ over $F$, the open chamber $U \subset \fraka_M^*$ associated to $P$, a point $\lambda\in U$, and a parameter $\psi_{M^*}\in \Psi(M^*)$, such that $\psi$ is the image of the twist $\psi_{M^*,\lambda}$ of $\psi_{M^*}$ by $\lambda$.
Suppose that the packet $\Pi_{\psi_{M^*}}(M)$ and a map $\Pi_{\psi_{M^*}}(M) \to \Irr(\frakS_{\psi_{M^*}}, \chi_M)$ is given.
Then we can define the packet for $\psi$ and the associated map by
\begin{align*}
     & \Pi_\psi(G)=\Set{\calI_P(\pi_{M,\lambda}) | \pi_M \in \Pi_{\psi_{M^*}}(M)}, &  & \an{-,\calI_P(\pi_{M,\lambda})}=\an{-,\pi_M}. &
\end{align*}
Note that $\frakS_\psi=\frakS_{\psi_{M^*}}$ and $\chi_G=\chi_M$.
See \cite[pp.44-46]{art13} for more details.

Let next $F$ be a global field.
Fix a nontrivial additive character $\psi_F:F\backslash\A_F\to \C^1$.
Let $G^*=\SO_{2n+1}$ be the split special orthogonal group of size $2n+1$ defined over $F$, equipped with the fixed pinning.
Let $(\xi,z):G^*\to G$ be a pure inner twist of $G^*$.
\begin{theorem}\label{3.12}
    There exists a decomposition
    \begin{equation*}
        L^2_\disc(G(F)\backslash G(\A_F))=\bigoplus_{\psi\in\Psi(G^*)} L^2_{\disc,\psi}(G(F)\backslash G(\A_F)),
    \end{equation*}
    where $L^2_{\disc,\psi}(G(F)\backslash G(\A_F))$ is a full near equivalence class of irreducible representations $\pi=\otimes_v\pi_v$ in the discrete spectrum such that the $L$-parameter of $\pi_v$ is $\phi_{\psi_v}$ for almost all $v$.
\end{theorem}
This theorem will be proved in \S\ref{3.1}.

For any global $A$-parameter $\psi\in\Psi(G^*)$ and a place $v$ of $F$, we have the localization $\psi_v\in\Psi^+(G^*_v)\cap\Psi^+_\unit(2n)_v$ thanks to Proposition \ref{2ndseed}.
Then we have the packet $\Pi_{\psi_v}(G_v)$ and the map $\Pi_{\psi_v}(G_v) \to \Irr(\frakS_{\psi_v}, \chi_{G_v})$.
The global packet for $\psi$ is defined as
\begin{equation*}
    \Pi_\psi(G)=\Set{\pi=\bigotimes_v \pi_v | \pi_v\in \Pi_{\psi_v}(G_v),\quad \an{-,\pi_v}=1 \ \text{for almost all}\ v}.
\end{equation*}
Note that $\pi_v$ is unramified for almost all $v$ (\cite[Lemma 7.3.4]{art13}).
If a local packet $\Pi_{\psi_v}(G_v)$ is empty for some $v$, then the global packet is defined to be empty.
For any $\pi=\otimes_v \pi_v\in\Pi_\psi(G)$, we define a character $\an{-,\pi}$ of $\frakS_\psi$ by
\begin{equation*}
    \an{x,\pi}=\prod_v \an{x_v,\pi_v}, \quad x\in\frakS_\psi,
\end{equation*}
where $x_v$ denotes the image of $x$ under the natural map $\frakS_\psi\to \frakS_{\psi_v}$.
The restriction of this character to $Z(\widehat{G})$ is $\prod_v \chi_{G_v}$, which is trivial by the product formula \eqref{033}.
Thus we may regard it as a character of $\overline{\frakS}_\psi$.
We have the character $\varepsilon_\psi : \overline{\frakS_\psi}\to\{\pm1\}$ defined by Arthur \cite[p.48]{art13}.
When we emphasize that the parameter $\psi$ is for $G^*$ (or equivalently for $G$), we shall write $\varepsilon^{G^*}_\psi$ or $\varepsilon^G_\psi$.
Note that $\varepsilon_\psi$ is trivial if $\psi$ is generic.
Put
\begin{equation*}
    \Pi_\psi(G,\varepsilon_\psi)=\Set{\pi\in\Pi_\psi(G) | \an{-,\pi}=\varepsilon_\psi}.
\end{equation*}
We now state the main global classification theorem.
\begin{theorem*}[AMF]\label{1.7.1}
    For $\psi\in\Psi(G^*)$, we have
    \begin{align*}
        L^2_{\disc,\psi}(G(F)\backslash G(\A_F))=\begin{dcases*}
                                                     \bigoplus_{\pi\in\Pi_\psi(G,\varepsilon_\psi)} \pi, & if $\psi\in\Psi_2(G^*)$,    \\
                                                     0,                                                  & if $\psi\notin\Psi_2(G^*)$.
                                                 \end{dcases*}
    \end{align*}
\end{theorem*}

Arthur \cite{art13} proved Theorems \ref{1.6.1}, \ref{3.12}, and \ref{1.7.1} for the case $G=G^*$.
In this paper, following Kaletha-Minguez-Shin-White \cite{kmsw}, we will give a proof of Theorems \ref{1.6.1}, \ref{3.12}, and \ref{1.7.1} for the case $\psi$ is generic.
The main theorem of this paper is stated as follows:
\begin{theorem}[main theorem]\label{main}
    Let $F$ be a local or global field and $G$ an inner form of $G^*=\SO_{2n+1}$ over $F$.
    \begin{enumerate}
        \item Let $F$ be local. Let $\psi=\phi\in\Phi_\bdd(G^*)$ be a generic parameter. Then part 1, 2, 3 of Theorem \ref{1.6.1} hold true for $\phi$. Moreover, part 4 and 5 also hold true.
        \item Let $F$ be global. Then Theorem \ref{3.12} holds. Let moreover $\psi=\phi\in\Phi(G^*)$ be a generic parameter. Then Theorem \ref{1.7.1} holds true for $\phi$.
    \end{enumerate}
\end{theorem}

As stated in Introduction, Arthur \cite{art13} also established the endoscopic classification of representations of symplectic groups and quasi-split even special orthogonal groups.
In the case of symplectic groups, he proved theorems similar to Theorems \ref{1.6.1}, \ref{3.12}, and \ref{1.7.1}.
On the other hand, in the case of quasi-split even special orthogonal groups, he proved the similar theorems up to outer automorphisms.
For a quasi-split even orthogonal group $G$ over a local field $F$, two representations $\pi$ and $\pi'$ are said to be $\epsilon$-equivalent if there exists $\rho\in\Aut_F(G)$ such that $\pi\simeq\pi'\circ\rho$, and the sets of $\epsilon$-equivalence classes is denoted by $\widetilde{\Pi}$ instead of $\Pi$.
Arthur defined a weak local $A$-packet $\Pi_\psi(G)$ for each $\epsilon$-equivalence class of $A$-parameters $\psi\in\widetilde{\Psi}(G)$, and then he proved a theorem similar to but weaker than Theorem \ref{1.6.1} in that $\Psi$ and $\Pi$ are replaced by $\widetilde{\Psi}$ and $\widetilde{\Pi}$.
For a quasi-split even orthogonal group $G$ over a number field $F$, two representations $\pi=\bigotimes_v\pi_v$ and $\pi'=\bigotimes_v\pi'_v$ are said to be $\epsilon$-equivalent if $\pi_v$ and $\pi'_v$ are so for all places $v$.
The weak global packet $\widetilde{\Pi}_\psi(G)$ for $\psi\in\widetilde{\Psi}(G)$ is defined as the set of $\pi=\bigotimes_v\pi_v$ such that $\pi_v\in\widetilde{\Pi}_{\psi_v}(G_v)$ for all $v$ and $\an{-,\pi_v}$ is trivial for almost all $v$.
As in the local case, he proved theorems similar to but weaker than Theorems \ref{3.12} and \ref{1.7.1} in that $\Psi$ and $\Pi$ are replaced by $\widetilde{\Psi}$ and $\widetilde{\Pi}$.
See also \cite{ag} for the weak theorems for even orthogonal groups.
These classification theorems for symplectic and even special orthogonal groups are needed for the globalization in \S\ref{4.3}

In the rest of this paper, we shall proof the main theorem \ref{main} by a long induction argument following \cite{kmsw}.
Let $n$ be a positive integer.
As an induction hypothesis, we assume that the Theorem \ref{main} (or Theorems \ref{1.6.1}, \ref{3.12}, and \ref{1.7.1}) holds for any positive integer $n_0<n$, and hence for any proper Levi subgroup $M^*\subsetneq G^*=\SO_{2n+1}$.

\section{Local intertwining relation}\label{2}
In this section, we define the normalized local intertwining operator and state the local intertwining relation, which is a key theorem in the theory of the endoscopic classification of representations.
After that, we reduce the proof LIR, and give a proof of LIR in the special (real and low rank) cases.
\subsection{The diagram}\label{2.1}
We shall describe explicitly the basic commutative diagram (\cite[(2.4.3)]{art13}, \cite[(2.1.1)]{kmsw}) for odd special orthogonal groups.
See \cite[\S2.1]{kmsw} for the diagram for general connected reductive groups.

Let $F$ be a local or global field, $G^*=\SO_{2n+1}$, and $M^*\subset G^*$ a standard Levi subgroup.
Let $\xi:G^*\to G$ be an inner twist over $F$ with a Levi subgroup $M=\xi(M^*)\subset G$ which is an inner form of $M^*$ over $F$.
Let $\psi\in\Psi(M^*)$ be a parameter, and $\psi_G\in \Psi(G^*)$ the image of $\psi$ under the natural map induced by $\widehat{M}\into\widehat{G}$.
Recall that $\widehat{G}=\Sp(2n,\C)$.
The parameter $\psi_G \in \Psi(G^*)$ and its centralizer group $S_\psi(G)=S_{\psi_G}(G)$ can be written as
\begin{align*}
    \psi_G
              & = \bigoplus_{i \in I_\psi^+} \ell_i \psi_i
    \oplus \bigoplus_{i \in I_\psi^-} \ell_i \psi_i
    \oplus \bigoplus_{j \in J_\psi} \ell_j (\psi_j \oplus \psi_j^\vee), \\
    S_\psi(G) & =\cent(\myim\psi,\Sp(2n,\C))
    \simeq \prod_{i \in I_\psi^+} \Or(\ell_i, \C)
    \times \prod_{i \in I_\psi^-} \Sp(\ell_i ,\C)
    \times \prod_{j \in J_\psi} \GL(\ell_j, \C).
\end{align*}
Moreover, we can also write as
\begin{align*}
    \psi_G
     & = \bigoplus_{t=1}^r
    e_t \left[\bigoplus_{i \in I_t^+} \ell_i^t \psi_i
        \oplus \bigoplus_{i \in I_t^-} \ell_i^t \psi_i
    \oplus \bigoplus_{j \in J_t} (\ell_j^t \psi_j \oplus \ell_j^{t \vee}\psi_j^\vee)\right]   \\
     & \qquad \oplus
    \left[\bigoplus_{i \in I_0^+} \ell_i^0 \psi_i
        \oplus \bigoplus_{i \in I_0^-} \ell_i^0 \psi_i
    \oplus \bigoplus_{j \in J_0} \ell_j^0 (\psi_j \oplus \psi_j^\vee)\right]                  \\
     & \qquad \qquad \oplus
    \bigoplus_{t=1}^r
    e_t \left[\bigoplus_{i \in I_t^+} \ell_i^t \psi_i
        \oplus \bigoplus_{i \in I_t^-} \ell_i^t \psi_i
    \oplus \bigoplus_{j \in J_t} (\ell_j^{t \vee} \psi_j \oplus \ell_j^t \psi_j^\vee)\right], \\
    \widehat{M}
     & \simeq \prod_{t=1}^r \GL(k_t, \C)^{e_t}
    \times \Sp(2n_0, \C),
\end{align*}
where $e_t$, $k_t$ are positive integers,
\begin{align*}
    \begin{dcases*}
        I_t^+ \subset I_\psi^+, &                       \\
        I_t^- \subset I_\psi^-, &                       \\
        J_t \subset J_\psi,     & for $t=0,1,\ldots,r$,
    \end{dcases*}
\end{align*}
are subsets of indices, and $\ell_i^t$, $\ell_j^t$ are positive integers such that
\begin{align*}
    \begin{dcases*}
        \ell_i^t \leq \frac{\ell_i}{2},                                                          &                       \\
        \ell_j^t, \ \ell_j^{t \vee} \leq \ell_j, \quad   \ell_j^t + \ell_j^{t \vee} \leq \ell_j, & for $t=0,1,\ldots,r$,
    \end{dcases*}
\end{align*}
and $\Set{((\ell_i^t)_{i \in I_t^+}, (\ell_i^t)_{i \in I_t^-}, (\ell_j^t)_{j \in J_t}, (\ell_j^{t \vee})_{j \in J_t}) | t=1,\ldots, r}$ is mutually distinct.
Put
\begin{align*}
    \psi_t                                                                            & = \bigoplus_{i \in I_t^+} \ell_i^t \psi_i
    \oplus \bigoplus_{i \in I_t^-} \ell_i^t \psi_i
    \oplus \bigoplus_{j \in J_t} (\ell_j^t \psi_j \oplus \ell_j^{t \vee}\psi_j^\vee), &                                           & \text{for $t=1,\ldots,r$}, \\
    \psi_0                                                                            & = \bigoplus_{i \in I_0^+} \ell_i^0 \psi_i
    \oplus \bigoplus_{i \in I_0^-} \ell_i^0 \psi_i
    \oplus \bigoplus_{j \in J_0} \ell_j^0 (\psi_j \oplus \psi_j^\vee).                &                                           &
\end{align*}
Then we have
\begin{align*}
    \psi_G & = \bigoplus_{t=1}^r e_t \psi_t \oplus \psi_0 \oplus \bigoplus_{t=1}^r e_t \psi_t^\vee, \\
    \psi   & =\bigoplus_{t=1}^r e_t\psi_t \oplus \psi_0,
\end{align*}
and each $\psi_t$ corresponds to $\GL(k_t, \C) \subset \widehat{M}$, and $\psi_0$ corresponds to $\Sp(2n_0, \C) \subset \widehat{M}$.
Here we note that
\begin{itemize}
    \item $k_1, \ldots, k_r$ are not necessarily mutually distinct;
    \item $\psi_1, \ldots, \psi_r$ are mutually distinct;
    \item $\psi_1, \ldots, \psi_r, \psi_1^\vee, \ldots, \psi_r^\vee$ are not necessarily mutually distinct.
\end{itemize}
Let $A_{\widehat{M}}$ be the maximal central split torus of $\widehat{M}$.
Then $\widehat{M} = \cent(A_{\widehat{M}}, \widehat{G})$.
One has
\begin{equation*}
    A_{\widehat{M}} \simeq \prod_{t=1}^r (\C^\times)^{e_t}.
\end{equation*}
Put
\begin{align*}
    S_\psi(G) & = \cent(\myim\psi, \widehat{G}), & S_\psi(M) & = \cent(\myim\psi, \widehat{M}),
\end{align*}
\begin{equation*}
    N_\psi(M, G) = S_\psi(G) \cap N(A_{\widehat{M}}, \widehat{G}) = N(A_{\widehat{M}}, S_\psi(G)).
\end{equation*}
Put moreover
\begin{align*}
    W_\psi(M, G)       & = \frac{N(A_{\widehat{M}}, S_\psi(G))}{Z(A_{\widehat{M}}, S_\psi(G))},             &
    W_\psi^\circ(M, G) & = \frac{N(A_{\widehat{M}}, S_\psi(G)^\circ)}{Z(A_{\widehat{M}}, S_\psi(G)^\circ)},
\end{align*}
\begin{align*}
    \frakN_\psi(M, G) & = \frac{N(A_{\widehat{M}}, S_\psi(G))}{Z(A_{\widehat{M}}, S_\psi(G)^\circ)}, &
    \frakS_\psi(M, G) & = \frac{N(A_{\widehat{M}}, S_\psi(G))}{N(A_{\widehat{M}}, S_\psi(G)^\circ)},
\end{align*}
\begin{equation*}
    R_\psi(M, G)= \frac{N(A_{\widehat{M}}, S_\psi(G))}{N(A_{\widehat{M}}, S_\psi(G)^\circ) \cdot Z(A_{\widehat{M}}, S_\psi(G))},
\end{equation*}
and
\begin{align*}
    \frakS_\psi(G) & = \pi_0(S_\psi(G))
    \simeq \prod_{i \in I_\psi^+} \Or(\ell_i, \C) / \SO(\ell_i, \C)
    \simeq \bigoplus_{i \in I_\psi^+} (\Z/2\Z)a_i, \\
    \frakS_\psi(M) & = \pi_0(S_\psi(M))
    \simeq \prod_{i \in I_0^+} \Or(\ell_i^0, \C) / \SO(\ell_i^0, \C)
    \simeq \bigoplus_{i \in I_0^+} (\Z/2\Z)a_i,
\end{align*}
where $\{a_i\}_i$ is a formal basis.
Direct calculations show that
\begin{align*}
    Z(A_{\widehat{M}}, S_\psi(G)^\circ)
     & \simeq \prod_{t=1}^r \left[ \prod_{i \in I_t^+} \GL(\ell_i^t, \C)
        \times \prod_{i \in I_t^-} \GL(\ell_i^t,\C)
    \times \prod_{j \in J_t} (\GL(\ell_j^t, \C) \times \GL(\ell_j^{t \vee}, \C) ) \right]^{e_t} \\
     & \qquad  \times \left[ \prod_{i \in I_0^+} \SO(\ell_i^0, \C)
        \times \prod_{i \in I_0^-} \Sp(\ell_i^0, \C)
    \times \prod_{j \in J_0} \GL(\ell_j^0, \C) \right],                                         \\
    Z(A_{\widehat{M}}, S_\psi(G))
     & \simeq \prod_{t=1}^r \left[ \prod_{i \in I_t^+} \GL(\ell_i^t, \C)
        \times \prod_{i \in I_t^-} \GL(\ell_i^t,\C)
    \times \prod_{j \in J_t} (\GL(\ell_j^t, \C) \times \GL(\ell_j^{t \vee}, \C) ) \right]^{e_t} \\
     & \qquad \times \left[ \prod_{i \in I_0^+} \Or(\ell_i^0, \C)
        \times \prod_{i \in I_0^-} \Sp(\ell_i^0, \C)
        \times \prod_{j \in J_0} \GL(\ell_j^0, \C) \right].
\end{align*}
In particular, the connected component of the identity in $Z(A_{\widehat{M}}, S_\psi(G))$ is equal to $Z(A_{\widehat{M}}, S_\psi(G)^\circ)$.
Thus, we have
\begin{gather}
    \widehat{M}\cap S_\psi(G)^\circ = S_\psi(M)^\circ, \label{shake}\\
    \frac{Z(A_{\widehat{M}}, S_\psi(G))}{Z(A_{\widehat{M}}, S_\psi(G)^\circ)} \simeq \frakS_\psi(M) \simeq \bigoplus_{i \in I_0^+} (\Z/2\Z)a_i. \label{d/b}
\end{gather}
We now consider the group $N(A_{\widehat{M}}, S_\psi(G))$.
Take a partition
\begin{equation*}
    \set{1, \ldots, r} = T_1 \sqcup T_2 \sqcup T_3 \sqcup T_4
\end{equation*}
such that
\begin{align*}
    T_1 & = \Set{t| \psi_t \simeq \psi_t^\vee},                                         \\
    T_2 & = \Set{t| \psi_t \not\simeq \psi_{t'}^\vee \text{ for any } 1\leq t' \leq r}, \\
    T_3 & = \Set{t| \psi_t \simeq \psi_{t'}^\vee \text{ for a unique } t' \in T_4},     \\
    T_4 & = \Set{t| \psi_t \simeq \psi_{t'}^\vee \text{ for a unique } t' \in T_3},
\end{align*}
and for each $t \in T_3$, put $e_t^\vee = e_{t'}$, where $t'$ is the element of $T_4$ such that $\psi_t \simeq \psi_{t'}^\vee$.

Before the description of $N(A_{\widehat{M}}, S_\psi(G))$, we must fix notation.
For a finite set $X$, let $\frakS_X$ denote its symmetric group, and put $\frakS_m=\frakS_{\{1,2,\ldots,m\}}$ for a positive integer $m$.
For any positive integers $e, e^\vee$, let
\begin{equation*}
    W(e,e^\vee)
\end{equation*}
be a subgroup of $\weylbc{e + e^\vee}$ generated by
\begin{align*}
    \frakS_{\set{1, \ldots, e}}, \ \frakS_{\set{e+1, \ldots, e+e^\vee}},
    \text{ and } \Set{(h,h^\vee) \ltimes (\cdots, 0, \overset{h}{1}, 0, \cdots, 0, \overset{h^\vee}{1}, 0, \cdots)
        | 1\leq h \leq e,\   e+1 \leq h^\vee \leq e+e^\vee}.
\end{align*}
Then we obtain an explicit description of the group $N(A_{\widehat{M}}, S_\psi(G))$.
We have
\begin{align*}
    N(A_{\widehat{M}}, S_\psi(G)) \simeq c' \times \left( \prod_{i \in I_0^+} \Or(\ell_i^0, \C),
    \times \prod_{i \in I_0^-} \Sp(\ell_i^0, \C)
    \times \prod_{j \in J_0} \GL(\ell_j^0, \C) \right),
\end{align*}
where a subgroup $c'$ of $\Sp(2(n-n_0), \C)$ can be written as
\begin{align*}
    c' & \simeq \prod_{t \in T_1} \left[ \bigcup_{w_t \in \weylbc{e_t} } w_t \cdot
        \left(
        \prod_{i \in I_t^+} \GL(\ell_i^t, \C)
        \times \prod_{i \in I_t^-} \GL(\ell_i^t,\C)
        \times \prod_{j \in J_t} (\GL(\ell_j^t, \C) \times \GL(\ell_j^{t \vee}, \C))
        \right)^{e_t}
    \right]                                                                                   \\
       & \quad \times \prod_{t \in T_2} \left[ \bigcup_{w_t \in \frakS_{e_t} } w_t \cdot
        \left(
        \prod_{i \in I_t^+} \GL(\ell_i^t, \C)
        \times \prod_{i \in I_t^-} \GL(\ell_i^t,\C)
        \times \prod_{j \in J_t} (\GL(\ell_j^t, \C) \times \GL(\ell_j^{t \vee}, \C))
        \right)^{e_t}
    \right]                                                                                   \\
       & \qquad \times \prod_{t \in T_3} \left[ \bigcup_{w_t \in W(e_t, e_t^\vee) } w_t \cdot
        \left(
        \prod_{i \in I_t^+} \GL(\ell_i^t, \C)
        \times \prod_{i \in I_t^-} \GL(\ell_i^t,\C)
        \times \prod_{j \in J_t} (\GL(\ell_j^t, \C) \times \GL(\ell_j^{t \vee}, \C))
        \right)^{e_t+e_t^\vee}
        \right].
\end{align*}
Therefore we have natural isomorphisms
\begin{align}
    W_\psi(M,G)      & \simeq \prod_{t \in T_1} \left(\weylbc{e_t} \right)
    \times \prod_{t \in T_2} \frakS_{e_t}
    \times \prod_{t \in T_3} W(e_t, e_t^\vee),\label{c/d}                   \\
    \frakN_\psi(M,G) & \simeq \frakS_\psi(M) \times W_\psi(M,G).\label{c/b}
\end{align}

We finally consider $N(A_{\widehat{M}}, S_\psi(G)^\circ)$.
It is too complicated to write explicitly, but we can describe the natural surjection
\begin{align*}
    \frakN_\psi(M,G)=\frac{N(A_{\widehat{M}}, S_\psi(G))}{Z(A_{\widehat{M}}, S_\psi(G)^\circ)} \lra \frakS_\psi(M,G)=\frac{N(A_{\widehat{M}}, S_\psi(G))}{N(A_{\widehat{M}}, S_\psi(G)^\circ)},
\end{align*}
to understand the group $N(A_{\widehat{M}}, S_\psi(G)^\circ)$ in terms of its kernel.
For $e \geq 1$, let us write $x$ for a group homomorphism
\begin{align*}
     & \frakS_e \ltimes (\Z/2\Z)^e \lra \Z/2\Z,               &
     & \sigma \ltimes (d_h)_{h=1}^e \mapsto \sum_{h=1}^e d_h. &
\end{align*}
For each $i \in I_\psi^+$, we shall define a corresponding function
\begin{equation*}
    x_i : \frakN_\psi(M,G) \lra \Z/2\Z
\end{equation*}
as follows.
By the equations \eqref{c/b}, \eqref{c/d}, and \eqref{d/b}, we may identify $\frakN_\psi(M,G)$ with
\begin{equation}\label{c/b2}
    \bigoplus_{i \in I_0^+} (\Z/2\Z)a_i
    \times \prod_{t \in T_1} \left( \weylbc{e_t} \right)
    \times \prod_{t \in T_2} \frakS_{e_t}
    \times \prod_{t \in T_3} W(e_t, e_t^\vee).
\end{equation}
For any element
\begin{equation*}
    u = \left( \sum_{i \in I_0^+} c_i a_i, (w_t)_{t \in T_1}, (w_t)_{t \in T_2}, (w_t)_{t \in T_3} \right)
\end{equation*}
of \eqref{c/b2}, we define
\begin{align*}
    x_i(u) = c_i + \sum_{t \in T_1 \sqcup T_3} \ell_i^t x(w_t),
\end{align*}
where we put $c_i = 0$ for $i \notin I_0^+$, and $\ell_i^t = 0$ for $i \notin I_t^+$.
Now, we can describe the natural surjection.
Let us define a homomorphism $x_\psi : \frakN_\psi(M,G) \to \bigoplus_{i \in I_\psi^+} (\Z/2\Z)a_i = \frakS_\psi(G)$ by
\begin{equation*}
    x_\psi(u) = \sum_{i \in I_\psi^+} x_i(u) a_i
\end{equation*}
for $u \in \frakN_\psi(M,G)$.
Then $\frakS_\psi(M,G)$ and $W_\psi^\circ(M,G)$ can be identified with the image and the kernel of $x_\psi$, respectively.
Put $\Or^+(l,\C) = \SO(l,\C)$ and $\Or^-(l,\C) = \Or(l,\C) \setminus \SO(l,\C)$.
The homomorphism $x_\psi$ gives a description:
\begin{align*}
    N(A_{\widehat{M}}, S_\psi(G)^\circ) & \simeq \bigsqcup_{u} \{
    \prod_{t \in T_1}
    \left[
        w_t \cdot \left(
        \prod_{i \in I_t^+} \GL(\ell_i^t, \C)
        \times \prod_{i \in I_t^-} \GL(\ell_i^t,\C)
        \times \prod_{j \in J_t} (\GL(\ell_j^t, \C) \times \GL(\ell_j^{t \vee}, \C))
        \right)^{e_t}
    \right]                                                                     \\
                                        & \qquad \times \prod_{t \in T_2}
    \left[
        w_t \cdot \left(
        \prod_{i \in I_t^+} \GL(\ell_i^t, \C)
        \times \prod_{i \in I_t^-} \GL(\ell_i^t,\C)
        \times \prod_{j \in J_t} (\GL(\ell_j^t, \C) \times \GL(\ell_j^{t \vee}, \C))
        \right)^{e_t}
    \right]                                                                     \\
                                        & \qquad \quad \times \prod_{t \in T_3}
    \left[
        w_t \cdot \left(
        \prod_{i \in I_t^+} \GL(\ell_i^t, \C)
        \times \prod_{i \in I_t^-} \GL(\ell_i^t,\C)
        \times \prod_{j \in J_t} (\GL(\ell_j^t, \C) \times \GL(\ell_j^{t \vee}, \C))
        \right)^{e_t + e_t^\vee}
    \right]                                                                     \\
                                        & \qquad \qquad \times \left[
        \prod_{i \in I_0^+} \Or^{\varepsilon_i}(\ell_i^0, \C),
        \times \prod_{i \in I_0^-} \Sp(\ell_i^0, \C)
        \times \prod_{j \in J_0} \GL(\ell_j^0, \C)
        \right]
    \},
\end{align*}
where $u = \left(\sum_{i \in I_0^+} c_i a_i, (w_t)_{t \in T_1}, (w_t)_{t \in T_2}, (w_t)_{t \in T_3} \right)$ runs over $\ker(x_\psi)$, and $\varepsilon_i = (-1)^{c_i}$.

We have obtained the explicit description of the following commutative diagram with exact rows and columns.
\begin{align}\label{cd}
    \begin{CD}
        @.                             @.                   1                  @.              1              @.       \\
        @.                   @.                                 @VVV                            @VVV                @.  \\
        @.                             @.       W_\psi^\circ(M, G)      @=    W_\psi^\circ(M, G)  @.        \\
        @.                   @.                                   @VVV                            @VVV                @.   \\
        1  @>>>  \frakS_\psi(M)  @>>>  \frakN_\psi(M, G)  @>>>  W_\psi(M, G)  @>>>  1   \\
        @.                  @|                                 @V{x_\psi}VV                       @VVV                  @.  \\
        1   @>>> \frakS_\psi(M)  @>>>  \frakS_\psi(M, G)     @>>>     R_\psi(M, G)    @>>>  1   \\
        @.                  @.                                   @VVV                             @VVV                @.  \\
        @.                             @.                      1                @.                 1          @.
    \end{CD}
\end{align}

\begin{lemma}\label{2.8.3}
    If $\psi\in \Psi_2(M^*)$, then $x_\psi$ is surjective.
    In particular, we have $\frakS_\psi(M,G)=\frakS_\psi(G)$.
\end{lemma}
\begin{proof}
    Suppose that $\psi\in \Psi_2(M^*)$.
    Then $K_t:=I_t^+\sqcup I_t^-\sqcup J_t$ is a singleton for $t=1,\ldots,r$, and $\ell_i^t=1$ for $t=0,1,\ldots,r$.
    Let $i\in I_0^+$.
    If we put $u=(a_i,1,1,1)\in \frakN_\psi(M,G)$, then $x_\psi(u)=a_i$.
    Let $i\in I_t^+$ ($t=1,\ldots,r$).
    Then $t\in T_1$, as $\psi_t=\psi_i$ is self-dual.
    If we put $u=(0,(w_s)_{s\in T_1},1,1)\in \frakN_\psi(M,G)$, where $w_t=1\ltimes (1,0,\ldots,0)$ and $w_s$is trivial if $s\neq t$.
    Then $x_\psi(u)=a_i$.
    Since $I_\psi^+$ is the union of $I_0^+, I_1^+,\ldots, T_r^+$, this completes the proof.
\end{proof}

Dividing the diagram \eqref{cd} by $Z(\widehat{G})^\Gamma=\{\pm1\}$, we obtain another diagram
\begin{align*}
    \begin{CD}
        @.                             @.                   1                  @.              1              @.       \\
        @.                   @.                                 @VVV                            @VVV                @.  \\
        @.                             @.       W_\psi^\circ(M, G)      @=    W_\psi^\circ(M, G)  @.        \\
        @.                   @.                                   @VVV                            @VVV                @.   \\
        1  @>>>  \overline{\frakS}_\psi(M)  @>>> \overline{\frakN}_\psi(M, G)  @>>>  W_\psi(M, G)  @>>>  1   \\
        @.                  @|                                   @VVV                       @VVV                  @.  \\
        1   @>>> \overline{\frakS}_\psi(M)  @>>> \overline{\frakS}_\psi(M, G)     @>>>     R_\psi(M, G)    @>>>  1   \\
        @.                  @.                                   @VVV                             @VVV                @.  \\
        @.                             @.                      1                @.                 1          @.
    \end{CD}
\end{align*}
where
\begin{align*}
    \overline{\frakS}_\psi(M)    & =\frac{\frakS_\psi(M)}{Z(\widehat{G})^\Gamma},    &
    \overline{\frakN}_\psi(M, G) & =\frac{\frakN_\psi(M, G)}{Z(\widehat{G})^\Gamma}, &
    \overline{\frakS}_\psi(M, G) & =\frac{\frakS_\psi(M, G)}{Z(\widehat{G})^\Gamma}. &
\end{align*}

\subsection{The first intertwining operator}\label{2.2}
Let $F$ be a local field, $\psi_F:F\to\C^1$ a nontrivial additive character, $G^*=\SO_{2n+1}$, and $M^*\subsetneq G^*$ a proper standard Levi subgroup.
Let $\xi:G^*\to G$ be an inner twist over $F$ which restrict to an inner twist $\xi|_{M^*} : M^*\to M$ over $F$.
Let $\psi\in\Psi(M^*)$ be a local parameter.
Assume that $\psi$ is $M$-relevant, and the corresponding packet $\Pi_\psi(M)$ is not empty.
Let $\pi\in\Pi_\psi(M)$.
We shall write $V_\pi$ for the representation space of $\pi$.

Let $P$ and $P'$ be parabolic subgroups of $G$ defined over $F$ with common Levi factor $M$.
We write $(\calI_P^G(\pi), \calH_P^G(\pi))$ or $(\calI_P(\pi), \calH_P(\pi))$ for the representation parabolically induced by $(\pi,V_\pi)$ from $P$ to $G$.

A function $H_M:M(F)\to\fraka_M$ is defined by
\begin{equation*}
    \exp(\an{H_M(m),\chi})=|\chi(m)|_F,
\end{equation*}
for any $\chi\in X^*(M)_F,$ $m\in M(F)$.
Then each $\lambda \in \fraka_{M,\C}^*$ gives a character
\begin{align*}
    M(F) \to \C^\times, \quad m \mapsto\exp(\an{H_M(m),\lambda}).
\end{align*}
We shall write $\pi_\lambda$ for the tensor product of $\pi$ and this character.
Then the unnormalized intertwining operator $J_{P'|P}(\xi,\psi_F) : \calH_P(\pi_\lambda) \to \calH_{P'}(\pi_\lambda)$ is given by
\begin{equation*}
    [J_{P'|P}(\xi,\psi_F)f](g) = \int_{N(F)\cap N'(F) \backslash N'(F)} f(n'g) dn',
\end{equation*}
where $N$ and $N'$ are the unipotent radical of $P$ and $P'$, respectively.
It is known that the integral converges absolutely when the real part of $\lambda$ lies in a certain open cone.
The Haar measure $dn'$ on $N(F)\cap N'(F) \backslash N'(F)$ that we use here is the measure \cite[\S2.2]{kmsw} defined.
Although the case of unitary groups is considered in loc. cit., we apply the same definition with the same notation.

We write $W_F \ni w \mapsto |w|^\lambda$ for the $L$-parameter of the character $M(F)\ni m \mapsto\exp(\an{H_M(m),\lambda})$ attached to $\lambda \in \fraka_{M,\C}^*$.
Then $\psi_\lambda=\psi|-|^\lambda$ is a parameter of $\pi_\lambda$.
Let $\rho_{P'|P}$ be the adjoint representation of $\widehat{M}$ on $\widehat{\frakn}'\cap\widehat{\frakn}\backslash\widehat{\frakn}'$.
Put
\begin{equation*}
    r_{P'|P}(\xi,\psi_\lambda, \psi_F)
    =\frac{L(0,\rho_{P'|P}^\vee\circ\phi_{\psi_\lambda})}{L(1,\rho_{P'|P}^\vee\circ\phi_{\psi_\lambda})} \frac{\epsilon(\frac{1}{2},\rho_{P'|P}^\vee\circ\phi_{\psi_\lambda}, \psi_F)}{\epsilon(0,\rho_{P'|P}^\vee\circ\phi_{\psi_\lambda}, \psi_F)},
\end{equation*}
where the factors are the Artin $L$- and $\epsilon$- factors.
The normalized intertwining operator is given by
\begin{equation*}
    R_{P'|P}(\xi,\psi_\lambda) = r_{P'|P}(\xi,\psi_\lambda,\psi_F)\inv J_{P'|P}(\xi,\psi_F).
\end{equation*}
It is known that $R_{P'|P}(\xi,\psi_\lambda)$ is independent of the choice of $\psi_F$, and that $\lambda \mapsto R_{P'|P}(\xi,\psi_\lambda)$ has meromorphic continuation to whole $\fraka_{M,\C}^*$.

\begin{lemma}\label{2.2.1}
    Assume that $F=\R$ and $\psi=\phi\in\Phi(M^*)$.
    Then the function $\lambda \mapsto R_{P'|P}(\xi,\phi_\lambda)$ has neither a zero nor a pole at $\lambda=0$.
    Moreover, let $P''\subset G$ be a parabolic subgroup defined over $F$ with Levi factor $M$.
    Then we have
    \begin{equation*}
        R_{P''|P}(\xi,\phi)=R_{P''|P'}(\xi,\phi) \circ R_{P'|P}(\xi,\phi).
    \end{equation*}
\end{lemma}
\begin{proof}
    The idea of the proof is the same as that of \cite[Lemma 2.2.1]{kmsw}.
    So it suffices to show that the measure used in the definition of $J_{P'|P}(\xi,\psi_F)$ coincides with the measure introduced by Arthur \cite[\S3]{artjfa}.

    Let us calculate the measure given in loc. cit.
    Since $F=\R$, there exist nonnegative integers $p$ and $q$ with $p+q=2n+1$ such that the inner twist $\xi:G^*\to G$ is realized by $\xi_{p,q}:\SO_{2n+1} \to \SO(p,q)$ given in \S\ref{reallie}.
    So, we will use the notation introduced there.
    Let $B$ be a $\SO(p,q)$-invariant bilinear form on $\frakg_\C=\so(p,q)_\C$ given by $B(X,Y)=\frac{1}{2}\tr(XY)$.
    Then the quadratic form
    \begin{equation*}
        X \mapsto -B(X,\theta X)
    \end{equation*}
    is positive definite on $\frakg_\R=\so(p,q)_\R$, where $\theta$ is the Cartan involution defined by $\theta(X)=-\tp{X}$.
    The straightforward calculation shows that the constant $\alpha_{P'|P}$ defined in \cite[\S3]{artjfa} is equal to $2^\frac{k}{2}$, where $k$ is the number of roots of the form $\chi_i$ whose restriction to $\fraka_M$ are roots of both $(P',A_M)$ and $(\overline{P},A_M)$.
    Here, $A_M$ and $\overline{P}$ denote the maximal central split torus in $M$ and the parabolic subgroup opposite to $P$ containing $M$, respectively.

    Next, we calculate the Euclidean measure $dX$ defined by the quadratic form $X\mapsto -B(X, \theta X)$ (\cite[\S3]{artjfa}).
    Put $T=\xi(T^*)$, and write $\frakt$ for its Lie algebra.
    Then we have a decomposition
    \begin{equation*}
        \frakg_\C=\frakt_\C \oplus \bigoplus_{a \in R(T,G)/\Gamma} \frakg_{a,\C},
    \end{equation*}
    where $\Gamma$ is the Galois group of $\overline{F}/F = \C/\R$.
    Note that each $\frakg_{a,\C}=\oplus_{\alpha\in a}\frakg_\alpha$ is defined over $\R$, but $\frakg_\alpha$ is not necessarily defined over $\R$.
    Let $a \in R(T,G)/\Gamma$ and $\alpha\in a$.
    If $\alpha$ is a complex root, then $a=\{\alpha, \sigma\alpha\}$, and $\dim \frakg_a=2$, where $\sigma\in\Gamma$ is the complex conjugation.
    In this case, $\frakg_a$ has an $\R$-basis $\{X_1, X_2\}$, where $X_1=\xi(X_\alpha)+\sigma(\xi(X_\alpha))$ and $X_2=-\sqrt{-1}(\xi(X_\alpha)-\sigma(\xi(X_\alpha)))$.
    By a straightforward calculation, we have $-B(X_1,\theta X_1)=-B(X_2, \theta X_2)=2$.
    If $\frakg_\alpha \subset \overline{\frakn}\cap\frakn'$, the contribution of $a$ to the measure $dX$ is $|d(2^{-\frac{1}{2}}X_1\wedge 2^{-\frac{1}{2}}X_2)|=|d(\xi(X_\alpha\wedge X_{\sigma\alpha}))|$, which is equal to the contribution to our measure.
    If $\alpha$ is a real root, then $a=\{\alpha\}$, and $\dim \frakg_a=1$.
    In this case, $\frakg_a$ has an $\R$-basis $\{X_1\}$, where $X_1=\xi(X_\alpha)$.
    By a straightforward calculation, we have $-B(X_1,\theta X_1)=1$ if $\alpha$ is of the form $\chi_i\pm\chi_j$, and $-B(X_1,\theta X_1)=2$ if $\alpha$ is of the form $\chi_i$.
    If $\frakg_\alpha \subset \overline{\frakn}\cap\frakn'$, the contribution of $a$ to the measure $dX$ is $|dX_1|=|d(\xi(X_\alpha))|$ if $\alpha$ is of the form $\chi_i\pm\chi_j$, and $|d(2^{-\frac{1}{2}}X_1)|=2^{-\frac{1}{2}}|d(\xi(X_1))|$ if $\alpha$ is of the form $\chi_i$.
    The product of coefficients $2^{-\frac{1}{2}}$ for all $\alpha$ of the form $\chi_i$ and $\alpha_{P'|P}=2^\frac{k}{2}$ cancel.
    If $\alpha$ is an imaginary root, then $\frakg_\alpha \nsubseteq \overline{\frakn}\cap\frakn'$.
    This completes the proof.
\end{proof}

\begin{lemma}\label{2.2.3}
    Assume that $F$ is a $p$-adic field and $\psi=\phi\in \Phi(M^*)$ an $L$-parameter.
    Then the function $\lambda \mapsto R_{P'|P}(\xi,\phi_\lambda)$ has neither a zero nor a pole at $\lambda=0$.
    Moreover, let $P''\subset G$ be a parabolic subgroup defined over $F$ with Levi factor $M$.
    Then we have
    \begin{equation*}
        R_{P''|P}(\xi,\phi)=R_{P''|P'}(\xi,\phi) \circ R_{P'|P}(\xi,\phi).
    \end{equation*}
\end{lemma}
\begin{proof}
    Note that the proofs of Lemmas \ref{2.2.2}, \ref{giom}, and \ref{4.2.3} are independent from this lemma.
    In particular, we may use them here.

    Let $\pi\in\Pi_\phi(M)$ be an arbitrary representation of $M$ corresponding to $\phi$.
    The operator $R_{P'|P}(\xi,\phi_\lambda)$ is an intertwining operator from $\calH_P(\pi_\lambda)$ to $\calH_{P'}(\pi_\lambda)$.
    By the reduction steps described in \cite[\S2]{artjfa} (more precisely, from the last line in p.29 of loc. cit.), we may assume that $\pi$ is supercuspidal.
    Let $\dot{F}$, $u$, $v_2$, $\dot{G}^*$, $\dot{G}$, $\dot{\xi}$, $\dot{M}^*$, $\dot{M}$, and $\dot{\pi}$ be as in Lemma \ref{4.2.3}.
    Furthermore, let $\dot{P}$, $\dot{P}'$, and $\dot{P}''$ be the parabolic subgroups of $\dot{G}$ such that $\dot{P}_u=P$, $\dot{P}'_u=P'$, and $\dot{P}''_u=P''$.
    Since we can choose $v_2\neq u$ arbitrarily, we may assume that $v_2$ is a real place.
    By the induction hypothesis, we can take the $A$-parameter of $\dot{\pi}$, write $\dot{\phi}$ for it.
    Note that $\dot{\phi}$ is a generic parameter because $\dot{\phi}_u=\phi$.

    The global intertwining operator $R_{\dot{P}'|\dot{P}}(\dot{\pi}_\lambda,\dot{\phi}_\lambda)$ has neither a zero nor a pole at $\lambda=0$, and is multiplicative in $\dot{P}'$, $\dot{P}$ by Lemma \ref{giom}.
    For any place $v\neq u, v_2$, $\dot{G}_v$ splits over $\dot{F}_v$.
    For $u=v_2$, we have $\dot{F}_{v_2}=\R$.
    Thus for any $v\neq u$, the local normalized intertwining operator $R_{\dot{P}'_v|\dot{P}_v}(\dot{\xi}_v,\dot{\phi}_{\lambda,v})$ has neither a zero nor a pole at $\lambda=0$, and is multiplicative by Lemma \ref{2.2.1} and the assumption from the split case.
    Now the assertion follows from Lemma \ref{2.2.2}.
\end{proof}

The following lemma is now proven for generic parameters.
In general, we expect it can be proven as \cite[Lemma 2.2.4]{kmsw}.
\begin{lemma*}\label{2.2.4}
    Let $F$ be a local field.
    Let $\psi\in\Psi(M^*)$ be a general local $A$-parameter for $M$.
    Then the function $\lambda \mapsto R_{P'|P}(\xi,\psi_\lambda)$ has neither a zero nor a pole at $\lambda=0$.
    Moreover, let $P''\subset G$ be a parabolic subgroup defined over $F$ with Levi factor $M$.
    Then we have
    \begin{equation*}
        R_{P''|P}(\xi,\psi)=R_{P''|P'}(\xi,\psi) \circ R_{P'|P}(\xi,\psi).
    \end{equation*}
\end{lemma*}

\subsection{The second intertwining operator}\label{2.3}
We maintain the assumptions of the previous subsection.
For every finite separable extension $E/F$, Langlands \cite[Theorem 1]{lan70} defined a complex number $\lambda(E/F,\psi_F)$.
By the definition, we have $\lambda(F/F, \psi_F)=1$.

For $w\in W(T^*,G^*)$, Keys-Shahidi \cite[(4.1)]{ks} defined the constant $\lambda(w,\psi_F)$.
In this subsection, we shall construct a local intertwining operator denoted by $\ell^{P'}_P(w,\xi, \psi, \psi_F)$.
In the case of unitary groups, its construction involves the $\lambda$-factor as in \cite[\S2.3]{kmsw}.
However, in the case of odd special orthogonal groups, we do not need the $\lambda$-factor since it always vanishes:
\begin{lemma}
    For any $w\in W(T^*,G^*)$, we have
    \begin{equation*}
        \lambda(w,\psi_F)=1.
    \end{equation*}
\end{lemma}
\begin{proof}
    By the definition of $\lambda(w,\psi_F)$, it suffices to show that
    \begin{equation*}
        \lambda(F_a/F, \psi_F)=1,
    \end{equation*}
    for any root $a \in R(T^*,G^*)$, where $F_a/F$ is a finite extension such that $G^*_{a,\mathrm{sc}}=\Res_{F_a/F}(\SL_2)$.
    Here, $G^*_{a,\mathrm{sc}}$ denotes the simply connected cover of the derived group $G^*_{a,\mathrm{der}}$ of the Levi subgroup $G^*_a$ of semisimple rank 1 attached to $a$.
    Since we know that $\lambda(F/F, \psi_F)=1$, it suffices to show that $F_a=F$ for all root $a \in R(T^*,G^*)$.

    If $a$ is a long root (i.e., of the form $\pm\chi_i\pm\chi_j$), then $G^*_a$ is isomorphic to a product of $\GL_2$ and a finite number of $\GL_1$.
    Thus we have $G^*_{a,\mathrm{der}}\cong\SL_2$ and $G^*_{a,\mathrm{sc}}\cong\SL_2$, which means that $F_a=F$.
    If $a$ is a short root (i.e., of the form $\pm\chi_i$), then $G^*_a$ is isomorphic to a product of $\SO_3$ and a finite number of $\GL_1$.
    Thus we have $G^*_{a,\mathrm{der}}\cong\SO_3\cong\PGL_2$ and $G^*_{a,\mathrm{sc}}\cong\SL_2$, which means that $F_a=F$.
\end{proof}

Let $\psi\in\Psi(M^*)$ be a local parameter such that $\psi$ is $M$-relevant and the corresponding packet $\Pi_\psi(M)$ is not empty.
Let $\pi\in\Pi_\psi(M)$.
Moreover, let $P$ and $P'$ be parabolic subgroups of $G$ defined over $F$ with common Levi factor $M$.
For an element $w\in W(M^*,G^*)\cong W(M,G)$, we write $\widetilde{w}\in N(M^*,G^*)$ for its Langlands-Shelstad lift, and put $\breve{w}=\xi(\widetilde{w})$.
By replacing $\xi$ by an equivalent inner twist if necessary, we assume that $\breve{w}\in G(F)$.
This gives us another representation
\begin{equation*}
    [\breve{w}\pi](m)=\pi(\breve{w}\inv m \breve{w}), \quad m\in M(F),
\end{equation*}
of $M(F)$ on the same vector space as $\pi$, which corresponds to a parameter $w\psi:=\Ad(w)\circ\psi$.

First, we define an unnormalized intertwining operator $\ell^{P'}(\breve{w})$ from $\calH_{P'^w}(\pi)$ to $\calH_{P'}(\breve{w}\pi)$ by
\begin{equation*}
    [\ell^{P'}(\breve{w}) f] (g) = f(\breve{w}\inv g),
\end{equation*}
where $P'^w$ denotes $w\inv P' w$.
Next, our normalizing factor is defined by
\begin{equation*}
    \epsilon_P(w,\psi,\psi_F)=\epsilon(\frac{1}{2}, \rho_{P^w|P}^\vee \circ \phi_\psi, \psi_F),
\end{equation*}
and we define
\begin{equation*}
    \ell^{P'}_P(w,\xi,\psi,\psi_F) = \epsilon_P(w,\psi,\psi_F) \ell^{P'}(\breve{w}).
\end{equation*}

\begin{lemma}\label{2.3.1}
    For any $w_1, w_2 \in W(M^*,G^*) \cong W(M,G)$, we have
    \begin{equation*}
        \ell^P_P(w_2w_1,\xi, \psi, \psi_F) = \ell^P_P(w_2,\xi, w_1\psi, \psi_F) \circ \ell^{P^{w_2}}_P(w_1,\xi, \psi, \psi_F).
    \end{equation*}
\end{lemma}
\begin{proof}
    The proof is similar to that of \cite[Lemma 2.3.4]{art13} or \cite[Lemma 2.3.1]{kmsw}.
    Our case is simpler since the $\lambda$-factor is trivial and the Galois action on the dual group $\widehat{G}$ is also trivial.
\end{proof}

\begin{lemma}\label{2.2.5}
    We have
    \begin{equation*}
        \ell^{P'}_P(w, \xi, \psi, \psi_F) \circ R_{P'^w|P^w}(\xi,\psi)
        = R_{P'|P}(\xi,w\psi) \circ \ell^P_P(w, \xi, \psi, \psi_F).
    \end{equation*}
\end{lemma}
\begin{proof}
    The assertion follows from two equalities
    $\ell^{P'}_P(w, \xi, \psi, \psi_F) \circ R_{P'^w|P^w}(\xi,\psi)
        = R_{P'|P}(\Ad(\breve{w})\circ\xi,\psi) \circ \ell^P_P(w, \xi, \psi, \psi_F)$
    and
    $R_{P'|P}(\Ad(\breve{w})\circ\xi,\psi)=R_{P'|P}(\xi,w\psi)$.
    We can show them by a similar way to the second and the third assertions in \cite[Lemma 2.2.5]{kmsw}.
\end{proof}

\subsection{The third intertwining operator}\label{2.4}
We maintain the assumptions of the previous subsection.
In addition, we assume that we are given $z\in Z^1(F, M^*) \subset Z^1(F, G^*)$, such that $(\xi, z)$ is a pure inner twist; $z$ commutes with the Langlands-Shelstad lift of every element of $W(M^*,G^*)$; and if we decompose $z$ as $z=z_+\times z_-$ according to the decomposition $M^*=M^*_+ \times M^*_-$ then $z_+$ takes the constant value 1, where $M^*_+$ is a product of general linear groups and $M^*_-$ is a special orthogonal group.
Let $u\in N(\widehat{T},\widehat{G})$ be such that $\Ad(u)\circ\psi=\psi$ and $u$ preserves the positive roots in $\widehat{M}$.
Let $u^\natural\in \frakN_\psi(M,G)$ and $w\in W(\widehat{M},\widehat{G})$ be its images.
We shall regard $w$ as an element of $W(M^*,G^*)$ or $W(M,G)$ via the natural isomorphisms.
Although our group $G^*$ is a special orthogonal group and $(\xi,z)$ is a pure inner twist, by the similar argument as \cite[\S2.4, \S2.4.1]{kmsw}, we can define
\begin{itemize}
    \item the operator $\pi(\breve{w})_\xi : (\breve{w}\pi, V_\pi) \to (\pi, V_\pi)$;
    \item the sophisticated splittings $s' : W_\psi(M,G)\to \frakN_\psi(M,G)$ and $s : \frakN_\psi(M,G)\to \frakS_\psi(M)$ of the exact sequence $1\to \frakS_\psi(M)\to \frakN_\psi(M,G)\to W_\psi(M,G)\to 1$;
    \item the constant $\an{u^\natural,\pi}_{\xi,z}=\an{s(u^\natural)\inv,\pi}_{\xi|_{M^*},z} \in \C^\times$.
\end{itemize}
Then we define the operator $\pi(u^\natural)_{\xi,z} : (\breve{w}\pi, V_\pi) \to (\pi, V_\pi)$ by $\pi(u^\natural)_{\xi,z}=\an{u^\natural,\pi}_{\xi,z} \pi(\breve{w})_\xi$.
The assignment $u^\natural \mapsto \pi(u^\natural)_{\xi,z}$ is multiplicative.

Next we shall give a more abstract characterization (i.e., another equivalent definition) of the operator $\pi(u^\natural)_{\xi,z} : (\breve{w}\pi, V_\pi) \to (\pi, V_\pi)$, following \cite[\S2.4.3]{kmsw}.
Let us put $u'=u\inv \in N_\psi(M,G)$, and let $\widetilde{u'}\in N(\widehat{T}, \widehat{G})$ be its Langlands-Shelstad lift.
Put $s=u'\widetilde{u'}\inv\in\widehat{M}$.
Then by the assumption that $u$ preserves the positive roots in $\widehat{M}$, we have $s\in\widehat{T}$.
Put $\widehat{\theta}=\Ad(\widetilde{u'})$, which is an automorphism of $\widehat{M}$ preserving the standard splitting inherited from $\widehat{G}$.
It is the dual of an automorphism $\theta^*=\Ad(\widetilde{w})$ of $M^*$.
Put
\begin{equation*}
    \widehat{M'}=\Set{x\in\widehat{M} | \Ad(s_\psi s)\circ \widehat{\theta} (x)=x}.
\end{equation*}
Let $M'$ be the subgroup of $M$ such that $\widehat{M'}$ is the dual of $M'$.
Note that $\myim \psi \subset \widehat{M'}$.
Then $(M', s_\psi s, \widehat{M'} \subset \widehat{M})$ is a twisted endoscopic triple of $(M^*,\theta^*)$, and hence of $(M,\theta)$, where $\theta=\Ad(\breve{w})$.
We shall write $\frake_{M,\psi}$ for it.
The operator $\pi(u^\natural)_{\xi,z}$ can be characterized as follows:
\begin{lemma}\label{2.4.1}
    For each $\pi \in \Pi_\psi(M)$, there exists a unique isomorphism $\pi(u^\natural)_{\xi,z} : (\breve{w}\pi, V_\pi) \to (\pi, V_\pi)$ such that
    \begin{equation*}
        f^{\frake_{M,\psi}}(\psi)=e(M^\theta)\sum_{\pi\in\Pi_\psi(M)}\tr(\pi(u^\natural)_{\xi,z}\circ \pi(f)),
    \end{equation*}
    for all $\Delta[\frake_{M,\psi},\xi|_{M^*},z]$-matching functions $f \in \calH(M)$ and $f^{\frake_{M,\psi}}\in\calH(M')$.
    Moreover, the operator $\pi(u^\natural)_{\xi,z}$ coincides with the one constructed above.
\end{lemma}
\begin{proof}
    The proof is similar to that of \cite[Lemma 2.4.1]{kmsw} for the case of the quasi-split unitary group $U_{E/F}(N)$.
\end{proof}

\subsection{The compound operator}\label{2.5}
We shall define the normalized self-intertwining operator $R_P^G(u^\natural, \pi, \psi, \psi_F)=R_P(u^\natural, \pi, \psi, \psi_F) \in \End_G(\calH_P(\pi))$ as a composition
\begin{equation*}
    R_P(u^\natural, \pi, \psi, \psi_F)
    =\calI_P(\pi(u^\natural)_{\xi,z}) \circ \ell^P_P(w, \xi, \psi, \psi_F) \circ R_{P^w|P}(\xi,\psi).
\end{equation*}
Recall that the definitions of the Haar measures on the unipotent radicals and the representative $\breve{w}$ which were used to define $R_{P^w|P}(\xi,\psi)$ and $\ell^P_P(w, \xi, \psi, \psi_F)$, respectively, depend on $\xi$.
Note also that the definition of $\calI_P(\pi(u^\natural)_{\xi,z})$ depends on $(\xi,z)$.
If the parameter $\psi$ is not generic, we shall assume from now on that Lemma \ref{2.2.4} holds true.
\begin{lemma}\label{2.5.1}
    The operator $R_P(u^\natural, \pi, \psi, \psi_F)$ does not depend on the choice of the pure inner twist $(\xi,z)$.
\end{lemma}
\begin{proof}
    Suppose that $(\xi,z)\simeq(\xi',z')$ as inner twists of $G^*=\SO_{2n+1}$ over $F$.
    Since every automorphism is inner, there exists $b\in G^*(\overline{F})$ such that $\xi'=\xi\circ \Ad(b)$.
    The proof is similar to that of \cite[Lemma 2.5.1]{kmsw} except that our $z$ is not a basic cocycle, but an ordinary 1-cocycle.
\end{proof}

The Kottwitz map \eqref{kot} provides a pairing $\an{-,-}$ on $Z(\widehat{G}) \times H^1(F,G^*)$.
Note that in our case $Z(\widehat{G})\simeq \{\pm1\}$.
\begin{lemma}\label{2.5.2}
    \begin{enumerate}
        \item Let $x\in Z(\widehat{G})$. Then $R_P(xu^\natural, \pi, \psi, \psi_F)=\an{x,z}\inv R_P(u^\natural, \pi, \psi, \psi_F)$.
        \item Let $y\in \frakS_\psi(M)$. Then $R_P(yu^\natural, \pi, \psi, \psi_F)=\an{y,\pi}_{\xi|_{M^*},z}R_P(u^\natural, \pi, \psi, \psi_F)$.
    \end{enumerate}
\end{lemma}
\begin{proof}
    The proof is similar to that of \cite[Lemma 2.5.2]{kmsw}.
\end{proof}

\begin{lemma}\label{2.5.3}
    The operator $R_P(u^\natural, \pi, \psi, \psi_F)$ is multiplicative in $u^\natural$.
\end{lemma}
\begin{proof}
    Let $u_1^\natural$ and  $u_2^\natural$ be elements in $\frakN_\psi(M,G)$, and $w_1$ and $w_2$ their images in $W_\psi(M,G)$ respectively.
    Since the assignment $u^\natural \mapsto \pi(u^\natural)_{\xi,z}$ is multiplicative, we have
    \begin{equation*}
        \calI_P(\pi(u_2^\natural u_1^\natural)_{\xi,z})
        =\calI_P(\pi(u_2^\natural)_{\xi,z} \circ \pi(u_1^\natural)_{\xi,z})
        =\calI_P(\pi(u_2^\natural)_{\xi,z}) \circ \calI_P(\pi(u_1^\natural)_{\xi,z}).
    \end{equation*}
    Combining this with Lemmas \ref{2.2.3}, \ref{2.3.1}, and \ref{2.2.5}, we get
    \begin{align*}
         & R_P(u_2^\natural u_1^\natural, \pi, \psi, \psi_F)                            \\
         & =\calI_P(\pi(u_2^\natural)_{\xi,z}) \circ \calI_P(\pi(u_1^\natural)_{\xi,z})
        \circ \ell^P_P(w_2, \xi, w_1\psi, \psi_F) \circ \ell^{P^{w_2}}_P(w_1, \xi, \psi, \psi_F)
        \circ R_{P^{w_2w_1}|P^{w_1}}(\xi,\psi) \circ R_{P^{w_1}|P}(\xi,\psi)            \\
         & =\calI_P(\pi(u_2^\natural)_{\xi,z}) \circ \calI_P(\pi(u_1^\natural)_{\xi,z})
        \circ \ell^P_P(w_2, \xi, \psi, \psi_F) \circ R_{P^{w_2}|P}(\xi,\psi)
        \circ \ell^P_P(w_1, \xi, \psi, \psi_F) \circ R_{P^{w_1}|P}(\xi,\psi).
    \end{align*}
    We know that the operator $\calI_P(\pi(u_1^\natural)_{\xi,z})$ commutes with $\ell^P_P(w_2, \xi, w_1\psi, \psi_F) \circ R_{P^{w_2}|P}(\xi,\psi)$, as the operator $\calI_P(\pi(u_1^\natural)_{\xi,z})$ acts on the values of the functions that comprise $\calH_P(\pi)$, while the operators $\ell^P_P(w_2, \xi, \psi, \psi_F)$ and $R_{P^{w_2}|P}(\xi,\psi)$ act on their variables.
    Thus we have
    \begin{align*}
         & \calI_P(\pi(u_2^\natural)_{\xi,z}) \circ \calI_P(\pi(u_1^\natural)_{\xi,z})
        \circ \ell^P_P(w_2, \xi, \psi, \psi_F) \circ R_{P^{w_2}|P}(\xi,\psi)
        \circ \ell^P_P(w_1, \xi, \psi, \psi_F) \circ R_{P^{w_1}|P}(\xi,\psi)                                         \\
         & =\calI_P(\pi(u_2^\natural)_{\xi,z})  \circ \ell^P_P(w_2, \xi, \psi, \psi_F) \circ R_{P^{w_2}|P}(\xi,\psi)
        \circ \calI_P(\pi(u_1^\natural)_{\xi,z})
        \circ \ell^P_P(w_1, \xi, \psi, \psi_F) \circ R_{P^{w_1}|P}(\xi,\psi)                                         \\
         & =R_P(u_2^\natural, \pi, \psi, \psi_F) \circ R_P(u_1^\natural, \pi, \psi, \psi_F).
    \end{align*}
\end{proof}

\subsection{The two linear forms, the local intertwining relation, and the construction of the non-discrete packets}\label{2.6}
Let $F$ be a local field, $\psi_F:F\to\C^1$ a nontrivial additive character, $G^*=\SO_{2n+1}$, and $M^*\subsetneq G^*$ a proper standard Levi subgroup.
Let $\xi:G^*\to G$ be an inner twist over $F$, and put $M=\xi(M^*)$.
Let $\psi\in\Psi(M^*)$ be a local parameter, and $u^\natural$ an element of $\frakN_\psi(M,G)$.
When $M^*$ transfers to $G$, we assume that $M\subsetneq G$ is a proper Levi subgroup over $F$ and $\xi|_{M^*}: M^* \to M$ is an inner twist over $F$.

We shall define the first linear form $f \mapsto f_G(\psi,u^\natural)$ on $\calH(G)$ as follows:
\begin{align*}
    f_G(\psi,u^\natural)=
    \begin{dcases*}
        \sum_{\pi\in\Pi_\psi(M)} \tr(R_P(u^\natural,\pi,\psi,\psi_F) \circ \calI_P(\pi, f)), & if $\psi$ is relevant,     \\
        0,                                                                                   & if $\psi$ is not relevant.
    \end{dcases*}
\end{align*}
\noindent
In this paper, following \cite{kmsw}, we shall sometimes write $\calI_P(\pi,f)$ for $\calI_P(\pi)(f)$, which is the action of $f$ on $(\calI_P(\pi), \calH_P(\pi))$.

Next we shall give the definition of the second linear form $f \mapsto f'_G(\psi,s)$ on $\calH(G)$.
For $f \in \calH(G)$ and $s \in S_{\psi,\semisimple}$, let $(\frake,\psi^\frake)\in X(G^*)$ be the element corresponding to $(s,\psi)\in Y(G^*)$ under the bijective map given in Lemma \ref{1.4.1}.
Then we write $f'_G(\psi,s)$ for the value $f^\frake(\psi^\frake)$ in the endoscopic character relation \eqref{ecr}.
It does not depend on the choice of $\xi$.
\begin{lemma}\label{2.6.1}
    Let $s_1,s_2\in S_{\psi,\semisimple}$ be two elements such that they have the same image in $\frakS_\psi(G)$ and the image belongs to $\frakS_\psi(M,G)$.
    Then $f'_G(\psi,s_1)=f'_G(\psi,s)$.
\end{lemma}
\begin{proof}
    The proof is similar to that of \cite[Lemma 2.6.1]{kmsw}.
\end{proof}
\begin{theorem*}[LIR]\label{2.6.2}
    Let $u\in N_\psi(M,G)$ and $f\in \calH(G)$.
    Then
    \begin{enumerate}
        \item If $u^\natural$ is in $W_\psi^\circ(M,G)$, then we have $R_P(u^\natural,\pi,\psi,\psi_F)=1$ for all $\pi\in\Pi_\psi(M)$.
        \item The value $f_G(\psi,u^\natural)$ depends on the image of $u^\natural$ in $\frakS_\psi(M,G)$.
        \item We have
              \begin{equation}\label{lir}
                  f'_G(\psi,s_\psi u\inv) = e(G)f_G(\psi,u^\natural),
              \end{equation}
              where $e(G)$ denotes the Kottwitz sign of $G$.
              This equation is called the local intertwining relation.
    \end{enumerate}
\end{theorem*}
\begin{lemma}\label{2.6.5}
    For any $y\in Z(\widehat{G}^*)$, we have $f'_G(\psi,s_\psi ys)=\an{y,z}f'_G(\psi,s_\psi s)$ and $f_G(\psi,yu^\natural)=\an{y,z}\inv f_G(\psi,u^\natural)$.
\end{lemma}
\begin{proof}
    They follow from Lemmas \ref{1.1.2} and \ref{2.5.2}.
\end{proof}

In this paper, Theorem \ref{2.6.2} will be completely proved in the case when $\psi=\phi\in\Phi(M^*)$ is an $L$-parameter.
The proof is inductive, and thus we assume that Theorem \ref{2.6.2} holds for the case the degree is smaller than $n$, from now on.
The case of a general parameter $\psi\in\Psi(M^*)$ is expected to be treated as a sequel of \cite{kmsw}.

We will now review the construction of local non-discrete $A$-packets which is an application of the local intertwining relation (Theorem \ref{2.6.2}).
Suppose that the theorem is true.
Assume that $\psi_{M^*}\in\Psi_2(M^*)$ and $M^*\subsetneq G^*$, hence $\psi_{M^*}$ is a discrete parameter for $M^*$ and not discrete for $G^*$.
We shall write $\psi$ for the image of $\psi_{M^*}$ under the natural map $\Psi(M^*)\to\Psi(G^*)$.
We apologize for this change of convention.
If $\psi$ is not $G$-relevant, then we simply set $\Pi_\psi(G)=\emptyset$.

Suppose that $\psi$ is $G$-relevant.
Then we may assume that a Levi subgroup $M=\xi(M^*)\subsetneq G$ is defined over $F$, and by the assumption Theorem \ref{2.6.2} holds true for $\psi_{M^*}$ and $M$.
According to Lemma \ref{2.8.3}, we have $\frakS_{\psi_{M^*}}(M,G)=\frakS_\psi(G)$, hence the natural map $N_{\psi_{M^*}}(M,G) \to \frakS_\psi(G)$ is surjective.
For any $\pi_M\in \Pi_{\psi_{M^*}}(M)$, where note that the packet for $M$ is given by the induction hypothesis, the map
\begin{align*}
     & \frakN_{\psi_{M^*}}(M,G) \times G(F) \to \Aut_{G(F)}(\calH_P(\pi_M)),                      &
     & (u^\natural, g) \mapsto R_P(u^\natural,\pi_M, {\psi_{M^*}},\psi_F) \circ \calI_P(\pi_M,g), &
\end{align*}
is a representation of $\frakN_{\psi_{M^*}}(M,G) \times G(F)$.
Let $\rho(\pi_M)$ denote this representation.
Put
\begin{equation*}
    \Pi_\psi^1(G) = \bigoplus_{\pi_M\in \Pi_{\psi_{M^*}}(M)} \rho(\pi_M).
\end{equation*}
Then we have
\begin{align*}
    \tr\left(\Pi_\psi^1(G)(u^\natural, f)\right)
     & = \sum_{\pi_M\in \Pi_{\psi_{M^*}}(M)} \tr\left(\rho(\pi_M)(u^\natural, f)\right) \\
     & =f_G({\psi_{M^*}},u^\natural),
\end{align*}
for $u^\natural\in \frakN_{\psi_{M^*}}(M,G)$ and $f\in \calH(G)$.
Thus, Theorem \ref{2.6.2} (and the character theory of representations of finite groups) implies that $\Pi_\psi^1(G)$ can be regarded as a representation of $\frakS_\psi(G)\times G(F)$.
By Lemma \ref{2.5.2}, we have $\Pi_\psi^1(G)|_{Z(\widehat{G})\times\{1\}}=\chi_G\inv$.
As $\pi_M$ is unitary, $\Pi_\psi^1(G)$ is also unitary, in particular completely reducible.
We define the packet $\Pi_\psi(G)$ to be the multiset of irreducible representations of $G(F)$ that appear in the direct summand of $\Pi_\psi^1(G)$.
Then we obtain
\begin{align*}
    \Pi_\psi^1(G)=\bigoplus_{\pi\in \Pi_\psi(G)} \an{-,\pi}\inv \otimes \pi,
\end{align*}
where $\an{-,\pi}$ are characters of $\frakS_\psi(G)$ whose restriction to $Z(\widehat{G})^\Gamma$ coincide $\chi_G$.
We equip the packet $\Pi_\psi(G)$ with the map
\begin{align*}
     & \Pi_\psi(G) \to \Irr(\frakS_\psi, \chi_G), &
     & \pi\mapsto \an{-,\pi}.                     &
\end{align*}
The endoscopic character relation \eqref{ecr} follows directly from the local intertwining relation \eqref{lir} and the definition of the pairing $\an{-,\pi}$ above.

\subsection{Reduction of LIR to discrete (relative to $M$) parameters}\label{2.7}
In this subsection we review from \cite[\S2.7]{kmsw} the twisted local intertwining relation and the reduction of the proof of the part 1 and 2 of Theorem \ref{2.6.2} to the case of discrete parameters.
Concretely speaking, this section will be devoted to the proofs of Lemmas \ref{2.7.1} and \ref{2.7.2} below.
Let us now consider two nested proper standard Levi subgroups $M_0^*\subset M^*\subset G^*$.
Let $\psi_0\in\Psi(M_0^*)$ be a local parameter, and $\psi\in \Psi(M^*)$ its image under the natural map from $\Psi(M_0^*)$ to $\Psi(M^*)$.
Note that $\psi_0$ is relevant if and only if $\psi$ is relevant.
When $\psi$ and $\psi_0$ are relevant, we assume that $M=\xi(M^*)$ and $M_0=\xi(M_0^*)$ are proper Levi subgroups of $G$ defined over $F$.

\begin{lemma}\label{2.7.1}
    For any $u \in N_{\psi_0}(M_0,G) \cap N_\psi(M,G)$, we have
    \begin{equation*}
        f_G(\psi_0, u^\natural) = f_G(\psi,u^\natural),
    \end{equation*}
    where $u^\natural$ in the left (resp. right) hand side is the image of $u$ in $\frakN_\psi(M_0,G)$ (resp. $\frakN_\psi(M,G)$).
\end{lemma}
Note that the lemma is trivial if the parameter $\psi_0$ is not relevant.

Let $P$ (resp. $P_0$, resp. $Q$) be the standard parabolic subgroup of $G$ (resp. $G$, resp. $M$), with the Levi subgroup $M$ (resp. $M_0$, resp. $M_0$).
Note that $N_\psi(M_0,M) \subset N_\psi(M_0,G)$.
Note also that the equation \eqref{shake} implies $Z(A_{\widehat{M_0}}, S_{\psi_0}(G)^\circ)=Z(A_{\widehat{M_0}}, S_{\psi_0}(M)^\circ)$, and hence $\frakN_\psi(M_0,M) \subset \frakN_\psi(M_0,G)$.
\begin{lemma}\label{2.7.2}
    Assume that $\psi_0$ is relevant.
    For any $u \in N_{\psi_0}(M_0,M)$ and $\pi_0 \in \Pi_{\psi_0}(M_0)$, we have
    \begin{equation*}
        R_{P_0}^G(u^\natural, \pi_0, \psi_0, \psi_F)=\calI_P^G ( R_Q^M(u^\natural, \pi_0, \psi_0, \psi_F) ).
    \end{equation*}
    Moreover, we have
    \begin{equation*}
        f_G(\psi_0, u^\natural) = f_M(\psi_0,u^\natural).
    \end{equation*}
\end{lemma}

Before the proofs of these lemmas, we shall now record a consequence of Lemma \ref{2.7.1}, which is a key ingredient of the proof of the local intertwining relation.
\begin{proposition}\label{2.7.3}
    Assume that the second and the third statements of Theorem \ref{2.6.2} hold for a standard parabolic pair $(M_0^*, P_0^*)$ and a discrete parameter $\psi_0\in \Psi_2(M_0^*)$.
    Then they hold for every standard parabolic pair $(M^*,P^*)$ and a parameter $\psi\in\Psi(M^*)$ such that $M_0^*\subset M^*$, where $\psi$ is the image of $\psi_0$ under the canonical map $\Psi(M_0^*) \to \Psi(M^*)$.
\end{proposition}
\begin{proof}
    The proof is similar to that of \cite[Proposition 2.7.3]{kmsw} with $S_\psi(G)^\circ$ in place of $S_\psi^{\mathrm{rad}}$.
    (Precisely speaking, we have $S_\psi^{\mathrm{rad}}=S_\psi(G)^\circ$ because $G^*=\SO_{2n+1}$.)
\end{proof}
The first statement of Theorem \ref{2.6.2} will also be reduced to the case of discrete parameters, in \S\ref{4.6} Lemma \ref{4.6.4}.

As a preparation of the proofs of Lemmas \ref{2.7.1} and \ref{2.7.2}, we now review the twisted local intertwining relation.
See \cite[\S2.7]{kmsw} and \cite[pp.115-119]{art13} for detail.
Let $N$ and $k$ be positive integers.
Put $H=\GL_N^k$.
The standard pinning of $\GL_N$ gives rise to a pinning of $H$, which we write $(T_H, B_H, \{X_{\alpha^H}\}_{\alpha^H})$ and call standard.
Let $\theta_H$ be the automorphism of $H$ given by $\theta_H(h_1,\ldots,h_k)=(\theta(h_k),h_1,\ldots,h_{k-1})$, where $\theta$ is either the identity automorphism of $\GL_N$, or the automorphism $\theta_N$.
Let $(M_H,P_H)$ be a standard parabolic pair of $H$.
Recall that any packet for a direct product of a finite number of general linear groups, is a singleton.
Let $\psi \in\Psi(M_H)$ and let $\pi$ be the corresponding representation of $M_H(F)$.
Put $S_\psi(H,\theta_H)=\cent(\myim\psi, \widehat{H}\rtimes \widehat{\theta_H})$, which is not a group.
Put also $N_\psi(M_H, H\rtimes\theta_H)=N(A_{\widehat{M_H}}, S_\psi(H,\theta_H\inv))$, $W(\widehat{M_H},\widehat{H}\rtimes\widehat{\theta_H}\inv)=N(A_{\widehat{M_H}},\widehat{H}\rtimes\widehat{\theta_H}\inv)/\widehat{M_H}$, and $W(M_H,H\rtimes \theta_H)=N(A_{M_H},H\rtimes \theta_H)/M_H$.
We identify the Weyl sets $W(\widehat{M_H},\widehat{H}\rtimes\widehat{\theta_H}\inv)$ and $W(M_H,H\rtimes \theta_H)$ in the standard way.
Let $u\in N_\psi(M_H, H\rtimes\theta_H)$ and let $w\in W(\widehat{M_H},\widehat{H}\rtimes\widehat{\theta_H}\inv) \simeq W(M_H,H\rtimes \theta_H)$ be the image of $u$.

As in \S\ref{2.2}, let us define the first operator
\begin{equation}\label{tfo}
    R_{P_H^w|P_H}(\pi_\lambda) = r_{P_H^w|P_H}(\pi_\lambda,\psi_F)\inv J_{P_H^w|P_H}(\pi_\lambda, \psi_F),
\end{equation}
where $J_{P_H^w|P_H}(\pi_\lambda, \psi_F) : \calH_{P_H}(\pi_\lambda) \to \calH_{P_H^w}(\pi_\lambda)$ is the unnormalized intertwining operator defined by
\begin{equation*}
    [J_{P_H^w|P_H}(\pi_\lambda,\psi_F)f](h) = \int_{N_H(F)\cap N_H^w(F) \backslash N_H^w(F)} f(nh) dn,
\end{equation*}
and $r_{P_H^w|P_H}(\pi_\lambda,\psi_F)$ is the normalizing factor given by
\begin{equation*}
    r_{P_H^w|P_H}(\pi_\lambda, \psi_F)
    =\frac{L(0,\rho_{P_H^w|P_H}^\vee\circ\phi_{\psi_\lambda})}{L(1,\rho_{P_H^w|P_H}^\vee\circ\phi_{\psi_\lambda})} \frac{\epsilon(\frac{1}{2},\rho_{P_H^w|P_H}^\vee\circ\phi_{\psi_\lambda}, \psi_F)}{\epsilon(0,\rho_{P_H^w|P_H}^\vee\circ\phi_{\psi_\lambda}, \psi_F)}.
\end{equation*}
Then \eqref{tfo} is holomorphic at $\lambda=0$, thus we put $R_{P_H^w|P_H}(\pi)=R_{P_H^w|P_H}(\pi_0)$.

In order to define the second operator, we have to fix a representative of $w$.
Let $\dot{w}$ be the representative of $w$ in the Weyl set $W(T_H,H\rtimes\theta_H)$ that stabilizes the simple positive roots inside $M_H$.
Since both $\dot{w}$ and $\theta_H$ preserve $T_H$, we have $\dot{w}=\dot{w}_0\rtimes \theta_H$ for some $\dot{w}_0\in W(T_H,H)$.
Then we can take the Langlands-Shelstad lift $\widetilde{w}_0\in N(T_H,H)$ of $\dot{w}_0$.
Put $\widetilde{w}=\widetilde{w}_0 \rtimes \theta_H$, which is the representative of $w$ we need.
As in \S\ref{2.3}, let us define the second operator
\begin{equation*}
    \ell^{P_H}_{P_H}(w,\pi,\psi_F) = \epsilon_{P_H}(w,\psi,\psi_F) \ell^{P^H}(\widetilde{w}),
\end{equation*}
where the normalizing factor is
\begin{equation*}
    \epsilon_{P_H}(w,\psi,\psi_F) = \epsilon(\frac{1}{2}, \rho_{P_H^w|P_H}^\vee \circ \phi_\psi, \psi_F),
\end{equation*}
and the operator $\ell^{P^H}(\widetilde{w})$ from $(\calI_{P_H^w}(\pi), \calH_{P_H^w}(\pi) )$ to $(\calI_{P_H}(\widetilde{w}\pi) \circ \theta_H, \calH_{P_H}(\widetilde{w}\pi) )$ is defined by
\begin{equation*}
    [\ell^{P_H}(\widetilde{w}) f] (h) = f(\widetilde{w}\inv \cdot h\rtimes \theta_H).
\end{equation*}
Here the identity component $\widetilde{H}^\circ=H\rtimes 1$ of the twisted group $\widetilde{H}=H\rtimes \an{\theta_H}$ is identified with $H$.

Note that $\widetilde{w}\pi$ is isomorphic to $\pi$.
Let us define the third operator
\begin{equation*}
    \calI_{P_H} (\pi(\widetilde{w})) :  \calH_{P_H}(\widetilde{w}\pi) \to \calH_{P_H}(\pi),
\end{equation*}
where $\pi(\widetilde{w}) : \widetilde{w}\pi\to \pi$ is the Whittaker normalized isomorphism.
Here we use the Whittaker datum corresponding to $\psi_F$ and the standard pinning.

Now we obtain the normalized intertwining operator
\begin{equation*}
    R_{P_H}(w,\psi,\psi_F)=\calI_{P_H} (\pi(\widetilde{w})) \circ \ell^{P_H}_{P_H}(w,\pi,\psi_F) \circ R_{P_H^w|P_H}(\pi),
\end{equation*}
from $(\calI_{P_H}(\pi), \calH_{P_H}(\pi) )$ to $(\calI_{P_H}(\pi) \circ \theta_H, \calH_{P_H}(\pi) )$, and the twisted first linear form
\begin{equation*}
    \calH(H)\ni f \mapsto f_H(\psi,w):=\tr\left( R_{P_H}(w,\psi,\psi_F) \circ \calI_{P_H}(\pi,f) \right).
\end{equation*}

On the other hand, for any $s\in S_\psi(H,\theta_H)$, we have a pair $(\frake, \psi^\frake)$ of a twisted endoscopic triple $\frake\in \calE(H\rtimes \theta_H)$ and its parameter $\psi^\frake\in \Psi(H^\frake)$, corresponding to the pair $(\psi,s)$.
For any $f \in \calH(H)$, choose $f^\frake \in \calH(H^\frake)$ such that $f$ and $f^\frake$ are matching.
Then we write $f'_H(\psi,s)$ for the value $f^\frake(\psi^\frake)$.
This is the twisted second linear form, which we need.

We call the following formula the twisted local intertwining relation, or twisted LIR for short.
\begin{proposition}[twisted LIR]\label{2.7.4}
    We have
    \begin{equation*}
        f_H(\psi,w) = f'_H(\psi,s_\psi u\inv).
    \end{equation*}
\end{proposition}
\begin{proof}
    The case $k=1$ is guaranteed by \cite[Corollary 2.5.4]{art13}.
    Then the proof is similar to that of \cite[Proposition 2.7.4, Lemma 2.7.6]{kmsw}.
\end{proof}

Now we return to the notation used in Lemmas \ref{2.7.1} and \ref{2.7.2}.
If $\psi_0$ and $\psi$ are not relevant, then Lemma \ref{2.7.1} is trivial.
Assume that $\psi_0$ and $\psi$ are relevant.
Hence $M_0$, $M$, $P_0$, $P$, and $Q$ are defined over $F$.
Let $u \in N_{\psi_0}(M_0,G) \cap N_\psi(M,G)$, and $u^\natural \in \frakN_{\psi_0}(M_0,G)$ be its image.
We will also write $u^\natural$ for the image in $\frakN_\psi(M,G)$, by abuse of notation.
Let $f\in \calH(G)$.
By definition we have
\begin{align*}
    f_G(\psi_0, u^\natural)
     & =\tr\left( \bigoplus_{\pi_0\in\Pi_{\psi_0}(M_0)} R_{P_0}^G(u^\natural,\pi_0,\psi_0,\psi_F) \circ \calI_{P_0}^G(\pi_0, f)\right).
\end{align*}
The calculation similar to that in \cite[pp.123-126]{kmsw} shows that
\begin{align*}
     & \bigoplus_{\pi_0\in\Pi_{\psi_0}(M_0)} R_{P_0}^G(u^\natural,\pi_0,\psi_0,\psi_F) \circ \calI_{P_0}^G(\pi_0)(f) \\
     & =\ell^P_P(w, \xi, \psi, \psi_F) \circ R_{P^w|P}(\xi,\psi) \circ \left(
    \bigoplus_{\pi_0\in\Pi_{\psi_0}(M_0)} \calI_P^G (R_Q^M(u^\natural,\pi_0,\psi_0,\psi_F)) \circ \calI_P^G(\calI_Q^M(\pi_0))(f)
    \right).
\end{align*}
ECR and LIR for inner form of $\SO_{2n_0+1}$ for $n_0<n$, which is assumed by induction, and Lemma \ref{2.7.4} imply ECR and (twisted) LIR for $M$.
Combined with the theory of induced characters (\cite{dijk}) and Lemma \ref{2.4.1}, we have
\begin{align*}
     & \quad \tr\left(
    \bigoplus_{\pi_0\in\Pi_{\psi_0}(M_0)} \calI_P^G (R_Q^M(u^\natural,\pi_0,\psi_0,\psi_F)) \circ \calI_P^G(\calI_Q^M(\pi_0))(f)
    \right)            \\
     & =\tr\left(
    \bigoplus_{\pi_0\in\Pi_{\psi_0}(M_0)} R_Q^M(u^\natural,\pi_0,\psi_0,\psi_F) \circ \calI_Q^M(\pi_0)(f_M)
    \right)            \\
     & =\tr\left(
    \bigoplus_{\pi\in\Pi_\psi(M)} \pi(u^\natural) \circ \pi(f_M)
    \right)            \\
     & =\tr\left(
    \bigoplus_{\pi\in\Pi_\psi(M)} \calI_P^G(\pi(u^\natural)) \circ \calI_P^G(\pi)(f)
    \right),
\end{align*}
where $f_M$ denotes the constant term of $f$ along $P$.
By the linear independence of characters, this means that
\begin{align*}
    \bigoplus_{\pi_0\in\Pi_{\psi_0}(M_0)} \calI_P^G (R_Q^M(u^\natural,\pi_0,\psi_0,\psi_F)) \circ \calI_P^G(\calI_Q^M(\pi_0))(f)
    =\bigoplus_{\pi\in\Pi_\psi(M)} \calI_P^G(\pi(u^\natural)) \circ \calI_P^G(\pi)(f).
\end{align*}
Therefore, we have
\begin{align*}
     & \quad \ell^P_P(w, \xi, \psi, \psi_F) \circ R_{P^w|P}(\xi,\psi) \circ \left(
    \bigoplus_{\pi_0\in\Pi_{\psi_0}(M_0)} \calI_P^G (R_Q^M(u^\natural,\pi_0,\psi_0,\psi_F)) \circ \calI_P^G(\calI_Q^M(\pi_0))(f)
    \right)                                                                                                           \\
     & =\ell^P_P(w, \xi, \psi, \psi_F) \circ R_{P^w|P}(\xi,\psi) \circ \left(
    \bigoplus_{\pi\in\Pi_\psi(M)} \calI_P^G(\pi(u^\natural)) \circ \calI_P^G(\pi)(f)
    \right)                                                                                                           \\
     & =\bigoplus_{\pi\in\Pi_\psi(M)}
    \calI_P^G(\pi(u^\natural)) \circ \ell^P_P(w, \xi, \psi, \psi_F) \circ R_{P^w|P}(\xi,\psi) \circ \calI_P^G(\pi)(f) \\
     & =\bigoplus_{\pi\in\Pi_\psi(M)} R_P^G(u^\natural,\pi,\psi,\psi_F) \circ \calI_P^G(\pi,f),
\end{align*}
and hence
\begin{align*}
    f_G(\psi_0, u^\natural)
    =f_G(\psi, u^\natural).
\end{align*}
Thus Lemma \ref{2.7.1} follows.

Next, let $u \in N_{\psi_0}(M_0,M)$.
Then $u$ is trivial in $W(\widehat{M},\widehat{G})$, and the calculations similar to those in \cite[p.127]{kmsw} prove both two equations in Lemma \ref{2.7.2}.

\subsection{Reduction of LIR to elliptic or exceptional (relative to $G$) parameters}\label{2.8}
In this subsection, we review the reduction of the proof of LIR to the case of elliptic or exceptional parameters.

Recall from \S\ref{1.2} that $\Psi_\el^2(G^*)$ is the set of the equivalence classes of parameters $\psi$ of the form
\begin{align*}
    \psi=2\psi_1\oplus\cdots2\psi_q\oplus\psi_{q+1}\oplus\cdots\oplus\psi_r,
\end{align*}
where $\psi_1,\ldots,\psi_r$ are irreducible, symplectic, and mutually distinct, and $r\geq q\geq1$.
Note that then we have
\begin{align*}
    S_\psi \cong \Or(2,\C)^q \times \Or(1,\C)^{r-q}.
\end{align*}
Let now $\Psi_{\exc1}(G^*)$ and $\Psi_{\exc2}(G^*)$ be the subsets of $\Psi(G^*)$ consisting of the parameters $\psi$ of the form
\begin{itemize}
    \item[$(\exc1)$]$\psi=2\psi_1\oplus\psi_2\oplus\cdots\oplus\psi_r,$ where $\psi_1$ is irreducible and orthogonal, and $\psi_2,\ldots,\psi_r$ are irreducible, symplectic, and mutually distinct,
    \item[$(\exc2)$]$\psi=3\psi_1\oplus\psi_2\oplus\cdots\oplus\psi_r,$ where $\psi_1,\ldots,\psi_r$ are irreducible, symplectic, and mutually distinct,
\end{itemize}
respectively.
They are disjoint.
We then have
\begin{itemize}
    \item[$(\exc1)$]$S_\psi \cong \Sp(2,\C) \times \Or(1,\C)^{r-1},$
    \item[$(\exc2)$]$S_\psi \cong \Or(3,\C) \times \Or(1,\C)^{r-1},$
\end{itemize}
respectively.
We put $\Psi_\exc(G^*)=\Psi_{\exc1}(G^*)\sqcup\Psi_{\exc2}(G^*)$, and we shall say that $\psi$ is exceptional (resp. of type ($\exc1$), resp. of type ($\exc2$)) if $\psi\in\Psi_\exc(G^*)$ (resp. $\Psi_{\exc1}(G^*)$, resp. $\Psi_{\exc2}(G^*)$).
One can see that $\Psi_\exc(G^*)$ and $\Psi_\el(G^*)$ are disjoint.
Put $\Psi^{\el,\exc}(G^*)=\Psi(G^*)\setminus (\Psi_\el(G^*)\sqcup\Psi_\exc(G^*))$.

Let $M^*\subsetneq G^*$ be a proper standard Levi subgroup, and $\xi:G^*\to G$ an inner twist over $F$, such that $M=\xi(M^*)\subsetneq G$ is defined over $F$ if $M^*$ transfers to $G$.
In view of \S\ref{2.7}, it is enough to consider discrete parameters for $M$.
Let $\psi_{M^*}\in\Psi_2(M^*)$ be a discrete local $A$-parameter, and we shall write $\psi$ for its image in $\Psi(G^*)$ in this subsection.
Recall from Lemma \ref{2.8.3} that $\frakS_\psi(M,G)=\frakS_\psi(G)$.
Here we do not assume that $M$ is defined over $F$, nor do that $\psi$ is relevant.

Let $T_\psi$ be the maximal central torus $A_{\widehat{M}}$ of $\widehat{M}$.
Since a Levi subgroup for which the parameter is discrete is unique, the torus $T_\psi$ is determined by $\psi$ and is a maximal torus of $S_\psi=S_\psi(G)$.
One can easily see that
\begin{align*}
    W_\psi=W_\psi(M,G)=W(T_\psi,S_\psi), \\ 
    W_\psi^\circ=W_\psi^\circ(M,G)=W(T_\psi,S_\psi^\circ). 
\end{align*}
Therefore, if $\psi$ is elliptic, then $W_\psi^\circ$ is trivial, as stated in \cite[Lemma 2.8.1]{kmsw}.
Similarly, let $B_\psi$ be the standard Borel subgroup of $S_\psi$ with a maximal torus $T_\psi$.
Then $(T_\psi,B_\psi)$ is a Borel pair of $S_\psi^\circ$.
For any $x\in\frakS_\psi$, following \cite[p.204]{art13}, put
\begin{align*}
    T_{\psi,x} = \cent(s_x,T_\psi)^\circ,
\end{align*}
where $s_x\in S_\psi$ is a representative of $x$ such that $\Ad(s_x)$ stabilizes $(T_\psi,B_\psi)$.
Since $s_x$ is determined up to a $T_\psi$-translate, $T_{\psi,x}$ is uniquely determined by $x$.

The following lemma is the key lemma in this subsection.
\begin{lemma}\label{2.8.4}
    If $\psi \in \Psi^{\el,\exc}(G^*)$, then
    \begin{enumerate}[(1)]
        \item every simple reflection $w\in W_\psi^\circ$ centralizes a torus of positive dimension in $T_\psi$;
        \item $\dim T_{\psi,x} \geq 1, \ \forall x\in \frakS_\psi$.
    \end{enumerate}
\end{lemma}
\begin{proof}
    The proof is similar to that of \cite[Lemma 2.8.4]{kmsw}.
\end{proof}

We now define
\begin{align*}
    S_{\psi,\el}                 & =\set{s\in S_{\psi,\semisimple} |  \text{$Z(\cent(s,S_\psi^\circ))$ is finite}},                      \\
    \overline{S}_{\psi,\el}      & =\set{s\in \overline{S}_{\psi,\semisimple} | \text{$Z(\cent(s,\overline{S}_\psi^\circ))$ is finite}}.
\end{align*}
We also define $\frakS_{\psi,\el}$ (resp. $\overline{\frakS}_{\psi,\el}$) to be the image of $S_{\psi,\el}$ (resp. $\overline{S}_{\psi,\el}$) under the natural surjection $S_\psi\to\frakS_\psi$ (resp. $\overline{S}_\psi\to\overline{\frakS}_\psi$).

\begin{lemma}\label{2.8.6}
    \begin{enumerate}
        \item The natural surjection $S_\psi\to \overline{S}_\psi$ carries $S_{\psi,\el}$ onto $\overline{S}_{\psi,\el}$.
        \item The natural surjection $\frakS_\psi\to \overline{\frakS}_\psi$ carries $\frakS_{\psi,\el}$ onto $\overline{\frakS}_{\psi,\el}$.
        \item If $\psi \in \Psi_\exc(G^*)$ then $\frakS_{\psi,\el}=\frakS_\psi$ and $\overline{\frakS}_{\psi,\el}=\overline{\frakS}_\psi$.
    \end{enumerate}
\end{lemma}
\begin{proof}
    The first part can be seen easily.
    The proof of the later two is similar to that of \cite[Lemma 2.8.6]{kmsw}.
\end{proof}

\begin{lemma}\label{2.8.7}
    Assume that either
    \begin{enumerate}[(1)]
        \item $\psi$ is elliptic, or
        \item every simple reflection $w\in W_\psi^\circ$ centralizes a torus of positive dimension in $T_\psi$.
    \end{enumerate}
    Then the parts 1 and 2 of Theorem \ref{2.6.2} hold for $\psi$ (and for all $u$ and $f$).
\end{lemma}
\begin{proof}
    The proof is similar to that of \cite[Lemma 2.8.7]{kmsw}.
    We appeal to Lemmas \ref{2.5.3} and \ref{2.7.2} instead of Lemmas 2.5.3 and 2.7.2 of loc. cit.
\end{proof}

\begin{lemma}\label{2.8.8}
    Let $x\in \frakS_\psi(M,G)$.
    Assume that either
    \begin{enumerate}[(1)]
        \item $\psi$ is elliptic and $x \notin \frakS_{\psi,\el}$, or
        \item $\psi$ is not elliptic, every simple reflection $w\in W_\psi^\circ$ centralizes a torus of positive dimension in $T_\psi$, and $\dim T_{\psi,x} \geq 1$.
    \end{enumerate}
    Then $f'_G(\psi,s_\psi s\inv)=e(G)f_G(\psi,u^\natural)$ whenever both $u^\natural \in \frakN_\psi(M,G)$ and $s\in S_{\psi,\semisimple}$ map to $x$.
\end{lemma}
\begin{proof}
    Let $s_x \in S_{\psi,\semisimple}$ be a representative of $x$ such that $\Ad(s_x)$ stabilizes $(T_\psi,B_\psi)$.
    Since $T_\psi=A_{\widehat{M}}$, we have $s_x\in N_\psi(M,G)$.
    Lemmas \ref{2.8.3}, \ref{2.6.1}, and \ref{2.8.7} tell us that $f'_G(\psi,s_\psi s\inv)=f'_G(\psi,s_\psi s_x\inv)$ and $f_G(\psi,u^\natural)=f_G(\psi, s_x^\natural)$.
    We need to show that $f'_G(\psi,s_\psi s_x\inv)=e(G)f_G(\psi,s_x^\natural)$.
    Put $\widehat{M}_x=Z_{\widehat{G}}(T_{\psi,x})$.
    Let $M_x^*\subset G^*$ be the Levi subgroup corresponding to $\widehat{M}_x$.
    If the assumption (1) holds, then we have $|T_{\psi,x}|=\infty$ and hence $\widehat{M}_x \subsetneq \widehat{G}$.
    If the assumption (2) holds, then we have $\dim T_{\psi,x}\geq1$. This implies $|T_{\psi,x}|=\infty$ and hence $\widehat{M}_x \subsetneq \widehat{G}$.
    Therefore, $M_x^*$ is a proper Levi subgroup of $G^*$.
    The rest of the proof is similar to that of \cite[Lemma 2.8.8]{kmsw}.
    We appeal to Lemmas \ref{2.6.1} and \ref{2.7.2} instead of Lemmas 2.6.1 and 2.7.2 of loc. cit.
\end{proof}

Lemmas \ref{2.8.4}, \ref{2.8.7}, and \ref{2.8.8} imply the following corollary.
\begin{corollary}\label{2.8.9}
    Let $x\in \frakS_\psi(M,G)$.
    Then Theorem \ref{2.6.2} holds for all $u^\natural \in \frakN_\psi(M,G)$ mapping to $x$ unless either
    \begin{enumerate}[(1)]
        \item $\psi$ is elliptic and $x \in \frakS_{\psi,\el}$, or
        \item $\psi$ is exceptional.
    \end{enumerate}
\end{corollary}

Let us now define
\begin{equation*}
    W_{\psi,\reg}(M,G)=\set{w\in W_\psi(M,G) | \text{$w$ has finitely many fixed points on $T_\psi$}}.
\end{equation*}
We write $\frakN_{\psi,\reg}(M,G)$ for the preimage of $W_{\psi,\reg}(M,G)$ under $\frakN_\psi(M,G)\to W_\psi(M,G)$.
\begin{lemma}\label{2.8.10}
    Assume that $\psi \in \Psi_\exc(G^*)$ and $u^\natural \notin \frakN_{\psi,\reg}(M,G)$.
    Let $s$ be an element of $S_{\psi,\semisimple}$ whose image in $\frakS_\psi(M,G)$ is the same as that of $u^\natural$.
    Then we have $f'_G(\psi,s_\psi s\inv)=e(G)f_G(\psi,u^\natural)$.
\end{lemma}
\begin{proof}
    The proof is similar to that of \cite[Lemma 2.8.10]{kmsw}.
    We appeal to Lemmas \ref{2.6.1} and \ref{2.8.8} instead of Lemmas 2.6.1 and 2.8.8 of loc. cit.
\end{proof}

Consequently, if $\psi\in \Psi^{\el,\exc}(G^*)$, then all the three statements of Theorem \ref{2.6.2} hold (under the induction hypothesis).
If $\psi\in\Psi^2_\el(G^*)$, then the first and second part of Theorem \ref{2.6.2} hold.
In particular, when $\psi$ is not exceptional, $f_G(\psi,x)$ is well-defined for $x\in \frakS_\psi(M,G)$.
Moreover, the third part of the theorem holds if the image of $u$ in $\frakS_\psi$ is not contained in $\frakS_{\psi,\el}$.
Let us consider the case that $\psi$ is exceptional. We have already proven only a part of the third part, in Lemma \ref{2.8.10}.
Let $x\in \frakS_\psi(M,G)$.
Since $\psi$ is exceptional, we have $\abs{W_\psi^\circ}=2$, and there are exactly two elements in the fiber of $x$ in $\frakN_\psi(M,G)$.
One is regular in $W_\psi(M,G)$, and the other is not.
Put $u_x$ to be the former one, and define $f_G(\psi,x)$ to be $f_G(\psi,u_x)$.
Now we have defined $f_G(\psi,x)$ for $x\in \frakS_\psi(M,G)$ for every $\psi\in\Psi(G^*)$

Put $\Phi_\heartsuit(G^*)=\Psi_\heartsuit(G^*)\cap\Phi(G^*)$ and $\Phi_{\bdd,\heartsuit}(G^*)=\Psi_\heartsuit(G^*)\cap\Phi_\bdd(G^*)$ for $\heartsuit\in\{\exc, \exc1, \exc2\}$.

\subsection{LIR for special cases}\label{2.9}
In this subsection we shall prove the two special cases of Theorem \ref{2.6.2}.
Later in \S\S\ref{4.3}-\ref{4.6}, we will reduce the proof of Theorem \ref{2.6.2} for $\phi\in\Phi_{\el,\exc}(G^*)$ to the special cases, which we shall treat in this subsection.

Before we treat the two cases, we recall LLC for $\GL_1(\R)$ and $\GL_2(\R)$, and fix some notation.
We shall realize the Weil group $W_\R$ as $\C^\times\sqcup j\C^\times$, where $j^2=-1$ and $jz=\overline{z}j$ for any $z\in\C^\times$.
Recall from \cite[\S1]{ntb} that the norm map on $W_\R$ is given by $|j|=1$ and $|z|=z\overline{z}$ for $z\in\C\subset W_\R$.
We define the (isomorphism classes of) finite dimensional representations $\sigma$ $\omega_t$, and $\tau_{(l,t)}$ of $W_\R$ by
\begin{align*}
    \sigma       & :W_\R \to \C^\times,                                                                                         &
                 & \begin{dcases}
                       \C^\times \ni z\mapsto 1, \\
                       j\mapsto -1,
                   \end{dcases}                                                                                      \\
    \omega_t     & :W_\R \to \C^\times,                                                                                         &
                 & \begin{dcases}
                       \C^\times \ni z\mapsto |z|^t=(z\overline{z})^t, \\
                       j\mapsto 1,
                   \end{dcases}                                                                \\
    \tau_{(l,t)} & :W_\R \to \GL(2,\C),                                                                                         &
                 & \begin{dcases}
                       \C^\times \ni re^{\sqrt{-1}\theta} \mapsto r^{2t}\left(\begin{array}{cc}
                                                                       e^{\sqrt{-1}l\theta} &                       \\
                                                                                            & e^{-\sqrt{-1}l\theta}
                                                                   \end{array}\right), \\
                       j\mapsto \left(\begin{array}{cc}
                                      & 1 \\
                               (-1)^l &
                           \end{array}\right),
                   \end{dcases}
\end{align*}
for $t\in\C$ and $l\in\Z_{>0}$, and put $\tau_l=\tau_{(l,0)}$.
Then the local Langlands correspondence over $\R$ says that $\sigma^\varepsilon \omega_t$ ($\varepsilon=0,1$) is corresponding to the 1-dimensional representation $|\cdot|_\R^t\sgn^\varepsilon$, and $\tau_{(l,t)}=\tau_l\otimes\omega_t$ to the 2-dimensional representation $D_l\otimes|\det|_\R^t$, where $D_l$ denotes the discrete series representation of $\GL_2(\R)$ of weight $l+1$.
Fix the standard additive character $\psi_\R(x)=\exp(2\pi \sqrt{-1} x)$ of $\R$.
Put $\Gamma_\R(s):=\pi^{-\frac{s}{2}}\Gamma(\frac{s}{2})$ and $\Gamma_\C(s):=2(2\pi)^{-s}\Gamma(s)$.
The $L$-functions and the $\epsilon$-factors are given by the following table.
See \cite{knapp} for more detail, but note that $(+,t)$ and $(-,t)$ in \cite[(3.2)]{knapp} should be $(+,\frac{t}{2})$ and $(-,\frac{t}{2})$ respectively.
\begin{table}[h]
    \centering
    \begin{tabular}{ccc}
        $\varphi$        & $L(s,\varphi)$               & $\epsilon(s,\varphi,\psi_\R)$ \\
        \hline
        $\omega_t$       & $\Gamma_\R(s+t)$             & $1$                           \\
        $\sigma\omega_t$ & $\Gamma_\R(s+t+1)$           & $\sqrt{-1}$                   \\
        $\tau_{(l,t)}$   & $\Gamma_\C(s+t+\frac{l}{2})$ & $\sqrt{-1}^{l+1}$
    \end{tabular}
\end{table}

For a symmetric matrix $Q$ and an alternative matrix $A$ of degree $N$, we shall write
\begin{align*}
    \SO(Q) & =\set{g\in \GL_N | \tp{g}Qg=Q, \ \det g=1}, \\
    \Sp(A) & =\set{g\in \GL_N | \tp{g}Ag=A}.
\end{align*}
\subsubsection{The first special case}\label{muri1}
The first special case we are concerned with is the following.
Let $F=\R$, $n=2$, and
\begin{align*}
    G^* & =\SO_5
    =\SO\left(\begin{array}{ccc}
                          &   & 1_2 \\
                          & 2 &     \\
                      1_2 &   &     \\
                  \end{array}\right),                                                                                                                                   \\
    M^* & =\Set{ \left(\begin{array}{cccc}
                                   t &       &       &   \\
                                     & A     &       & b \\
                                     &       & t\inv &   \\
                                     & \beta &       & c
                               \end{array}\right) | t\in \GL_1,\ A\in \Mat_{2\times2},\ \left(\begin{array}{cc}
                                                                                                  A     & b \\
                                                                                                  \beta & c
                                                                                              \end{array}\right) \in \SO\left(\begin{array}{ccc}
                                                                                                                                    &   & 1 \\
                                                                                                                                    & 2 &   \\
                                                                                                                                  1 &   &   \\
                                                                                                                              \end{array}\right)} \cong \GL_1 \times \SO_3.
\end{align*}
Let $\phi\in\Phi_{\exc1}(G^*)$ be an $L$-parameter of type $(\exc1)$ of the form
\begin{align*}
    \phi=2\omega_0\oplus \tau_1,
\end{align*}
and $\phi_M=\omega_0\oplus \tau_1\in\Phi_2(M^*)$ so that $\phi$ is the natural image of $\phi_M$.
More precisely, we write as
\begin{align*}
    \widehat{G} & =\Sp_\C\left(\begin{array}{cccc}
                                          &    &   & 1 \\
                                          &    & 1 &   \\
                                          & -1 &   &   \\
                                       -1 &    &   &
                                   \end{array}\right),                                                                   \\
    \widehat{M} & =\Set{\left(\begin{array}{ccc}
                                          a &   &       \\
                                            & X &       \\
                                            &   & a\inv
                                      \end{array}\right) \in \widehat{G}
    | a\in\GL(1,\C),\ X\in\Sp_\C\left(\begin{array}{cc} &1 \\ -1& \end{array}\right)=\SL(2,\C)} \subset \widehat{G}, \\
    \phi        & =\phi_M=\left(\begin{array}{ccc}
                                        \omega_0 &        &          \\
                                                 & \tau_1 &          \\
                                                 &        & \omega_0
                                    \end{array}\right).
\end{align*}
Up to isomorphism, there exists only one nontrivial inner twist of $G^*$ for which $\phi$ is relevant.
Let
\begin{align*}
    G & =\SO(1,4)=\SO\left(\begin{array}{cc}
                               1 &      \\
                                 & -1_4
                           \end{array}\right).
\end{align*}
Put
\begin{align*}
    \alpha & =\frac{1}{\sqrt{2}}\left(\begin{array}{ccccc}
                                              1 &           &            & 1  &           \\
                                                & \sqrt{-1} &            &    & \sqrt{-1} \\
                                                &           & 2\sqrt{-1} &    &           \\
                                                & 1         &            &    & -1        \\
                                              1 &           &            & -1 &
                                          \end{array}\right).
\end{align*}
Let $z\in Z^1(\R,G^*)$ be a 1-cocycle such that
\begin{align*}
    z_\rho & =\alpha\inv \rho(\alpha)
    =\left(\begin{array}{ccccc}
               1 &    &    &   &    \\
                 & 0  &    &   & -1 \\
                 &    & -1 &   &    \\
                 &    &    & 1 &    \\
                 & -1 &    &   & 0  \\
           \end{array}\right),
\end{align*}
where $\rho$ is the nontrivial element in $\Gamma_F=\Gal(\C/\R)$, i.e., the complex conjugation. Let $\xi=\Ad(\alpha) : G^*\to G$.
Put
\begin{align*}
    M & =\xi(M^*)                                          \\
      & =\Set{ m(t,B) | t\in\GL_1,\  B\in \SO(-1_3)=\SO(3)
    },
\end{align*}
where
\begin{align*}
    m(t,B) & = \left(\begin{array}{ccc}
                             c(t) &   & s(t) \\
                                  & B &      \\
                             s(t) &   & c(t)
                         \end{array}\right), \\
    c(t)   & =\frac{t+t\inv}{2},         \\
    s(t)   & =\frac{t-t\inv}{2}.
\end{align*}
Then $(\xi, z)$ is a pure inner twist $G^*\to G$ which restricts to a pure inner twist $M^*\to M$ such that $z_+\in Z^1(\R, \GL_1)$ is trivial.
We have $\Pi_{\omega_0}(\GL_1(\R))=\{1\}$, $\Pi_{\tau_1}(\SO(3))=\{1\}$, and hence $\Pi_{\phi_M}(M)=\{1\}$, where $1$ denotes the trivial representation of each group.
Let us write $\pi$ for the unique element in $\Pi_{\phi_M}(M)$, i.e., the trivial representation of $M(\R)$.
\begin{proposition}\label{2.9.1}
    Theorem \ref{2.6.2} is valid for $G$, $M$, and $\phi$.
\end{proposition}
\begin{proof}
    The centralizer $S_\phi(G)$ is
    \begin{align*}
        \Set{\left(\begin{array}{ccc}
                           a &                & b \\
                             & \varepsilon1_2 &   \\
                           c &                & d
                       \end{array}\right) | \left(\begin{array}{cc}a&b\\c&d\end{array}\right) \in \SL(2,\C),\ \varepsilon \in \{\pm1\}}
        \cong \SL(2,\C)\times \Or(1,\C).
    \end{align*}
    Thanks to this explicit structure, one can easily understand the diagram \eqref{cd}.
    The diagram has the form
    \begin{align*}
        \begin{CD}
            @.                             @.                   1                  @.              1              @.       \\
            @.                   @.                                 @VVV                            @VVV                @.  \\
            @.                             @.       \set{\pm1}\times1      @=    \set{\pm1}\times1  @.        \\
            @.                   @.                                   @VVV                            @VVV                @.   \\
            1  @>>> 1\times\set{\pm1} @>>> \set{\pm1}\times\set{\pm1} @>>>  \set{\pm1}\times1  @>>>  1   \\
            @.                  @|                                 @VVV                       @VVV                  @.  \\
            1  @>>> 1\times\set{\pm1} @>>>      1\times\set{\pm1}        @>>>          1          @>>>  1   \\
            @.                  @.                                   @VVV                             @VVV                @.  \\
            @.                             @.                      1                @.                 1          @.
        \end{CD}
    \end{align*}
    where $(-1,1)$ and $(1,-1)$ are represented by
    \begin{align*}
        \left(\begin{array}{cccc}
                         &   &   & 1 \\
                         & 1 &   &   \\
                         &   & 1 &   \\
                      -1 &   &   &
                  \end{array}\right) & \in N(A_{\widehat{M}}, S_\phi(G)^\circ), \\
        \left(\begin{array}{cccc}
                      1 &    &    &   \\
                        & -1 &    &   \\
                        &    & -1 &   \\
                        &    &    & 1
                  \end{array}\right) & \in S_\phi(M),
    \end{align*}
    respectively.
    Note that $(-1,1)$ vanishes in $\frakS_\phi(G)$.
    Let us write $w$ for the image of $(-1,1)$ in $W_\phi(M,G)=W_\phi(M,G)^\circ\cong W(M^*,G^*)\cong W(M,G)$.
    One can calculate the Langlands-Shelstad lift in $\widehat{G}=\Sp(4,\C)$ of $w$ as in \cite{ishi}.
    It is
    \begin{align*}
        \left(\begin{array}{cccc}
                    &    &    & -1 \\
                    & -1 &    &    \\
                    &    & -1 &    \\
                  1 &    &    &
              \end{array}\right),
    \end{align*}
    and thus the first sophisticated splitting $s' : W_\psi(M,G)\to \frakN_\psi(M,G)$ sends $(-1,1)$ to $(-1,-1)$.
    Therefore, the other sophisticated splitting $s : \frakN_\psi(M,G)\to \frakS_\psi(M)$ sends $(-1,-1)$ to $1$, and $(-1,1)$ to $(1,-1)$.

    Let $P^*$ be the standard parabolic subgroup of $G^*$ with Levi subgroup $M^*$, and $P=\xi(P^*)$.
    Now in the same way as \cite[pp.136-137]{kmsw}, we know that the induced representation $(\calI_P(\pi), \calH_P(\pi))$ is irreducible and is the unique element of $\Pi_\phi(G)$, and that it is sufficient to show that the operator $R_P((-1,1), \pi,\phi,\psi_\R)$ is the identity map.
    Since $\calI_P(\pi)$ is irreducible, $R_P((-1,1), \pi,\phi,\psi_\R)$ is a scalar map.
    Hence it is sufficient to show
    \begin{align*}
        R_P((-1,1), \pi,\phi,\psi_\R) f=f,
    \end{align*}
    for a nonzero element $f$ in $\calH_P(\pi)$.
    Recall that
    \begin{equation*}
        R_P((-1,1), \pi, \phi, \psi_\R)=\calI_P(\pi((-1,1))_{\xi,z}) \circ \ell^P_P(w, \xi, \phi, \psi_\R) \circ R_{P^w|P}(\xi,\phi).
    \end{equation*}

    We focus first on $\calI_P(\pi((-1,1))_{\xi,z})=\calI_P(\an{(-1,1),\pi}_{\xi,z} \pi(\breve{w})_\xi)$.
    Since $\pi$ is the trivial representation on the 1-dimensional vector space $\C$, the operator $\pi(\breve{w})_\xi$ is the identity map.
    On the other hand, we have $\an{(-1,1),\pi}_{\xi,z}=\an{s(-1,1),\pi}_M=\an{-1,\pi}_M=-1$.
    We obtain $\calI_P(\pi((-1,1))_{\xi,z})=-1$.

    Next we consider $\ell^P_P(w, \xi, \phi, \psi_\R) \circ R_{P^w|P}(\xi,\phi)$.
    Let $\widetilde{w}$ be the Langlands-Shelstad lift of $w$ in $G^*(\R)$.
    The calculation of $\widetilde{w}$ has already been done in \cite[\S8]{ishi}:
    \begin{align*}
        \widetilde{w}=\left(\begin{array}{ccccc}
                                      &    &    & 1 &    \\
                                      & -1 &    &   &    \\
                                      &    & -1 &   &    \\
                                    1 &    &    &   &    \\
                                      &    &    &   & -1
                                \end{array}\right).
    \end{align*}
    We have $\breve{w}=\xi(\widetilde{w})=\mmatrix{1}{}{}{-1_4}\in G(\R)$.
    Let us define a connected compact subgroup $K\subset G(\R)$ as
    \begin{align*}
        K=\Set{\left(\begin{array}{cc}
                             1 &        \\
                               & \kappa
                         \end{array}\right) | \kappa \in \SO(4)}.
    \end{align*}
    By abuse of notation, we shall write $\kappa$ for the element $\mmatrix{1}{}{}{\kappa}$.
    Although $K$ is not maximal, the Iwasawa decomposition tells us that $G(\R)=P(\R)K$.
    Let $N^*$ be the unipotent radical of $P^*$.
    Put
    \begin{align*}
        n^*(b)
         & =\left(\begin{array}{ccccc}
                      1 & b_1 & b_2 & -b_1b_3-\frac{b_2^2}{4} & b_3 \\
                        & 1   &     & -b_3                    &     \\
                        &     & 1   & -\frac{b_2}{2}          &     \\
                        &     &     & 1                       &     \\
                        &     &     & -b_1                    & 1   \\
                  \end{array}\right),
    \end{align*}
    for $b=(b_1,b_2,b_3)\in\C^3$, so that $N^*(\R)=\{n^*(b) \mid b\in\R^3 \}$.
    A direct calculation following \S\ref{2.2} show that the Haar measure on $N^*(\R)$ is $d(n^*(b))=\frac{1}{2}db_1 db_2 db_3$.
    Put $P=\xi(P^*)$ and $N=\xi(N^*)$, which are the standard parabolic subgroup with Levi subgroup $M$ and its unipotent radical.
    Put also
    \begin{align*}
        n(x)
         & =\left(\begin{array}{ccc}
                      1+\frac{\|x\|^2}{2} & x   & -\frac{\|x\|^2}{2}  \\
                      \tp{x}              & 1_3 & -\tp{x}             \\
                      \frac{\|x\|^2}{2}   & x   & 1-\frac{\|x\|^2}{2}
                  \end{array}\right)
        =\left(\begin{array}{ccccc}
                   1+\frac{\|x\|^2}{2} & x_1 & x_2 & x_3 & -\frac{\|x\|^2}{2}  \\
                   x_1                 & 1   &     &     & -x_1                \\
                   x_2                 &     & 1   &     & -x_2                \\
                   x_3                 &     &     & 1   & -x_3                \\
                   \frac{\|x\|^2}{2}   & x_1 & x_2 & x_3 & 1-\frac{\|x\|^2}{2} \\
               \end{array}\right),
    \end{align*}
    for $x=(x_1,x_2,x_3)\in\C^3$, where $\|x\|^2=x_1^2+x_2^2+x_3^2$ so that $N(\R)=\{n(x) \mid x\in\R^3 \}$.
    A straightforward calculation shows that $\xi(n^*(b))=n(x)$ where
    \begin{align*}
        x_1 & =-\frac{b_1+b_3}{2}\sqrt{-1}, & x_2 & =-\frac{b_2}{2}\sqrt{-1}, & x_3 & =\frac{b_1-b_3}{2}, &
    \end{align*}
    and hence the measure on $N(\R)$ is $d(n(x))=2dx_1 dx_2 dx_3$.

    For $\lambda\in\C$, put $\phi_\lambda=\omega_\lambda\oplus\tau_1\oplus\omega_{-\lambda}$ and $\pi_\lambda=|\cdot|^\lambda\boxtimes1$, so that $\Pi_{\phi_\lambda}(M)=\{\pi_\lambda\}$ and $\Pi_{\phi_\lambda}(G)=\{\calI_P(\pi_\lambda)\}$.
    Define a function $\varphi^{(\lambda)}\in \calH_P(\pi_\lambda)$ by
    \begin{align*}
        \varphi^{(\lambda)}(m(t,B)n(x)\kappa)=\abs{t}^{\lambda+\frac{3}{2}},
    \end{align*}
    which is holomorphic in $\lambda\in\C$.
    Assume that $\myre(\lambda)>0$.
    We have $0\neq \varphi^{(0)}\in\calH_P(\pi)$ and
    \begin{align*}
         & \left[\ell^P_P(w, \xi, \phi, \psi_\R) \circ R_{P^w|P}(\xi,\phi) \varphi^{(0)} \right](g) \notag                                           \\
         & =\lim_{\lambda\to+0} \left[\ell^P_P(w, \xi, \phi_\lambda, \psi_\R) \circ R_{P^w|P}(\xi,\phi_\lambda) \varphi^{(\lambda)}\right](g) \notag \\
         & =\lim_{\lambda\to+0}
        \epsilon(0,\rho_{P^w|P}^\vee\circ\phi_\lambda, \psi_\R)
        \frac{L(1,\rho_{P^w|P}^\vee\circ\phi_\lambda)} {L(0,\rho_{P^w|P}^\vee\circ\phi_\lambda)}
        \int_{N(\R)} \varphi^{(\lambda)}(\breve{w}\inv n g) dn.
    \end{align*}
    A direct calculation implies that $\rho_{P^w|P}^\vee\circ\phi_\lambda$ is isomorphic to
    \begin{equation*}
        \tau_{(1,\lambda)} \oplus \omega_{2\lambda}.
    \end{equation*}
    Therefore, we have
    \begin{align*}
        \epsilon(0,\rho_{P^w|P}^\vee\circ\phi_\lambda, \psi_\R)
        \frac{L(1,\rho_{P^w|P}^\vee\circ\phi_\lambda)} {L(0,\rho_{P^w|P}^\vee\circ\phi_\lambda)}
        =-\frac{\Gamma_\C(\lambda+\frac{3}{2})}{\Gamma_\C(\lambda+\frac{1}{2})} \frac{\Gamma_\R(2\lambda+1)}{\Gamma_\R(2\lambda)}.
    \end{align*}

    We now turn to the integral.
    Put
    \begin{align*}
        [\calM_{\SO(1,4)} \varphi^{(\lambda)}](g)
        =\int_{N(\R)} \varphi^{(\lambda)}(\breve{w}\inv n g) dn.
    \end{align*}
    In Lemma \ref{so14} below, we will show that
    \begin{align*}
        \calM_{\SO(1,4)} \varphi^{(\lambda)}=2^{-\frac{1}{2}}\frac{\Gamma_\C(\lambda)}{\Gamma_\C(\lambda+\frac{3}{2})}\varphi^{(-\lambda)}.
    \end{align*}
    This leads to
    \begin{align*}
         & \epsilon(0,\rho_{P^w|P}^\vee\circ\phi_\lambda, \psi_\R)
        \frac{L(1,\rho_{P^w|P}^\vee\circ\phi_\lambda)} {L(0,\rho_{P^w|P}^\vee\circ\phi_\lambda)}
        \int_{N(\R)} \varphi^{(\lambda)}(\breve{w}\inv n \cdot) dn                                                                       \\
         & =-\frac{\Gamma_\C(\lambda+\frac{3}{2})}{\Gamma_\C(\lambda+\frac{1}{2})} \frac{\Gamma_\R(2\lambda+1)}{\Gamma_\R(2\lambda)}
        \cdot
        2^{-\frac{1}{2}}\frac{\Gamma_\C(\lambda)}{\Gamma_\C(\lambda+\frac{3}{2})}\varphi^{(-\lambda)}                                    \\
         & =-2^{-\frac{1}{2}}\frac{\Gamma_\C(\lambda)}{\Gamma_\C(\lambda+\frac{1}{2})} \frac{\Gamma_\R(2\lambda+1)}{\Gamma_\R(2\lambda)}
        \varphi^{(-\lambda)},
    \end{align*}
    whose limit as $\lambda$ approaching $0$ from the right is $-\varphi^{(0)}$.
    Since $\varphi^{(0)}\neq0$, this completes the proof.
\end{proof}

To finish the proof of Proposition \ref{2.9.1}, it remains to show the following lemma.
\begin{lemma}\label{so14}
    For $\myre(\lambda)>0$, we have
    \begin{align*}
        \calM_{\SO(1,4)} \varphi^{(\lambda)}=2^{-\frac{1}{2}}\frac{\Gamma_\C(\lambda)}{\Gamma_\C(\lambda+\frac{3}{2})}\varphi^{(-\lambda)}.
    \end{align*}
\end{lemma}
\begin{proof}
    By definition, $\calM_{\SO(1,4)} \varphi^{(\lambda)}$ is right $K$-invariant and left $N(\R)$-invariant.
    Moreover, we have
    \begin{equation*}
        [\calM_{\SO(1,4)} \varphi^{(\lambda)}](m(t,B)g)=|t|^{-\lambda+\frac{3}{2}}[\calM_{\SO(1,4)} \varphi^{(\lambda)}](g).
    \end{equation*}
    Indeed, a direct calculation shows that
    \begin{align*}
        n(x) m(t,B)=m(t,B)n(t\inv xB), \\
        \breve{w}\inv m(t,B)=m(t\inv,B)\breve{w}\inv.
    \end{align*}
    Thus we have
    \begin{align*}
        [\calM_{\SO(1,4)} \varphi^{(\lambda)}](m(t,B)g)
         & =\int_{x\in\R^3} \varphi^{(\lambda)}(\breve{w}\inv n(x) m(t,B) g) d(n(x))                                       \\
         & =\int_{x\in\R^3} \varphi^{(\lambda)}( m(t\inv,B) \breve{w}\inv n(t\inv xB) g) \cdot 2dx_1 dx_2 dx_3             \\
         & =|t|^{-\lambda-\frac{3}{2}}\int_{y\in\R^3} \varphi^{(\lambda)}(\breve{w}\inv n(y) g) \cdot 2|t|^3dy_1 dy_2 dy_3 \\
         & =|t|^{-\lambda+\frac{3}{2}}[\calM_{\SO(1,4)} \varphi^{(\lambda)}](g),
    \end{align*}
    where we have changed the variables $y=t\inv xB$.
    Combining these properties with $\varphi^{(\lambda)}(1)=1$, we now have
    \begin{equation*}
        \calM_{\SO(1,4)} \varphi^{(\lambda)}=\calM_{\SO(1,4)} \varphi^{(\lambda)}(1) \cdot \varphi^{(-\lambda)}.
    \end{equation*}
    We must then show that
    \begin{equation*}
        \calM_{\SO(1,4)} \varphi^{(\lambda)}(1)=2^{-\frac{1}{2}}\frac{\Gamma_\C(\lambda)}{\Gamma_\C(\lambda+\frac{3}{2})}.
    \end{equation*}

    For $x, y\in\R^3$, $t\in\R^\times$, $B\in\SO(3)$, and $\kappa\in\SO(4)\cong K$, a direct calculation shows that
    \begin{align*}
        \breve{w}\inv n(x)
         & =\left(\begin{array}{ccc}
                          1+\frac{\|x\|^2}{2} & x    & -\frac{\|x\|^2}{2}   \\
                          -\tp{x}             & -1_3 & \tp{x}               \\
                          -\frac{\|x\|^2}{2}  & -x   & -1+\frac{\|x\|^2}{2}
                      \end{array}\right), \\
        m(t,B)n(y)\kappa
         & =\left(\begin{array}{ccc}
                          c(t)+\frac{t\|y\|^2}{2} & * & * \\
                          B\tp{y}                 & * & * \\
                          s(t)+\frac{t\|y\|^2}{2} & * & *
                      \end{array}\right),
    \end{align*}
    so if we write $\breve{w}\inv n(x)=m(t,B)n(y)\kappa$ then we have
    \begin{align*}
        \begin{dcases}
            1+\frac{\|x\|^2}{2} & =c(t)+\frac{t\|y\|^2}{2}, \\
            -\tp{x}             & =B\tp{y},                 \\
            -\frac{\|x\|^2}{2}  & =s(t)+\frac{t\|y\|^2}{2},
        \end{dcases}
    \end{align*}
    which leads $t=(1+\|x\|^2)\inv$.
    Now the value $\calM_{\SO(1,4)} \varphi^{(\lambda)}(1)$ is equal to
    \begin{align*}
         & \int_{x\in\R^3} \varphi^{(\lambda)}(\breve{w}\inv n(x)) d(n(x))                       \\
         & =\int_{x\in\R^3} \left| 1+\|x\|^2 \right|^{-\lambda-\frac{3}{2}} \cdot2dx_1 dx_2 dx_3 \\
         & =\int_{r=0}^\infty \int_{\theta=0}^\pi \int_{\alpha=0}^{2\pi} \frac{2r^2\sin\theta}
        {(1+r^2)^{\lambda+\frac{3}{2}}} d\alpha d\theta dr                                       \\
         & =8\pi \int_0^\infty \frac{r^2}{(1+r^2)^{\lambda+\frac{3}{2}}} dr,
    \end{align*}
    where we have changed the variables $x_1=r\cos\alpha\sin\theta$, $x_2=r\sin\alpha\sin\theta$, and $x_3=r\cos\theta$.
    Since $\myre(\lambda)>0$, the last integral converges absolutely, and the gamma function $\Gamma(\lambda+\frac{3}{2})$ has the integral expression which also converges absolutely.
    Therefore, the product of them is
    \begin{align*}
         & 8\pi \int_0^\infty \frac{r^2}{(1+r^2)^{\lambda+\frac{3}{2}}} dr \times \Gamma(\lambda+\frac{3}{2})             \\
         & =8\pi \int_0^\infty \int_0^\infty \left(\frac{t}{1+r^2}\right)^{\lambda+\frac{3}{2}} r^2 e^{-t} dr d^\times t,
    \end{align*}
    where $d^\times t=t\inv dt$.
    First we fix $r$ and change the variables $u=(1+r^2)\inv t$, and then fix $u$ and change the variables $s=\sqrt{u}r$, to obtain
    \begin{align*}
         & 8\pi \int_0^\infty \int_0^\infty \left(\frac{t}{1+r^2}\right)^{\lambda+\frac{3}{2}} r^2 e^{-t} dr d^\times t \\
         & =8\pi \int_0^\infty u^\lambda  e^{-u} d^\times u  \int_0^\infty s^2 e^{-s^2} ds                              \\
         & =2\pi^\frac{3}{2}\Gamma(\lambda).
    \end{align*}
    We now have
    \begin{align*}
        \calM_{\SO(1,4)} \varphi^{(\lambda)}(1)
         & =8\pi \int_0^\infty \frac{r^2}{(1+r^2)^{\lambda+\frac{3}{2}}} dr      \\
         & =2\pi^\frac{3}{2} \frac{\Gamma(\lambda)}{\Gamma(\lambda+\frac{3}{2})}
        =2^{-\frac{1}{2}}\frac{\Gamma_\C(\lambda)}{\Gamma_\C(\lambda+\frac{3}{2})}.
    \end{align*}
\end{proof}


\subsubsection{The second special case}\label{muri2}
The second special case we are concerned with is the following.
Let $F=\R$, $n=3$, and
\begin{align*}
    G^* & =\SO_7
    =\SO\left(\begin{array}{ccc}
                          &   & 1_3 \\
                          & 2 &     \\
                      1_3 &   &
                  \end{array}\right),                                                                                                                                   \\
    M^* & =\Set{ \left(\begin{array}{cccc}
                                   A &       &            &   \\
                                     & X     &            & b \\
                                     &       & \tp{A}\inv &   \\
                                     & \beta &            & c
                               \end{array}\right) | A\in \GL_2,\ X\in \Mat_{2\times2},\ \left(\begin{array}{cc}
                                                                                                  X     & b \\
                                                                                                  \beta & c
                                                                                              \end{array}\right) \in \SO\left(\begin{array}{ccc}
                                                                                                                                    &   & 1 \\
                                                                                                                                    & 2 &   \\
                                                                                                                                  1 &   &   \\
                                                                                                                              \end{array}\right)} \cong \GL_2 \times \SO_3.
\end{align*}
Let $\phi\in\Phi_{\exc2}(G^*)$ be an $L$-parameter of type $(\exc2)$ of the form
\begin{align*}
    \phi=3\tau_1,
\end{align*}
and $\phi_M=2\tau_1\in\Phi_2(M^*)$ (as equivalent classes) so that $\phi$ is the natural image of $\phi_M$.
More precisely, we define the parameter as
\begin{align*}
    \widehat{G} & =\Sp_\C\left(\begin{array}{cccc}
                                            &    &   & 1_2 \\
                                            &    & 1 &     \\
                                            & -1 &   &     \\
                                       -1_2 &    &   &
                                   \end{array}\right),                                                                   \\
    \widehat{M} & =\Set{\left(\begin{array}{ccc}
                                          A &   &            \\
                                            & X &            \\
                                            &   & \tp{A}\inv
                                      \end{array}\right) \in \widehat{G}
    | A\in\GL(2,\C),\ X\in\Sp_\C\left(\begin{array}{cc} &1 \\ -1& \end{array}\right)=\SL(2,\C)} \subset \widehat{G}, \\
    \phi        & =\phi_M=\left(\begin{array}{ccc}
                                        \tau_1 &        &                 \\
                                               & \tau_1 &                 \\
                                               &        & \tp{\tau_1}\inv
                                    \end{array}\right).
\end{align*}
Up to isomorphism, there exists only one nontrivial inner twist of $G^*$ for which $\phi$ is relevant.
Let
\begin{align*}
    G & =\SO(2,5)=\SO\left(\begin{array}{cc}
                               1_2 &      \\
                                   & -1_5
                           \end{array}\right).
\end{align*}
Put
\begin{align*}
    \alpha & =\frac{1}{\sqrt{2}}\left(\begin{array}{ccccc}
                                              1_2 &           &            & 1_2  &           \\
                                                  & \sqrt{-1} &            &      & \sqrt{-1} \\
                                                  &           & 2\sqrt{-1} &      &           \\
                                                  & 1         &            &      & -1        \\
                                              1_2 &           &            & -1_2 &
                                          \end{array}\right).
\end{align*}
Let $z\in Z^1(\R,G^*)$ be a 1-cocycle such that
\begin{align*}
    z_\rho & =\alpha\inv \rho(\alpha)
    =\left(\begin{array}{ccccc}
               1_2 &    &    &     &    \\
                   & 0  &    &     & -1 \\
                   &    & -1 &     &    \\
                   &    &    & 1_2 &    \\
                   & -1 &    &     & 0  \\
           \end{array}\right),
\end{align*}
where $\rho$ is the nontrivial element in $\Gamma_F=\Gal(\C/\R)$.
Put $\xi=\Ad(\alpha) : G^*\to G$ and
\begin{align*}
    M & =\xi(M^*)                                           \\
      & =\Set{ m(A,B)| A\in\GL_2,\  B\in \SO(-1_3)=\SO(3)},
\end{align*}
where
\begin{align*}
    m(A,B) & = \left(\begin{array}{ccc}
                             c(A) &   & s(A) \\
                                  & B &      \\
                             s(A) &   & c(A)
                         \end{array}\right), \\
    c(A)   & =\frac{A+\tp{A}\inv}{2},    \\
    s(A)   & =\frac{A-\tp{A}\inv}{2}.
\end{align*}
Then $(\xi, z)$ is a pure inner twist $G^*\to G$ which restricts to a pure inner twist $M^*\to M$ such that $z_+\in Z^1(\R, \GL_2)$ is trivial.
We have $\Pi_{\tau_1}(\GL_2(\R))=\{D_1\}$, $\Pi_{\tau_1}(\SO(3))=\{1\}$, and hence $\Pi_{\phi_M}(M)=\{D_1\boxtimes1\}$, where $1$ denotes the trivial representation of $\SO(3)$.
Let us write $\pi$ for the unique element in $\Pi_{\phi_M}(M)$.
\begin{proposition}\label{2.9.2}
    Theorem \ref{2.6.2} is valid for $G$, $M$, and $\phi$.
\end{proposition}
\begin{proof}
    Put $J=\mmatrix{}{1}{-1}{}$, and $\bfJ=\mmatrix{1_4}{}{}{J}$.
    Let us define a mapping $A\mapsto A^{(2)}$ from the set of $3\times3$-matrices to those of $6\times6$-matrices by
    \begin{align*}
        A=\left(\begin{array}{ccc}
                        a_{1,1} & a_{1,2} & a_{1,3} \\
                        a_{2,1} & a_{2,2} & a_{2,3} \\
                        a_{3,1} & a_{3,2} & a_{3,3}
                    \end{array}\right)
        \mapsto
        A^{(2)}:=
        \left(\begin{array}{ccc}
                      a_{1,1}1_2 & a_{1,2}1_2 & a_{1,3}1_2 \\
                      a_{2,1}1_2 & a_{2,2}1_2 & a_{2,3}1_2 \\
                      a_{3,1}1_2 & a_{3,2}1_2 & a_{3,3}1_2
                  \end{array}\right).
    \end{align*}
    Since $\tp{\tau_1}\inv=J \tau_1 J\inv$, the centralizer $S_\phi(G)$ is
    \begin{align*}
        \Set{\bfJ h^{(2)} \bfJ\inv | h\in \Or_\C\left(\begin{array}{ccc}
                                                                    &   & 1 \\
                                                                    & 1 &   \\
                                                                  1 &   &
                                                              \end{array}\right)}
        \cong \Or(3,\C).
    \end{align*}
    Thanks to this explicit structure, one can easily understand the diagram \eqref{cd}.
    The diagram has the form
    \begin{align*}
        \begin{CD}
            @.                             @.                   1                  @.              1              @.       \\
            @.                   @.                                 @VVV                            @VVV                @.  \\
            @.                             @.       \set{\pm1}\times1      @=    \set{\pm1}\times1  @.        \\
            @.                   @.                                   @VVV                            @VVV                @.   \\
            1  @>>> 1\times\set{\pm1} @>>> \set{\pm1}\times\set{\pm1} @>>>  \set{\pm1}\times1  @>>>  1   \\
            @.                  @|                                 @VVV                       @VVV                  @.  \\
            1  @>>> 1\times\set{\pm1} @>>>      1\times\set{\pm1}        @>>>          1          @>>>  1   \\
            @.                  @.                                   @VVV                             @VVV                @.  \\
            @.                             @.                      1                @.                 1          @.
        \end{CD}
    \end{align*}
    where $(-1,1)$ and $(1,-1)$ are represented by
    \begin{align*}
        \left(\begin{array}{ccc}
                        &      & J\inv \\
                        & -1_2 &       \\
                      J &      &
                  \end{array}\right) & \in N(A_{\widehat{M}}, S_\phi(G)^\circ), \\
        \left(\begin{array}{ccc}
                      1_2 &      &     \\
                          & -1_2 &     \\
                          &      & 1_2
                  \end{array}\right) & \in S_\phi(M),
    \end{align*}
    respectively.
    Note that $(-1,1)$ vanishes in $\frakS_\phi(G)$.
    Let us write $w$ for the image of $(-1,1)$ in $W_\phi(M,G)=W_\phi(M,G)^\circ\cong W(M^*,G^*)\cong W(M,G)$.
    One can calculate the Langlands-Shelstad lift in $\widehat{G}$ of $w$ as in \cite{ishi}.
    It is
    \begin{align*}
        \left(\begin{array}{ccc}
                    &     & J\inv \\
                    & 1_2 &       \\
                  J &     &
              \end{array}\right),
    \end{align*}
    and thus the first sophisticated splitting $s' : W_\psi(M,G)\to \frakN_\psi(M,G)$ sends $(-1,1)$ to $(-1,-1)$.
    Therefore, the other splitting $s : \frakN_\psi(M,G)\to \frakS_\psi(M)$ sends $(-1,-1)$ to $1$, and $(-1,1)$ to $(1,-1)$.

    Let $P^*$ be the standard parabolic subgroup of $G^*$ with Levi subgroup $M^*$, and $P=\xi(P^*)$.
    As in the previous subsubsection, the induced representation $(\calI_P(\pi), \calH_P(\pi))$ is irreducible and is the unique element of $\Pi_\phi(G)$, and it is sufficient to show
    \begin{align*}
        R_P((-1,1), \pi,\phi,\psi_\R) f=f,
    \end{align*}
    for a nonzero element $f$ in $\calH_P(\pi)$.
    Recall that
    \begin{equation*}
        R_P((-1,1), \pi, \phi, \psi_\R)=\calI_P(\pi((-1,1))_{\xi,z}) \circ \ell^P_P(w, \xi, \phi, \psi_\R) \circ R_{P^w|P}(\xi,\phi).
    \end{equation*}

    We focus first on $\calI_P(\pi((-1,1))_{\xi,z})=\calI_P(\an{(-1,1),\pi}_{\xi,z} \pi(\breve{w})_\xi)$.
    Let $\widetilde{w}$ be the Langlands-Shelstad lift of $w$ in $G^*(\R)$.
    The calculation of $\widetilde{w}$ has already been done in \cite[\S8]{ishi}:
    \begin{align*}
        \widetilde{w}=\left(\begin{array}{ccccc}
                                      &   &   & J &   \\
                                      & 1 &   &   &   \\
                                      &   & 1 &   &   \\
                                    J &   &   &   &   \\
                                      &   &   &   & 1
                                \end{array}\right).
    \end{align*}
    Thus we have
    \begin{align*}
        \breve{w}=\xi(\widetilde{w})=\left(\begin{array}{ccc}
                                                   J &     &       \\
                                                     & 1_3 &       \\
                                                     &     & J\inv
                                               \end{array}\right)\in G(\R).
    \end{align*}
    Since $\breve{w}\inv m(A,B) \breve{w}=m((\det A)\inv A, B)$ and the central character of $D_1$ is trivial, the representations $\breve{w}\pi$ and $\pi$ coincide.
    Thus $\pi(\breve{w})_\xi$ is the identity map.
    On the other hand, we have $\an{(-1,1),\pi}_{\xi,z}=\an{s(-1,1),\pi}_M=\an{-1,\pi}_M=-1$.
    We obtain $\calI_P(\pi((-1,1))_{\xi,z})=-1$.

    Next we consider $\ell^P_P(w, \xi, \phi, \psi_\R) \circ R_{P^w|P}(\xi,\phi)$.
    Let us define a connected compact subgroup $K\subset G(\R)$ as
    \begin{align*}
        K=\Set{\kappa=\left(\begin{array}{cc}
                                    \kappa_2 &          \\
                                             & \kappa_5
                                \end{array}\right) | \kappa_m \in \SO(m),\ (m=2,5)}.
    \end{align*}
    Although $K$ is not maximal, the Iwasawa decomposition tells us that $G(\R)=P(\R)K$.
    Let $N^*$ be the unipotent radical of $P^*$, which is generated by $\{\exp(X_\alpha) \mid \alpha=\chi_1\pm\chi_3, \chi_1, \chi_2\pm\chi_3, \chi_2, \chi_1+\chi_2\}$.
    Here $X_\alpha$ denotes the Chevalley basis.
    The Haar measure on $N^*(\R)$ is also given by the Chevalley basis.
    Put $P=\xi(P^*)$ and $N=\xi(N^*)$, which are the standard parabolic subgroup with Levi subgroup $M$ and its unipotent radical.
    Now we describe $N(\R)$ explicitly.
    Let $\iota_1$ and $\iota_2$ be embeddings of $\SO(1,4)$ into $G=\SO(2,5)$ given by
    \begin{align*}
        h=\left(\begin{array}{cc}
                    a     & b \\
                    \beta & C
                \end{array}\right)
         & \mapsto
        \iota_1(h)=\left(\begin{array}{cccc}
                             a     &   & b &   \\
                                   & 1 &   &   \\
                             \beta &   & C &   \\
                                   &   &   & 1
                         \end{array}\right), \\
        h=\left(\begin{array}{cc}
                    A     & b \\
                    \beta & c
                \end{array}\right)
         & \mapsto
        \iota_2(h)=\left(\begin{array}{cccc}
                             1 &       &   &   \\
                               & A     &   & b \\
                               &       & 1 &   \\
                               & \beta &   & c
                         \end{array}\right),
    \end{align*}
    where $A$ and $C$ are 4-by-4 matrices.
    For $x=(x_1,x_2,x_3)\in\C^3$ and $u\in\C$, put $n_1(x)=\iota_1(n(x))$ and $n_2(x)=\iota_2(n(x))$, where $n(x)$ is the element defined in the proof of Proposition \ref{2.9.1}, and put
    \begin{align*}
        n_c(u)
         & =\left(\begin{array}{ccccc}
                      1            & \frac{u}{2} &     & 0           & -\frac{u}{2} \\
                      -\frac{u}{2} & 1           &     & \frac{u}{2} & 0            \\
                                   &             & 1_3 &             &              \\
                      0            & \frac{u}{2} &     & 1           & -\frac{u}{2} \\
                      -\frac{u}{2} & 0           &     & \frac{u}{2} & 1
                  \end{array}\right),
    \end{align*}
    so that
    \begin{align*}
        n_1(x) & =\xi( \exp( \sqrt{-1}x_1(X_{\chi_1-\chi_3}+X_{\chi_1+\chi_3}) + \sqrt{-1}x_2X_{\chi_1} + x_3(X_{\chi_1-\chi_3}-X_{\chi_1+\chi_3}) ) ), \\
        n_2(x) & =\xi( \exp( \sqrt{-1}x_1(X_{\chi_2-\chi_3}+X_{\chi_2+\chi_3}) + \sqrt{-1}x_2X_{\chi_2} + x_3(X_{\chi_2-\chi_3}-X_{\chi_2+\chi_3}) ) ), \\
        n_c(u) & =\xi( \exp( u X_{\chi_1+\chi_2} ) ).
    \end{align*}
    Let $N_1$ (resp. $N_2$, resp. $N_c$) be the unipotent subgroup of $G$ consisting of $n_1(x)$ (resp. $n_2(x)$, resp. $n_c(u)$).
    Then the Haar measure on $N_1(\R)$ (resp. $N_2(\R)$, resp. $N_c(\R)$) is given by $d(n_1(x))=2dx_1dx_2dx_3$ (resp. $d(n_2(x))=2dx_1dx_2dx_3$, resp. $d(n_c(u))=du$).
    We have $N(\R)=N_1(\R) N_2(\R) N_c(\R)=\{n_1(x) n_c(u) n_2(y) \mid x, y\in\R^3, \ u\in\R \}$ and the Haar measure is the product of those on $N_1(\R)$, $N_2(\R)$, and $N_c(\R)$.

    Let us realize the discrete series representation $D_1$ of $\GL_2(\R)$ as a subrepresentation of a parabolically induced representation from a character $|-|^{\frac{1}{2}}\boxtimes|-|^{-\frac{1}{2}}$ on the diagonal maximal torus of $\GL_2(\R)$, and let $0\neq v_0\in D_1$ be a lowest weight vector such that
    \begin{align*}
        v_0\left(\left(\begin{array}{cc}a&b\\&d\end{array}\right) \left(\begin{array}{cc}\cos\theta&\sin\theta\\-\sin\theta&\cos\theta\end{array}\right)\right)
        =\left|\frac{a}{d}\right| e^{2\sqrt{-1}\theta}.
    \end{align*}
    For $\lambda\in\C$, put $\phi_\lambda=\tau_{(1,\lambda)}\oplus\tau_1\oplus\tau_{(1,\lambda)}^\vee$ and $\pi_\lambda=(D_1\otimes|\cdot|^\lambda)\boxtimes1$, so that $\Pi_{\phi_\lambda}(M)=\{\pi_\lambda\}$ and $\Pi_{\phi_\lambda}(G)=\{\calI_P(\pi_\lambda)\}$.
    Since $D_1$ is realized as a subrepresentation of a parabolically induced representation, we have a natural injection
    \begin{align*}
        \calH_P(\pi_\lambda) & \into \calH_{P_0}(|-|^{\frac{1}{2}+\lambda}\boxtimes|-|^{-\frac{1}{2}+\lambda}\boxtimes1), &
        f(\cdot) \mapsto f(\cdot)(1),
    \end{align*}
    where $P_0$ is the minimal standard parabolic subgroup of $G$ defined over $\R$.
    Define a function $f^{(\lambda)}\in \calH_P(\pi_\lambda) \subset \calH_{P_0}(|-|^{\frac{1}{2}+\lambda}\boxtimes|-|^{-\frac{1}{2}+\lambda}\boxtimes1)$ by
    \begin{align*}
        f^{(\lambda)}(m(A,B)n\kappa)
         & =|\det A|^{\lambda+2}[\tau_1(A)v_0](1)                                                                                                                                                                   \\
         & =|a|^{\lambda+3} |d|^{\lambda+1} e^{2\sqrt{-1}\theta}, & A & =\left(\begin{array}{cc}a&b\\&d\end{array}\right) \left(\begin{array}{cc}\cos\theta&\sin\theta\\-\sin\theta&\cos\theta\end{array}\right), &
    \end{align*}
    which is holomorphic in $\lambda\in\C$.
    Assume that $\myre(\lambda)>0$.
    We have $0\neq f^{(0)}\in\calH_P(\pi)$ and
    \begin{align*}
         & \left[\ell^P_P(w, \xi, \phi, \psi_\R) \circ R_{P^w|P}(\xi,\phi) f^{(0)} \right](g) \notag                                           \\
         & =\lim_{\lambda\to+0} \left[\ell^P_P(w, \xi, \phi_\lambda, \psi_\R) \circ R_{P^w|P}(\xi,\phi_\lambda) f^{(\lambda)}\right](g) \notag \\
         & =\lim_{\lambda\to+0}
        \epsilon(0,\rho_{P^w|P}^\vee\circ\phi_\lambda, \psi_\R)
        \frac{L(1,\rho_{P^w|P}^\vee\circ\phi_\lambda)} {L(0,\rho_{P^w|P}^\vee\circ\phi_\lambda)}
        \int_{N(\R)} f^{(\lambda)}(\breve{w}\inv n g) dn.
    \end{align*}

    A direct calculation implies that $\rho_{P^w|P}^\vee\circ\phi_\lambda$ is isomorphic to
    \begin{equation*}
        \tau_{(2,\lambda)} \oplus \sigma\omega_\lambda \oplus \omega_\lambda \oplus \tau_{(2,2\lambda)} \oplus \sigma\omega_{2\lambda}.
    \end{equation*}
    Therefore, we have
    \begin{align*}
        \epsilon(0,\rho_{P^w|P}^\vee\circ\phi_\lambda, \psi_\R)
        \frac{L(1,\rho_{P^w|P}^\vee\circ\phi_\lambda)} {L(0,\rho_{P^w|P}^\vee\circ\phi_\lambda)}
        =\frac{\Gamma_\C(\lambda+2)}{\Gamma_\C(\lambda+1)} \frac{\Gamma_\R(\lambda+2)}{\Gamma_\R(\lambda)}
        \frac{\Gamma_\C(2\lambda+2)}{\Gamma_\C(2\lambda+1)} \frac{\Gamma_\R(2\lambda+2)}{\Gamma_\R(2\lambda+1)}.
    \end{align*}

    We now turn to the integral.
    Let $\iota_{\SL_2}$ be an embedding of $\SL_2$ into $G=\SO(2,5)$ given by $\iota_{\SL_2}(A)=m(A,1_3)$.
    Put $w_0=\mmatrix{1}{}{}{-1_4}$, which is the Langlands-Shelstad lift in $\SO(1,4)$ given in the proof of proposition \ref{2.9.1}.
    Equations
    \begin{gather*}
        \breve{w}=\iota_2(w_0) \iota_{\SL_2}(J) \iota_2(w_0),\\
        \Ad( \iota_{\SL_2}(J) \iota_2(w_0) )\inv (n_1(x)) = n_2(-x),\\
        \Ad(\iota_2(w_0))\inv (n_c(u))=\iota_{\SL_2}(\mmatrix{1}{u}{}{1}),
    \end{gather*}
    imply that
    \begin{align*}
        \int_{N(\R)} f^{(\lambda)}(\breve{w}\inv n g) dn
         & =[\calM_2 \circ \calM_c \circ \calM_2 (f^{(\lambda)})](g),
    \end{align*}
    where we have put
    \begin{align*}
        [\calM_2 f](g)
         & =\int_{N_2(\R)} f(\iota_2(w_0)\inv n_2 g) dn_2,     \\
        [\calM_c f](g)
         & =\int_{N_c(\R)} f(\iota_{\SL_2}(J)\inv n_c g) dn_c,
    \end{align*}
    for any function $f$ on $G(\R)$.

    For $\alpha, \beta \in \C$, we define $f^{(\alpha,\beta)} \in \calH_{P_0}(|-|^\alpha\boxtimes|-|^\beta\boxtimes1)$ by
    \begin{align*}
        f^{(\alpha,\beta)}(m(A,B)n\kappa)
          & =|a|^{\alpha+\frac{5}{2}} |d|^{\beta+\frac{3}{2}} e^{2\sqrt{-1}\theta},                                                                   &
        A & =\left(\begin{array}{cc}a&b\\&d\end{array}\right) \left(\begin{array}{cc}\cos\theta&\sin\theta\\-\sin\theta&\cos\theta\end{array}\right), &
    \end{align*}
    so that $f^{(\lambda)}=f^{(\lambda+\frac{1}{2}, \lambda-\frac{1}{2})}$.
    By Lemmas \ref{sai} and \ref{ai} below, we have
    \begin{align*}
         & \calM_2 \circ \calM_c \circ \calM_2 (f^{(\lambda)})                                                                                                 \\
         & =\calM_2 \circ \calM_c \circ \calM_2 (f^{(\lambda+\frac{1}{2},\lambda-\frac{1}{2})})                                                                \\
         & =2^{-\frac{1}{2}}\frac{\Gamma_\C(\lambda-\frac{1}{2})}{\Gamma_\C(\lambda+1)}
        \calM_2 \circ \calM_c(f^{(\lambda+\frac{1}{2},-\lambda+\frac{1}{2})})                                                                                  \\
         & =-2^{-\frac{1}{2}}\frac{\Gamma_\C(\lambda-\frac{1}{2})}{\Gamma_\C(\lambda+1)}
        \frac{\Gamma_\R(2\lambda) \Gamma_\R(2\lambda+1)}{\Gamma_\R(2\lambda+3) \Gamma_\R(2\lambda-1)}  \calM_2(f^{(-\lambda+\frac{1}{2},\lambda+\frac{1}{2})}) \\
         & =-2\inv \frac{\Gamma_\C(\lambda-\frac{1}{2})}{\Gamma_\C(\lambda+1)}
        \frac{\Gamma_\R(2\lambda) \Gamma_\R(2\lambda+1)}{\Gamma_\R(2\lambda+3) \Gamma_\R(2\lambda-1)}
        \frac{\Gamma_\C(\lambda+\frac{1}{2})}{\Gamma_\C(\lambda+2)}
        f^{(-\lambda)}.
    \end{align*}
    This leads to
    \begin{align*}
         & \epsilon(0,\rho_{P^w|P}^\vee\circ\phi_\lambda, \psi_\R)
        \frac{L(1,\rho_{P^w|P}^\vee\circ\phi_\lambda)} {L(0,\rho_{P^w|P}^\vee\circ\phi_\lambda)}
        \int_{N(\R)} f^{(\lambda)}(\breve{w}\inv n g) dn           \\
         & =
        \frac{\Gamma_\C(\lambda+2)}{\Gamma_\C(\lambda+1)} \frac{\Gamma_\R(\lambda+2)}{\Gamma_\R(\lambda)}
        \frac{\Gamma_\C(2\lambda+2)}{\Gamma_\C(2\lambda+1)} \frac{\Gamma_\R(2\lambda+2)}{\Gamma_\R(2\lambda+1)}
        \cdot \frac{-1}{2} \frac{\Gamma_\C(\lambda-\frac{1}{2})}{\Gamma_\C(\lambda+1)}
        \frac{\Gamma_\R(2\lambda) \Gamma_\R(2\lambda+1)}{\Gamma_\R(2\lambda+3) \Gamma_\R(2\lambda-1)}
        \frac{\Gamma_\C(\lambda+\frac{1}{2})}{\Gamma_\C(\lambda+2)}
        f^{(-\lambda)}(g),
    \end{align*}
    whose limit as $\lambda$ approaching $0$ from the right is $-f^{(\lambda)}$.
    Since $f^{(0)}\neq0$, this completes the proof.
\end{proof}

In order to finish the proof of Proposition \ref{2.9.2}, it remains to show the following two lemmas.
\begin{lemma}\label{sai}
    Suppose that $\myre(\beta)>0$.
    Then
    \begin{align*}
        \calM_2 f^{(\alpha,\beta)}
         & =2^{-\frac{1}{2}}\frac{\Gamma_\C(\beta)}{\Gamma_\C(\beta+\frac{3}{2})} f^{(\alpha,-\beta)}.
    \end{align*}
\end{lemma}
\begin{proof}
    As in the proof of Lemma \ref{so14}, one has
    \begin{equation*}
        \calM_2 f^{(\alpha,\beta)}=\calM_2 f^{(\alpha,\beta)}(1) \cdot f^{(\alpha,-\beta)},
    \end{equation*}
    if $\calM_2 f^{(\alpha,\beta)}(1)$ converges absolutely.
    Since $f^{(\alpha,\beta)}\circ\iota_2$ is equal to $\varphi^{(\beta)}$ defined in the previous subsection, we have
    \begin{align*}
        [\calM_2 f^{(\alpha,\beta)}](1)
         & =\int_{x\in\R^3} f^{(\alpha,\beta)}(\iota_2(w_0)\inv \iota_2(n(x))) d(n(x)) \\
         & =\int_{x\in\R^3} \varphi^{(\beta)}(w_0\inv n(x)) d(n(x))                    \\
         & =[\calM_{\SO(1,4)}\varphi^{(\beta)}](1).
    \end{align*}
    Hence the assertion follows from Lemma \ref{so14}.
\end{proof}
\begin{lemma}\label{ai}
    Suppose that $\myre(\alpha-\beta)>0$.
    Then
    \begin{align*}
        \calM_c f^{(\alpha,\beta)}
         & =-\frac{\Gamma_\R(\alpha-\beta) \Gamma_\R(\alpha-\beta+1)}{\Gamma_\R(\alpha-\beta+3) \Gamma_\R(\alpha-\beta-1)} f^{(\beta,\alpha)}.
    \end{align*}
\end{lemma}
\begin{proof}
    As in the proof of Lemma \ref{so14}, one has
    \begin{equation*}
        \calM_c f^{(\alpha,\beta)}
        =\calM_c f^{(\alpha,\beta)}(1) \cdot f^{(\beta,\alpha)},
    \end{equation*}
    if $\calM_c f^{(\alpha,\beta)}(1)$ converges absolutely.
    Following the proof of \cite[Lemma 1.4]{ike}, for $s\in \C$ we define a function $h^{(s)}$ on $\SL_2(\R)$ by
    \begin{align*}
        h^{(s)}\left(
        \left(\begin{array}{cc}a&b\\&a\inv\end{array}\right) \left(\begin{array}{cc}\cos\theta&\sin\theta\\-\sin\theta&\cos\theta\end{array}\right)
        \right)
        =\left| \frac{a}{d} \right|^{s+1} e^{2\sqrt{-1}\theta}.
    \end{align*}
    Then $f^{(\alpha,\beta)}\circ\iota_{\SL_2}$ is equal to $h^{(\alpha-\beta)}$, and hence we have
    \begin{align*}
        [\calM_c f^{(\alpha,\beta)}](1)
         & =\int_{u\in\R} f^{(\alpha,\beta)} \left(\iota_{\SL_2}(J)\inv
        \iota_{\SL_2}\left(\begin{array}{cc}1&u\\ &1\end{array}\right)\right) du                                      \\
         & =\int_{u\in\R} h^{(\alpha-\beta)} \left(J\inv \left(\begin{array}{cc}1&u\\ &1\end{array}\right)\right) du.
    \end{align*}
    Hence the assertion follows from the second equation in the proof of \cite[Lemma 1.4]{ike}.
\end{proof}

\section{The decomposition into near equivalence classes and the standard model}\label{3}
In this section we shall roughly recall from \cite{kmsw,art13} the stable multiplicity formula, which implies the decomposition of $L^2_\disc(G(F)\backslash G(\A_F))$ into near equivalence classes, the global intertwining relation, and its weaker identity.

\subsection{Stable multiplicity formula}\label{3.1}
Let $F$ be a number field, $G^*$ the split special orthogonal group $\SO_{2n+1}$ over $F$, and $G$ an inner form of $G^*$.
Although $G$ does not have simply connected derived subgroup, every endoscopic data for $G$ comes from an endoscopic triple.
Hence the argument in \cite[\S\S3.1-3.2]{kmsw} also holds for our $G$.
Thus we have the decompositions
\begin{align*}
    L^2_{\disc}(G(F)\backslash G(\A_F)) & =\bigoplus_{\substack{c\in\calC(G) \\ t\geq0}} L^2_{\disc,t,c}(G(F)\backslash G(\A_F)),\\
    \tr R_{\disc}(f)                    & =\sum_{\substack{c\in\calC(G)      \\ t\geq0}} R_{\disc,t,c}(f),
\end{align*}
for $f\in\calH(G)$, where $R_\diamondsuit$ denotes the regular representation of $G(\A_F)$ on $L^2_\diamondsuit$.
Moreover, for $t\in\R_{\geq0}$, $c\in\calC(G)$, and $\frake=(G^\frake,s^\frake,\eta^\frake)\in\overline{\calE}_\el(G)$, we have
\begin{itemize}
    \item the $c$-variant $I^G_{\disc,t,c}$ of the discrete part $I^G_{\disc,t}$ of the trace formula on the constituents of which the norm of the imaginary part of the infinitesimal character is $t$, where the central character datum is trivial;
    \item the transfer mapping $\calH(G)\to\calS(G^\frake)$, $f\mapsto f^\frake=f^{G^\frake}$, where $\calS(G^\frake)$ denotes a space of functions on stable conjugacy classes of semisimple elements defined in \cite[p.53, p.132]{art13};
    \item the stable linear form $S^\frake_{\disc,t,c}=S^{G^\frake}_{\disc,t,c}$ on $\calH(G^\frake)$ and the associated linear form $\widehat{S}^\frake_{\disc,t,c}$ on $\calS(G^\frake)$ (see \cite[(2.1.2)]{art13} for the notion of associated linear forms),
\end{itemize}
and the stabilization
\begin{align*}
    I^G_{\disc,t,c}(f) & =\sum_{\frake\in\overline{\calE}_\el(G)} \iota(G,G^\frake) \widehat{S}^\frake_{\disc,t,c}(f^\frake), & f & \in\calH(G). &
\end{align*}
Here $\iota(G,G^\frake)$ are the global coefficients introduced by Kottwitz and Shelstad.
See \cite[(3.2.4)]{art13} for an explicit formula for them.

Let $\psi\in\Psi(G^*)$ and $\frake=(G^\frake,s^\frake,\eta^\frake)\in\overline{\calE}_\el(G)$.
Note that $c(\psi)=c(\pi_\psi)\in\calC(\GL_{2n})$ belongs to $\calC(G)$.
We shall write $t(\psi)$ for the norm of the imaginary part of the infinitesimal character of $\pi_\psi$.
As in \cite[\S3.3]{kmsw}, put
\begin{align*}
    I^G_{\disc,\psi}                         & =I^G_{\disc,t(\psi),c(\psi)},                         \\
    S^{G^\frake}_{\disc,\psi}                & =S^{G^\frake}_{\disc,t(\psi),c(\psi)},                \\
    L^2_{\disc,\psi}(G(F)\backslash G(\A_F)) & =L^2_{\disc,t(\psi),c(\psi)}(G(F)\backslash G(\A_F)), \\
    R_{\disc,\psi}                           & =R_{\disc,t(\psi),c(\psi)}.
\end{align*}
We also write $R^G_{\disc,\psi}$ if the group $G$ is to be emphasized.
Let us recall from \cite[Theorem 4.1.2]{art13} the stable trace formula for the split odd special orthogonal group $G^*=\SO_{2n+1}$:
\begin{proposition}[Stable trace formula]\label{3.3.1}
    Let $\psi\in\widetilde{\Psi}(N)$. Then we have
    \begin{align*}
        S^{G^*}_{\disc,\psi}(f)=\begin{dcases*}
                                    |\frakS_\psi|\inv \varepsilon^{G^*}_\psi(s_\psi) \sigma(\overline{S}_\psi^\circ) f^{G^*}(\psi), & if $\psi\in\Psi(G^*)$,                             \\
                                    0,                                                                                              & if $\psi\in\widetilde{\Psi}(N)\setminus\Psi(G^*)$,
                                \end{dcases*}
    \end{align*}
    for $f\in\calH(G^*)$, where $\varepsilon^{G^*}_\psi$ is the character in the multiplicity formula, $\sigma(\overline{S}_\psi^\circ)$ the constant given in \cite[Proposition 4.1.1]{art13}, and $f^{G^*}(\psi)$ denotes the linear form defined by \cite[(4.1.3)]{art13}.
\end{proposition}
As a consequence we have a decomposition
\begin{equation}\label{gap}
    \begin{gathered}
        L^2_{\disc}(G(F)\backslash G(\A_F))=\bigoplus_{\psi\in \Psi(G^*)} L^2_{\disc,\psi}(G(F)\backslash G(\A_F)),\\
        \tr R_{\disc}(f)=\sum_{\psi\in\Psi(G^*)} R_{\disc,\psi}(f), \qquad f\in\calH(G),
    \end{gathered}
\end{equation}
of the discrete spectrum.
Theorem \ref{3.12} is now proven.

\subsection{Global intertwining operator}\label{3.4}
Let $\xi : G^* \to G$ be an inner twist.
Let $P^*\subset G^*$ be a standard parabolic subgroup with a Levi decomposition $P^*=M^*N^*$ over $F$, and put $P=\xi(P^*)$, $M=\xi(M^*)$, and $N=\xi(N^*)$.
We consider the case that $P$, $M$, and $N$ are defined over $F$, and a restriction $\xi|_{M^*}:M^*\to M$ is an inner twist.

Assume that $M\neq G$.
Let $(\pi, V_\pi)$ be an irreducible component of $L^2_{\disc}(A_{M,\infty}^+ M(F) \backslash M(\A_F))$, where $A_{M,\infty}^+$ denotes the connected component of 1 in $\Res_{F/\Q}(M)(\R)$.
Let $\psi \in \Psi_2(M^*)$ be the corresponding $A$-parameter (given by the induction hypothesis).
We shall consider the induced representation $(\calI_P(\pi), \calH_P(\pi))$.
It is a right regular representation on the Hilbert space $\calH_P(\pi)$ of measurable functions $f: G(\A_F)\to V_\pi$ such that $f(nmg)=\delta_P^\frac{1}{2}(m) \pi(m) f(g)$ for any $n\in N(\A_F)$, $m\in M(\A_F)$, and $g \in G(\A_F)$ whose restriction to an open compact subgroup $K\subset G(\A_F)$ is square-integrable.

Let $P'\subset G$ be a parabolic subgroup over $F$ with Levi component $M$.
For $\lambda \in \fraka_{M,\C}^*$, the intertwining operator $J_{P'|P} : \calI_P(\pi_\lambda)\to\calI_{P'}(\pi_\lambda)$ is defined by
\begin{equation*}
    [J_{P'|P}(\pi_\lambda)(f)](g)=\int_{N(\A_F)\cap N'(\A_F)\backslash N'(\A_F)} f(n'g) dn',
\end{equation*}
where $N'$ denotes the unipotent radical of $P'$, and $\pi_\lambda=\pi\otimes\lambda$.
As in \cite{kmsw}, we take the Haar measure $dn'$ determined by the Haar measure on $\A_F$ which assigns the quotient $\A_F/F$ volume 1.
It is known that the integral converges absolutely when the real part of $\lambda$ lies in a certain open cone.
As a function of $\lambda$, it has a meromorphic continuation and is nonzero and holomorphic at $\lambda=0$.
Then the operator $J_{P'|P}(\pi)=J_{P'|P}(\pi_\lambda)|_{\lambda=0}$ is defined.
It is also known that $J_{P'|P}(\pi)$ is unitary.

We now consider the case $P'=P^w=w\inv Pw$ for some $w\in W(M,G)^\Gamma$.
Let $\breve{w}\in N(M,G)(F)$ be a representative of $w$.
Then one has a twist $(\breve{w}\pi,V_\pi)$.
We define two intertwining operators $\ell(\breve{w})$ and $C_{\breve{w}}$ by
\begin{align*}
    \ell(\breve{w}) & :\calH_{P^w}(\pi) \lra \calH_P(\breve{w}\pi), & [\ell(\breve{w})f](g)     & =f(\breve{w}\inv g),                 & \\
    C_{\breve{w}}   & :(\breve{w}\pi, V_\pi) \lra (\pi, V_\pi),     & [C_{\breve{w}}\varphi](m) & =\varphi(\breve{w}\inv m \breve{w}). &
\end{align*}
The composition $\calI_P(C_{\breve{w}})\circ\ell(\breve{w})\circ J_{P^w|P}(\pi)$ is then a self-intertwining operator of the induced representation $(\calI_P(\pi), \calH_P(\pi))$, and independent of the choice of $\breve{w}$.
We put
\begin{equation*}
    M_P(w,\pi)=\calI_P(C_{\breve{w}})\circ\ell(\breve{w})\circ J_{P^w|P}(\pi).
\end{equation*}

Let $\rho_{P'|P}$ be the adjoint representation of $\widehat{M}$ on $\widehat{\frakn}\cap\widehat{\frakn}' \backslash \widehat{\frakn}' \cong \widehat{\overline{\frakn}}\cap\widehat{\frakn'}$, where $\widehat{\frakn}$, $\widehat{\frakn'}$, and $\widehat{\overline{\frakn}}$ are the Lie algebras of $\widehat{N}$, $\widehat{N'}$, and $\widehat{\overline{N}}$, respectively, where $\overline{N}$ denotes the unipotent radical of the opposite parabolic subgroup $\overline{P}$ of $P$ containing $M$.
Following \cite{kmsw}, we define normalizing factors
\begin{align*}
    r_{P'|P}(\psi)
     & =\frac{L(0,\psi,\rho_{P'|P}^\vee)}{L(1,\psi,\rho_{P'|P}^\vee)}\frac{\epsilon(\frac{1}{2},\psi,\rho_{P'|P}^\vee)}{\epsilon(0,\psi,\rho_{P'|P}^\vee)}, \\
    r_P(w,\psi)
     & =r_{P^w|P}(\psi)\epsilon(\frac{1}{2},\psi,\rho_{P^w|P}^\vee)\inv,
\end{align*}
and normalized intertwining operators
\begin{align*}
    R_{P'|P}(\pi,\psi) & =r_{P'|P}(\psi)\inv J_{P'|P}(\pi), \\
    R_P(w,\pi,\psi)    & =r_P(w,\psi)\inv M_P(w,\pi),
\end{align*}
where $L$- and $\epsilon$-factors are the automorphic ones.

\begin{lemma}\label{2.2.2}
    We have
    \begin{equation*}
        R_{P'|P}(\pi_\lambda,\psi_\lambda)=\bigotimes_v R_{P'_v|P_v}(\pi_{\lambda,v},\psi_{\lambda,v}).
    \end{equation*}
\end{lemma}
\begin{proof}
    The proof is similar to that of \cite[Lemma 2.2.2]{kmsw}.
\end{proof}

\begin{lemma}\label{giom}
    Let $P''\subset G$ be a parabolic subgroup over $F$ with Levi component $M$.
    We have
    \begin{equation*}
        R_{P''|P}(\pi_\lambda,\psi_\lambda)=R_{P''|P'}(\pi_\lambda,\psi_\lambda) \circ R_{P'|P}(\pi_\lambda,\psi_\lambda).
    \end{equation*}
\end{lemma}
\begin{proof}
    The unnormalized intertwining operator $J_{P'|P}(\pi_\lambda)$ satisfies the multiplicativity property.
    It is known that at each place $v$ the local factors of automorphic $L$- and $\epsilon$-factors are equal to the Artin $L$- and $\epsilon$-factors.
    Thus the normalizing factor has a decomposition
    \begin{equation*}
        r_{P'|P}(\psi_\lambda)=\prod_v r_{P'_v|P_v}(\xi_v, \psi_v, \psi_{F,v}).
    \end{equation*}
    Since every local factor $r_{P'_v|P_v}(\xi_v, \psi_v, \psi_{F,v})$ has the multiplicativity property, so is $r_{P'|P}(\psi_\lambda)$.
    This completes the proof.
\end{proof}

\subsection{The global intertwining relation}\label{3.5}
In this subsection, we define two global linear forms and state the global intertwining relation, which we shall call GIR for short.
This subsection can be regarded as the global analogue of \S\ref{2.6}.

Let $\psi_{M^*}\in \Psi_2(M^*)$, and $\psi\in\Psi(G^*)$ be its image.
Note the change of notation from the previous subsection that we write $\psi_{M^*}$ for a parameter for $M^*$ rather than $\psi$.
Let $\psi_F:\A_F/F\to\C^1$ be a nontrivial additive character.
For $\pi_M=\bigotimes_v \pi_{M,v}\in\Pi_{\psi_{M^*}}(M)$ and $u^\natural\in \frakN_\psi(M,G)$ we define a global intertwining operator
\begin{align*}
    R_P(u^\natural,\pi_M,\psi_{M^*},\psi_F)=\bigotimes_v R_{P_v}(u^\natural_v,\pi_{M,v},\psi_{M^*,v},\psi_{F,v}),
\end{align*}
where $u^\natural_v\in\frakN_{\psi_v}(M_v,G_v)$, $\psi_{M^*,v}\in\Psi^+(M^*)$, and $\psi_{F,v}$ are the localizations of $u^\natural$, $\psi_{M^*}$, and $\psi_F$ respectively.
\begin{proposition}\label{3.5.3}
    Assume that $\pi_M\in\Pi_{\psi_{M^*}}(M)$ is automorphic, i.e., $\an{-,\pi_M}=\varepsilon_{\psi_{M^*}}$.
    Then for any $u^\natural\in\frakN_\psi(M,G)$ we have an equation
    \begin{equation*}
        \bigotimes_v \pi_{M,v}(u^\natural_v) = \varepsilon_{\psi_{M^*}}(u^\natural) C_{\breve{w}},
    \end{equation*}
    of isomorphisms $(\breve{w}\pi_M,V_{\pi_M}) \to (\pi_M,V_{\pi_M})$, where $w$ is the image of $u^\natural$ in $W_\psi(M,G)$.
\end{proposition}
\begin{proof}
    The proof is similar to that of Proposition 3.5.3 of \cite{kmsw}.
\end{proof}
\begin{proposition}\label{3.5.4}
    Let $u^\natural\in\frakN_\psi(M,G)$ and $\pi_M\in\Pi_{\psi_{M^*}}(M)$. Then
    \begin{enumerate}
        \item $R_P(yu^\natural,\pi_M,\psi_{M^*},\psi_F)=\an{\overline{y},\pi_M}R_P(u^\natural,\pi_M,\psi_{M^*},\psi_F)$ for any $y\in\frakS_\psi(M)$, where $\overline{y}$ is the image of $y$ in $\overline{\frakS}_\psi(M)$.
        \item $R_P(w_u,\pi_M,\psi_{M^*})=\varepsilon_{\psi_{M^*}}(u^\natural)R_P(u^\natural,\pi_M,\psi_{M^*},\psi_F)$ if $\pi_M$ is automorphic.
    \end{enumerate}
\end{proposition}
\begin{proof}
    The first assertion follows from Lemma \ref{2.5.2}. The other one follows from Lemma \ref{2.2.2} and Proposition \ref{3.5.3}.
\end{proof}

The first linear form $\calH(G)\ni f\mapsto f_G(\psi_{M^*},u^\natural)$ is defined by
\begin{align*}
    f_G(\psi_{M^*},u^\natural) & =\prod_v f_{v,G_v}(\psi_{M^*,v},u^\natural_v)                                                                         &  &                                  & \\
                               & =\sum_{\pi_M \in \Pi_{\psi_{M^*}}(M)} \tr\left(R_P(u^\natural,\pi_M,\psi_{M^*},\psi_F) \circ \calI_P(\pi_M,f)\right), &  & f=\bigotimes_v f_v \in \calH(G). &
\end{align*}
Note that the parameter $\psi_{M^*}$ is relevant since it is discrete and $M=\xi(M^*)$ is defined over $F$.
Put $f_G(\psi_{M^*},u^\natural)$ to be $0$ for $M^*\subset G^*$ which does not transfer to $G$.
Proposition \ref{3.5.4} implies the following lemma.
\begin{lemma}\label{3.5.5}
    The linear form $f_G(\psi_{M^*},u^\natural)$ depends only on the image of $u^\natural$ in $\overline{\frakN}_\psi(M,G)$.
\end{lemma}
Thus we also write $f_G(\psi_{M^*},\overline{u})$ for $f_G(\psi_{M^*},u^\natural)$, where $\overline{u}\in \overline{\frakN}_\psi(M,G)$ is the image of $u^\natural$.

Next let us recall the definition of the second linear form.
For a parameter $\psi\in \Psi(G^*)$ and a semisimple element $s\in S_\psi$, Lemma \ref{1.4.3} attaches an endoscopic triple $\frake=(G^\frake,s^\frake,\eta^\frake)\in\calE(G)$ and a parameter $\psi^\frake\in \Psi(G^\frake)$, where $\psi^\frake$ is not unique (determined up to $\Out_G(G^\frake)$-action).
The second linear form $\calH(G)\ni f\mapsto f'_G(\psi,s)$ is defined by
\begin{align*}
    f'_G(\psi,s) & =f^\frake(\psi^\frake)                                               &  &                                  & \\
                 & =\prod_v f^\frake_v(\psi^\frake_v) = \prod_v f'_{v,G_v}(\psi_v,s_v), &  & f=\bigotimes_v f_v \in \calH(G), &
\end{align*}
where $s_v\in S_{\psi_v}$ denotes the localization of $s$.
If $\psi$ is the image of $\psi_{M^*}\in\Psi_2(M^*)$, we also write $f'_G(\psi_{M^*},s)$ for $f'_G(\psi,s)$.
\begin{lemma}\label{3.5.6}
    The linear form $f'_G(\psi,s)$ depends only on the image of $s$ in $\overline{\frakS}_\psi$, and hence on the image of $\frake$ in $\overline{\calE}(G)$.
\end{lemma}
\begin{proof}
    By Lemma \ref{2.6.1}, we have $f'_G(\psi,s)=f'_G(\psi,ss_0)$ for any $s_0\in S_\psi^\circ$.
    By Lemma \ref{1.1.2} and the product formula \eqref{033}, we have $f'_G(\psi,s)=f'_G(\psi,sx)$ for any $x\in Z(\widehat{G})^\Gamma$.
\end{proof}

The following theorem is called the global intertwining relation (GIR), which will follow from LIR.
\begin{theorem*}[Global intertwining relation]\label{3.5.7}
    If $\overline{u}\in\overline{\frakN}_\psi(M,G)$ and $s\in S_{\psi,\semisimple}$ have the same image in $\overline{\frakS}_\psi(G)$, then we have
    \begin{align*}
        f'_G(\psi,s_\psi s\inv)=f_G(\psi_{M^*},\overline{u}).
    \end{align*}
\end{theorem*}

\subsection{Reduction of GIR to elliptic or exceptional (relative to $G$) parameters}\label{3.55}
This subsection is the global version of \S\ref{2.8}.
As in the local case, one can see that $\Psi_\el^2(G^*)$ is the set of the equivalence classes of parameters $\psi$ of the form
\begin{equation}\label{glel}
    \psi=2\psi_1\boxplus\cdots2\psi_q\boxplus\psi_{q+1}\boxplus\cdots\boxplus\psi_r,
\end{equation}
where $\psi_1,\ldots,\psi_r$ are simple, symplectic, mutually distinct, and $r\geq q\geq1$.
Then we have
\begin{align*}
    S_\psi \cong \Or(2,\C)^q \times \Or(1,\C)^{r-q}.
\end{align*}
Let now $\Psi_{\exc1}(G^*)$ and $\Psi_{\exc2}(G^*)$ be the subsets of $\Psi(G^*)$ consisting of the parameters $\psi$ of the form
\begin{itemize}
    \item[$(\exc1)$]$\psi=2\psi_1\boxplus\psi_2\boxplus\cdots\boxplus\psi_r,$ where $\psi_1$ is simple and orthogonal, and $\psi_2,\ldots,\psi_r$ are simple, symplectic, and mutually distinct,
    \item[$(\exc2)$]$\psi=3\psi_1\boxplus\psi_2\boxplus\cdots\boxplus\psi_r,$ where $\psi_1,\ldots,\psi_r$ are simple, symplectic, and mutually distinct,
\end{itemize}
respectively.
They are disjoint.
We then have
\begin{itemize}
    \item[$(\exc1)$]$S_\psi \cong \Sp(2,\C) \times \Or(1,\C)^{r-1},$
    \item[$(\exc2)$]$S_\psi \cong \Or(3,\C) \times \Or(1,\C)^{r-1},$
\end{itemize}
respectively.
We put $\Psi_\exc(G^*)=\Psi_{\exc1}(G^*)\sqcup\Psi_{\exc2}(G^*)$, and we shall say that $\psi$ is exceptional (resp. of type ($\exc1$), resp. of type ($\exc2$)) if $\psi\in\Psi_\exc(G^*)$ (resp. $\psi\in\Psi_{\exc1}(G^*)$, resp. $\psi\in\Psi_{\exc2}(G^*)$).
One can see that $\Psi_\exc(G^*)$ and $\Psi_\el(G^*)$ are disjoint.
Put $\Psi^{\el,\exc}(G^*)=\Psi(G^*)\setminus (\Psi_\el(G^*)\sqcup\Psi_\exc(G^*))$.
Put also $\Phi_\heartsuit(G^*)=\Psi_\heartsuit(G^*)\cap\Phi(G^*)$ for $\heartsuit\in\{\exc, \exc1, \exc2\}$, and $\Phi^{\el,\exc}(G^*)=\Psi^{\el,\exc}(G^*)\cap\Phi(G^*)$.

Let $\psi_{M^*}\in \Psi_2(M^*)$ and $\psi\in\Psi(G^*)$ be parameters such that $\psi$ is the image of $\psi_{M^*}$.
For $x\in\frakS_\psi$, we define $T_\psi$, $B_\psi$, $s_x\in S_\psi$, and $T_{\psi,x}$ in the same way as in \S\ref{2.8}.
Moreover, let $\overline{T}_\psi$, $\overline{B}_\psi$, $\overline{s}_x$, and $\overline{T}_{\psi,x}$ be their images in $\overline{S}_\psi$.
Then it can be seen that $\overline{s}_x$ and $\overline{T}_{\psi,x}$ are determined by the image of $x$ in $\overline{\frakS}_\psi$, and hence they are well-defined for $x\in\overline{\frakS}_\psi$.
Define subsets $\frakS_{\psi,\el}\subset\frakS_\psi$ and $\overline{\frakS}_{\psi,\el}\subset\overline{\frakS}_\psi$ as in the local case.
The following four lemmas can be proved in the same way as Lemmas \ref{2.8.4}, \ref{2.8.6}, \ref{2.8.7}, and \ref{2.8.9} were, respectively.
\begin{lemma}\label{3.5.9}
    Suppose that $\psi\in\Psi^{\el,\exc}(G^*)$.
    Then
    \begin{enumerate}[(1)]
        \item every simple reflection $w\in W_\psi^\circ(M^*,G^*)$ centralizes a torus of positive dimension in $\overline{T}_\psi$ and
        \item $\dim \overline{T}_{\psi,x}\geq1$ for all $x\in \overline{\frakS}_\psi$.
    \end{enumerate}
\end{lemma}
\begin{lemma}\label{3.5.10}
    If $\psi\in\Psi_{\exc}(G^*)$, then $\frakS_{\psi,\el}=\frakS_\psi$ and $\overline{\frakS}_{\psi,\el}=\overline{\frakS}_\psi$.
\end{lemma}
Let $\xi:G^*\to G$ be an inner twist, and assume that $M^*\subsetneq G^*$ is a proper standard Levi subgroup such that $M=\xi(M^*)\subsetneq G$ is a proper Levi subgroup defined over $F$.
\begin{lemma}\label{3.5.11}
    Let $x\in \overline{\frakS}_\psi(M,G)$.
    Assume that either
    \begin{enumerate}[(1)]
        \item $\psi$ is elliptic, or
        \item every simple reflection $w\in W_\psi^\circ(M^*,G^*)$ centralizes a torus of positive dimension in $\overline{T}_\psi$.
    \end{enumerate}
    Then $f_G(\psi_{M^*},\overline{u})$ is the same for every $\overline{u}\in\overline{\frakN}_\psi(M,G)$ mapping to $x$.
\end{lemma}
\begin{lemma}\label{3.5.12}
    Let $x\in \overline{\frakS}_\psi$.
    We have $f_G(\psi_{M^*},\overline{u})=f'_G(\psi,s_\psi s\inv)$ whenever $\overline{u}\in\overline{\frakN}_\psi(M,G)$ and $s\in S_{\psi,\semisimple}$ map to $x$ unless
    \begin{enumerate}[(1)]
        \item $\psi$ is elliptic and $x\in\overline{\frakS}_{\psi,\el}$, or
        \item $\psi\in\Psi_\exc(G^*)$.
    \end{enumerate}
\end{lemma}

Lemmas \ref{3.5.11} and \ref{3.5.12} imply that $f_G(\psi_{M^*},\overline{u})$ depends only on the image of $\overline{u}$ in $\overline{\frakS}_\psi(M,G)=\overline{\frakS}_\psi(G)$, unless $\psi$ is exceptional.
Thus $f_G(\psi_{M^*},x)$ is well-defined for $x\in \overline{\frakS}_\psi(M,G)$.
On the other hand, when $\psi$ is exceptional, we can define $f_G(\psi_{M^*},x)$ for $x\in \overline{\frakS}_\psi(M,G)$ similarly to local setting in the end of \S\ref{2.8}.
Now we have defined $f_G(\psi_{M^*},x)$ for $x\in\overline{\frakS}_\psi(M,G)$.

\subsection{The weaker identities}\label{3.7}
In the previous subsection, GIR for $\psi\in\Psi^{\el,\exc}(G^*)$ was deduced from the induction hypothesis (Lemma \ref{3.5.12}).
Now we shall see two weaker identities.
We omit Arthur's procedure named the standard model for $G$, since it is very similar to that for other classical groups explained in \cite[\S4]{art13} and \cite[\S3.6]{kmsw}.
We continue to be in the setup of the previous subsection, but do not assume that $M=\xi(M^*)$ is defined over $F$.
\begin{lemma}\label{3.7.1}
    Suppose that $\psi\in\Psi_\exc(G^*)$.
    Let $x\in\overline{\frakS}_\psi$ and $f\in\calH(G)$.
    If $M^*$ does not transfer to $G$, then we have
    \begin{align*}
        f'_G(\psi,s_\psi x\inv)=f_G(\psi,x)=0.
    \end{align*}
    If $M^*$ transfers to $G$, then we have
    \begin{align*}
        \sum_{x\in\overline{\frakS}_\psi} \varepsilon^{G^*}_\psi(x) \left(f'_G(\psi,s_\psi x\inv)-f_G(\psi,x)\right)=0,
    \end{align*}
    and $R_P(w,\pi_M,\psi_{M^*})=1$ for $w\in W_\psi$.
\end{lemma}
\begin{proof}
    The proof is similar to that of Lemma 3.7.1 of \cite{kmsw}.
\end{proof}
\begin{lemma}\label{3.8.1}
    Suppose that $\psi\in\Psi_\el^2(G^*)$ is of the form \eqref{glel}.
    Let $f\in\calH(G)$.
    Then we have
    \begin{align*}
        \tr\left(R^G_{\disc,\psi}(f)\right)=2^{-q} \abs{\overline{\frakS}_\psi}\inv \sum_{x\in\overline{\frakS}_{\psi,\el}} \varepsilon^{G^*}_\psi(x) \left(f'_G(\psi,s_\psi x\inv)-f_G(\psi,x)\right).
    \end{align*}
\end{lemma}
\begin{proof}
    The proof is similar to that of Lemmas 3.8.1 of \cite{kmsw}.
\end{proof}

\section{Globalizations and the proof of local classification}\label{4}
In this section, we will finish the proof of LIR and ECR for generic parameters, as well as the proof of LLC.
The argument is similar to that of \cite[\S4]{kmsw}.
The primary difference is the globalization of parameters, which is caused by the difference between unitary groups and orthogonal groups.
\subsection{Globalizations of fields, groups, and representations}\label{4.1}
First recall from \cite[Lemma 6.2.1]{art13} the following lemma:
\begin{lemma}\label{art6.2.1}
    Let $F$ be a local field other than $\C$, and $r_0$ a positive integer.
    Then we can find a totally real number field $\dot{F}$ and a place $u$ of $\dot{F}$ such that $\dot{F}_u$ is isomorphic to $F$ and $\dot{F}$ has at least $r_0$ real places.
\end{lemma}
As a corollary, we have the following.
\begin{lemma}\label{4.1.1}
    Let $F$ be a local field other than $\C$, and $r_0$ a positive integer.
    Then we can find a totally real number field $\dot{F}$ and places $u_1$ and $u_2$ of $\dot{F}$ such that $\dot{F}_u$ is isomorphic to $F$ for $u=u_1, u_2$, and that $\dot{F}$ has at least $r_0$ real places.
\end{lemma}
\begin{proof}
    Let $\dot{F}$ be as in Lemma \ref{art6.2.1}.
    It can be easily seen that there exists an element $\alpha\in\dot{F}^\times$ that is not square in $\dot{F}$, totally positive, and square in $\dot{F}_u=F$.
    Then $\dot{F}(\sqrt{\alpha})$ is a number field what we want.
\end{proof}

The classification of quadratic forms or the exact sequence \eqref{033} leads to the following lemma of special orthogonal groups:
\begin{lemma}\label{4.1.2}
    Let $F$ be a local field, and $G$ an inner form of $G^*=\SO_{2n+1}$ over $F$.
    Let $\dot{F}$ be a number field with a place $u$ such that $\dot{F}_u=F$.
    Assume that $\dot{F}$ is totally real unless $F=\C$.
    Then for any place $v_2$ of $\dot{F}$ other than $u$, there exists an inner form $\dot{G}$ of $\dot{G}^*=\SO_{2n+1}$ over $\dot{F}$ with following properties:
    \begin{enumerate}
        \item $\dot{G}_u=G$;
        \item $\dot{G}_v$ is split over $\dot{F}_v$ for any place $v$ of $\dot{F}$ except $u$ and $v_2$;
        \item if $G$ is non-quasi-split and $v_2$ is a finite (resp. real) place, then $\dot{G}_{v_2}$ is the unique non-quasi-split inner form (resp. isomorphic to $\SO(n-1,n+2)$).
    \end{enumerate}
    In place of the third property, we can also choose $\dot{G}$ with the properties $1$, $2$, and
    \begin{itemize}
        \item[3'] if $\dot{F}_{v_2}$ is isomorphic to $\dot{F}_u=F$, then $\dot{G}_{v_2}$ is isomorphic to $\dot{G}_u=G$.
    \end{itemize}
\end{lemma}

We have the following globalization results of discrete series representations.
\begin{lemma}\label{4.2.1}
    Let $\dot{F}$ be a totally real number field, and $\dot{G}$ a simple twisted endoscopic group of $\GL_N$ or an inner form of $\dot{G}^*=\SO_{2n+1}$ over $\dot{F}$.
    Let $V$ be a finite set of places of $\dot{F}$ such that at least one real place $v_\infty$ is not contained in $V$.
    For all $v\in V$, let $\pi_v\in\Pi_{2,\temp}(\dot{G}_v)$, and let $\pi_{v_\infty}$ be a discrete series representation with a sufficiently regular infinitesimal character.
    Then there exists a cuspidal automorphic representation $\dot{\pi}$ of $\dot{G}(\A_{\dot{F}})$ such that $\dot{\pi}_v=\pi_v$ for all $v\in V$ and $\dot{\pi}_{v_\infty}=\pi_{v_\infty}$.
\end{lemma}
\begin{proof}
    The proof is the same as that of \cite[Lemma 4.2.1]{kmsw}.
\end{proof}
\begin{lemma}\label{4.2.2}
    Let $\dot{F}$ be a totally real number field, and $\dot{G}$ a simple twisted endoscopic group of $\GL_N$.
    Let $V$ be a finite set of places of $\dot{F}$ such that at least one real place $v_\infty$ is not contained in $V$.
    For all $v\in V$, let $M_v \subset \dot{G}_v$ be a Levi subgroup such that $M_v=\dot{G}_v$ if $v$ is a real place.
    For each $v\in V$, let $\pi_{M_v}\in \Pi_{2,\temp}(M_v)$.
    Let $\pi_{v_\infty}$ be a discrete series representation with a sufficiently regular infinitesimal character.
    Then there exists a cuspidal automorphic representation $\dot{\pi}$ of $\dot{G}(\A_{\dot{F}})$ such that for all $v\in V$, if $M_v=\dot{G}_v$ then $\dot{\pi}_v=\pi_{M_v}$ and if $M_v\neq \dot{G}_v$ then $\dot{\pi}_v$ is an irreducible subquotient of the induced representation $\calI_{P_v}^{\dot{G}_v} (\pi_{M_v}\otimes \chi_v)$ for some unramified unitary character $\chi_v\in \Psi(M_v)$.
\end{lemma}
\begin{proof}
    The proof is the same as that of \cite[Lemma 4.2.2]{kmsw}.
\end{proof}
The following lemma is used in the proof of Lemma \ref{2.2.3}.
Note that the proof of the following lemma is independent of \S\ref{2}.
\begin{lemma}\label{4.2.3}
    Let $F$ be a $p$-adic field, $\xi : G^* \to G$ an inner twist of $G^*=\SO_{2n+1}$ over $F$, $M^*\subset G^*$ a standard Levi subgroup such that $M^*$ transfers to $G$, and $\pi\in\Pi_{\scusp}(M)$ an irreducible supercuspidal representation.
    Put $M:=\xi(M^*)$, which is a Levi subgroup in $G^*$ defined over $F$.
    Let $\dot{F}$, $u$, $v_2$, $\dot{G}^*$, and $\dot{G}$ be as in Lemma \ref{4.1.2}.
    Then there exist an inner twist $\dot{\xi}:\dot{G}^*\to\dot{G}$ over $\dot{F}$, a standard Levi subgroup $\dot{M}^*\subset\dot{G}^*$, and an irreducible cuspidal automorphic representation $\dot{\pi}$ of $\dot{M}(\A_{\dot{F}})$ with following properties:
    \begin{itemize}
        \item $\dot{\xi}_u=\xi$;
        \item $\dot{M}:=\dot{\xi}(\dot{M}^*)$ is defined over $\dot{F}$, and $\dot{\xi}|_{\dot{M}^*}:\dot{M}^*\to\dot{M}$ is an inner twist;
        \item $\dot{M}_u=M$, and $\dot{M}_v$ is split over $\dot{F}_v$ for any place $v$ of $\dot{F}$ except $u$ and $v_2$;
        \item $\dot{\pi}_u = \pi$.
    \end{itemize}
\end{lemma}
\begin{proof}
    Recall that $M^*$ is isomorphic to a group of the form \eqref{levi}.
    Let $\dot{M}^*$ be the group \eqref{levi} defined over $\dot{F}$.
    Note that we can regard $\dot{M}^*$ as a Levi subgroup in $\dot{G}^*$, and that $\dot{M}^*_u$ is isomorphic to $M^*$.
    By Lemma \ref{4.1.2}, we have an inner form $\dot{M}$ of $\dot{M}^*$ and a pure inner twist $(\dot{\xi}_M, \dot{z}_M)$ such that $\dot{M}_u=M$, and $\dot{M}_v$ is split over $\dot{F}_v$ for any place $v$ of $\dot{F}$ except $u$ and $v_2$.

    Let $\dot{z}\in Z^1(\dot{F},\dot{G}^*)$ be the natural image of $\dot{z}_M\in Z^1(\dot{F},\dot{M}^*)$.
    Let $\dot{\xi} : \dot{G}^* \to \dot{G}$ be an inner twist such that $(\dot{\xi}, \dot{z})$ is a pure inner twist.
    Since $\dot{G}^*_u$ and $\dot{G}_u$ are isomorphic to $G^*$ and $G$, respectively, $\dot{\xi}_u$ is isomorphic to $\xi$.
    Thus we can take $\dot{\xi}$ so that $\dot{\xi}_u=\xi$.
    Then we have
    \begin{equation*}
        \dot{\xi}\inv \sigma \dot{\xi} \sigma\inv = \Ad(\dot{z}(\sigma)) = \Ad(\dot{z}_M(\sigma)) = \dot{\xi}_M\inv \sigma \dot{\xi}_M \sigma\inv, 
    \end{equation*}
    for any $\sigma\in \Gamma_{\dot{F}}$.
    Hence we have
    \begin{equation*}
        \dot{\xi} \dot{\xi}_M\inv \sigma = \sigma\dot{\xi} \dot{\xi}_M\inv,
    \end{equation*}
    which means that an injective homomorphism $\dot{\xi} \dot{\xi}_M\inv$ from $\dot{M}$ to $\dot{G}$ is defined over $\dot{F}$.
    Let us regard $\dot{M}$ as a Levi subgroup in $\dot{G}$ by this injective homomorphism.
    Then $\dot{M}=\dot{\xi}(\dot{M}^*)$ is defined over $\dot{F}$, and $\dot{\xi}|_{\dot{M}^*}:\dot{M}^*\to\dot{M}$ is an inner twist.
    
    The last property follows from Lemma \ref{4.2.1} if $M=G$ and \cite[Proposition 5.1]{sha4} if $M\neq G$.
\end{proof}

\subsection{Globalizations of parameters}\label{4.3}
Now we shall globalize generic parameters.
Let us first consider globalizations of simple parameters.
\begin{lemma}\label{4.3.1}
    Let $\dot{F}$ be a totally real number field, and $\dot{G}^*$ a simple twisted endoscopic group of $\GL_N$, i.e., the split symplectic group, the split odd special orthogonal group, or a quasi-split even special orthogonal group over $\dot{F}$.
    Let $V$ be a finite set of places of $\dot{F}$ which misses at least one real place.
    For each $v\in V$, let $\phi_v\in \widetilde{\Phi}_{2,\bdd}(\dot{G}^*_v)$ be a square integrable parameter.
    Assume that for at least one $v\in V$, $\phi_v\in\widetilde{\Phi}_\simple(\dot{G}^*_v)$.
    Then there exists a simple generic parameter $\dot{\phi}\in\widetilde{\Phi}_\simple(\dot{G}^*)$ such that $\dot{\phi}_v=\phi_v$ for all $v\in V$.
\end{lemma}
\begin{proof}
    The proof is similar to that of \cite[Lemma 4.3.1]{kmsw}, except that we need the following one or two modifications.
    The first one is to use Lemma \ref{4.2.1} of this paper in place of Lemma 4.2.1 of \cite{kmsw}.
    Then the same proof carries over word by word in the case when $\dot{G}^*$ is a symplectic or an odd special orthogonal group.
    The second one is, in the case of even special orthogonal groups, to replace $\Phi_\simple$ by $\widetilde{\Phi}_\simple$ and the ordinary equivalence classes of the representations by the $\epsilon$-equivalence classes of the representations.
    Here, note that in the case of quasi-split even special orthogonal groups, the endoscopic classification (LLC, ECR, and AMF) is known up to $\epsilon$-equivalence, by Arthur \cite{art13} and Atobe-Gan \cite{ag}.
\end{proof}
\begin{lemma}\label{4.3.2}
    Let $N\geq 2$.
    Let $\dot{F}$ be a totally real number field, and $\dot{G}^*$ a simple twisted endoscopic group of $\GL_N$, i.e., the split symplectic group, the split odd special orthogonal group, or a quasi-split even special orthogonal group over $\dot{F}$.
    Let $V$ and $\{\phi_v\}_{v\in V}$ be as in Lemma \ref{4.3.1}.
    Assume that for at least one $v\in V$, $\phi_v\in\widetilde{\Phi}_\simple(\dot{G}^*_v)$.
    Let $v_2\notin V$ be a finite place.
    Then there exists a simple generic parameter $\dot{\phi}\in\widetilde{\Phi}_\simple(\dot{G}^*)$ such that $\dot{\phi}_v=\phi_v$ for all $v\in V$ and $\dot{\phi}_{v_2}$ is of the form
    \begin{equation*}
        \dot{\phi}_{v_2}=\phi_+ \oplus \phi_+^\vee \oplus \phi_-,
    \end{equation*}
    where $\phi_+\in\Phi_\bdd(\GL_1)$ is a non-self-dual parameter and $\phi_-\in\widetilde{\Phi}_{\simple,\bdd}(N-2)$ a simple self-dual parameter.
\end{lemma}
\begin{proof}
    The proof is similar to that of \cite[Lemma 4.3.2]{kmsw}, except the following two differences.
    First, we appeal to Lemmas \ref{4.3.1}, \ref{4.2.2}, and \ref{4.2.1} instead of Lemmas 4.3.1, 4.2.2, and 4.2.1, respectively.
    Second, we need the modifications similar to that of Lemma \ref{4.3.1}.
\end{proof}
Next we shall consider globalizations of elliptic or exceptional parameters.
Let $F$ be a local field, and $G$ an inner form of $G^*=\SO_{2n+1}$ over $F$.
Let $M^*\subset G^*$ be a Levi subgroup.
Following \cite[\S4.4]{kmsw}, we shall say that $M^*$ is linear if $M^*$ is isomorphic to a direct product of general linear groups, i.e., $n_0=0$ in \eqref{levi}.
As in the case of even unitary groups, any nontrivial inner form $G$ of $G^*$ does not admit a globalization $\dot{G}$ whose localization at one place is isomorphic to $G$ and at all the other places are split.
Hence the same complication is caused.

Let $\phi\in\Phi_\bdd(G^*)$ be a bounded $L$-parameter for $G$, which is the image of a discrete parameter $\phi_{M^*}\in\Phi_2(M^*)$ for $M^*$.
Assume that $\phi$ is either elliptic or exceptional.
Explicitly $M^*$, $\phi$, and $\phi_{M^*}$ has the form
\begin{align*}
    M^*        & \simeq \GL_{N_1}^{e_1}\times\cdots\times \GL_{N_r}^{e_r} \times M^*_-, \\
    \phi       & =\bigoplus_{i=1}^r \ell_i\phi_i
    =\bigoplus_{i=1}^r \left(e_i(\phi_i\oplus\phi_i^\vee)\oplus\delta_i\phi_i\right),   \\
    \phi_{M^*} & =e_1\phi_1 \oplus\cdots\oplus e_r\phi_r \oplus \phi_-,
\end{align*}
where $\phi_i\in\widetilde{\Phi}_{\simple,\bdd}(N_i)$ are mutually distinct self-dual simple bounded $L$-parameters for $\GL_{N_i}$, $\phi_-=\oplus_i \delta_i\phi_i\in\Phi_2(M^*_-)$, $\delta_i\in\{0,1\}$, $N_i\in\Z_{\geq1}$, $e_i\in\Z_{\geq0}$, $\ell_i=2e_i+\delta_i$, $M^*_-=\SO_{2n_0+1}$, and $2n_0=\sum_i\delta_iN_i$ so that $\sum_i e_iN_i+n_0=n$ and $e_i=\lfloor\ell_i/2\rfloor$.

In the lemmas and propositions below, we will start from the following local data with the assumption given above:
\begin{itemize}
    \item a local field $F$;
    \item $G^*=\SO_{2n+1}$ and a Levi subgroup $M^*\subset G^*$ over $F$;
    \item an inner form $G$ of $G^*$;
    \item an $L$-parameter $\phi\in\Phi_\bdd(G^*)$ that is the image of a parameter $\phi_{M^*}\in\Phi_2(M^*)$.
\end{itemize}

First we treat the case when $M^*=G^*$.
\begin{lemma}\label{red}
    Assume that $M^*=G^*$, and hence $\phi\in\Phi_2(G^*)$.
    Then there exists a global data $(\dot{F}, \dot{G}^*, \dot{\phi}, u, v_1, v_2)$, where $\dot{F}$ is a totally real number field, $\dot{G}^*=\SO_{2n+1}$ over $\dot{F}$, $\dot{\phi}\in\Phi(\dot{G}^*)$, and $u, v_1, v_2$ are places of $\dot{F}$, such that $v_1$ is a finite place, $v_2$ is a real place, and
    \begin{enumerate}
        \item $\dot{F}_u=F$, $\dot{G}^*_u=G^*$, and $\dot{\phi}_u=\phi$;
        \item $\dot{\phi}\in\Phi_2(\dot{G}^*)$;
        \item $\dot{\phi}_v\in\Phi_{2,\bdd}(\dot{G}^*_v)$ for $v\in\{v_1,v_2\}$;
        \item the canonical map $S_{\dot{\phi}}\to S_{\dot{\phi}_v}$ is isomorphic for $v\in\{u,v_1\}$.
    \end{enumerate}
\end{lemma}
\begin{proof}
    A totally real number field $\dot{F}$ and a place $u$ such that $\dot{F}_u=F$ are given by Lemma \ref{art6.2.1}.
    Here we take $\dot{F}$ to have more than two real places.
    Let $v_1$ and $v_2$ be finite and real places of $\dot{F}$, respectively.
    Put $\dot{G}^*=\SO_{2n+1}$ over $\dot{F}$.

    For each $i=1,\ldots, r$, let $G^*_i\in\widetilde{\calE}_\simple(N_i)$ be a classical group over $F$ such that $\phi_i\in\widetilde{\Phi}_{\simple,\bdd}(G^*_i)$.
    Such $G^*_i$ exists, since $\phi_i$ is an $N_i$-dimensional irreducible self-dual representation of $L_F$, which is either symplectic or orthogonal.
    Take $\dot{G}^*_i\in\widetilde{\calE}_\simple(N_i)$ to be a simple twisted endoscopic group over $\dot{F}$ so that $\dot{G}^*_{i,u}=G^*_i$.
    Choose a collection $\phi_{v_1,i}\in\widetilde{\Phi}_{\simple,\bdd}(\dot{G}^*_{i,v_1})$ $(i=1,\ldots,r)$ of pairwise distinct parameters.
    Choose also a collection $\phi_{v_2,i}\in\widetilde{\Phi}_{2,\bdd}(\dot{G}^*_{i,v_2})$ $(i=1,\ldots,r)$ of parameters such that $\phi_{v_2,i}$ and $\phi_{v_2,j}$ do not have a common constituent for all $i\neq j$.
    Then Lemma \ref{4.3.1} gives us a collection of parameters $\dot{\phi}_i\in\widetilde{\Phi}_\simple(\dot{G}^*_i)$ such that $\dot{\phi}_{i,u}=\phi_i$, $\dot{\phi}_{i,v_1}=\phi_{v_1,i}$, and $\dot{\phi}_{i,v_2}=\phi_{v_2,i}$.

    By the assumption, we have $e_i=0$ and $\delta_i=1$ for all $i$.
    Put
    \begin{align*}
        \dot{\phi} & =\dot{\phi}_1 \boxplus\cdots\boxplus \dot{\phi}_r\in\Phi(\dot{G}^*).
    \end{align*}
    Then by the construction the canonical map $S_{\dot{\phi}}\to S_{\dot{\phi}_v}$ is isomorphic for $v\in\{u,v_1\}$.
    Thus the fourth and second conditions are satisfied.
    The first condition is clearly satisfied.
    The third one follows immediately from the construction.
\end{proof}
\begin{proposition}\label{4.4.3}
    Assume that $M^*=G^*$, and $\phi\in\Phi_2(G^*)$ is relevant for $G$.
    Then there exists a global data
    \begin{equation*}
        (\dot{F},\dot{G}^*,\dot{G}, \dot{\phi}, u,v_1,v_2),
    \end{equation*}
    where $\dot{F}$ is a totally real number field, $\dot{G}^*=\SO_{2n+1}$ over $\dot{F}$, $\dot{G}$ an inner form of $\dot{G}^*$, $\dot{\phi}\in\Phi_2(\dot{G}^*)$ a parameter relevant for $\dot{G}$, and $u, v_1, v_2$ are places of $\dot{F}$, such that $v_1$ is a finite place, $v_2$ is a real place, and
    \begin{enumerate}
        \item $\dot{F}_u=F$, $\dot{G}^*_u=G^*$, $\dot{G}_u=G$, and $\dot{\phi}_u=\phi$;
        \item $\dot{G}_v$ is split unless $v\in\{u,v_2\}$;
        \item $\dot{\phi}_v\in\Phi_{2,\bdd}(\dot{G}^*_v)$ for $v\in\{v_1,v_2\};$
        \item the canonical map $S_{\dot{\phi}}\to S_{\dot{\phi}_v}$ is isomorphic for $v\in\{u,v_1\}$.
    \end{enumerate}
\end{proposition}
\begin{proof}
    By Lemma \ref{red}, we obtain $\dot{F}$, $\dot{G}^*$, $\dot{\phi}$, $u$, $v_1$, and $v_2$ satisfying the third and fourth conditions.
    If $G$ is not split over $F$, then by Lemma \ref{4.1.2}, we obtain $\dot{G}$ that satisfies the second condition for which $\dot{\phi}$ is relevant.
    The first condition is now clear.
    If $G$ splits over $F$, then $\dot{G}=\SO_{2n+1}$ clearly satisfies the conditions.
\end{proof}

Now we treat the case that $M^*$ is proper, which implies that $\phi\notin\Phi_2(G^*)$.
First, we consider the case when $M^*$ is not linear, excluding some special cases as below.
\begin{lemma}\label{blue}
    Assume that $M^*\subsetneq G^*$ is not linear.
    Assume also that $(n,\phi)$ is neither $(2, \text{of type }(\exc1))$ nor $(3, \text{of type }(\exc2))$.
    Then there exists a global data $(\dot{F}, \dot{G}^*, \dot{M}^*, \dot{\phi}, \dot{\phi}_{\dot{M}^*}, u, v_1, v_2)$, where $\dot{F}$ is a totally real number field, $\dot{G}^*=\SO_{2n+1}$ over $\dot{F}$, $\dot{M}^*\subset\dot{G}^*$ a Levi subgroup over $\dot{F}$, $\dot{\phi}\in\Phi(\dot{G}^*)$, $\dot{\phi}_{\dot{M}^*}\in\Phi_2(\dot{M}^*)$, and $u, v_1, v_2$ are places of $\dot{F}$, such that $v_1$ and $v_2$ are finite places and
    \begin{enumerate}
        \item $\dot{F}_u=F$, $\dot{G}^*_u=G^*$, $\dot{M}^*_u=M^*$, $\dot{\phi}_u=\phi$, and $\dot{\phi}_{\dot{M}^*,u}=\phi_{M^*}$;
        \item if $\phi\in\Phi^2_\el(G^*)$ (resp. $\Phi_{\exc1}(G^*)$, resp. $\Phi_{\exc2}(G^*)$), then $\dot{\phi}\in\Phi^2_\el(\dot{G}^*)$ (resp. $\Phi_{\exc1}(\dot{G}^*)$, resp. $\Phi_{\exc2}(\dot{G}^*)$);
        \item $\dot{\phi}_{v_1}\in\Phi_\bdd(\dot{G}^*_{v_1})$ and $\dot{\phi}_{\dot{M}^*,v_1}\in\Phi_{2,\bdd}(\dot{M}^*_{v_1})$;
        \item $\dot{\phi}_{v_2}\in\Phi_\bdd^{\el,\exc}(\dot{G}^*_{v_2})$ and it has a symplectic simple component with odd multiplicity;
        \item the canonical maps $S_{\dot{\phi}}\to S_{\dot{\phi}_v}$ and $S_{\dot{\phi}_{\dot{M}^*}}\to S_{\dot{\phi}_{\dot{M}^*,v}}$ are isomorphic for $v\in\{u,v_1\}$.
    \end{enumerate}
\end{lemma}
\begin{proof}
    A totally real number field $\dot{F}$ and a place $u$ such that $\dot{F}_u=F$ are given by Lemma \ref{art6.2.1}.
    Let $v_1$ and $v_2$ be finite places of $\dot{F}$.
    Put $\dot{G}^*=\SO_{2n+1}$ and $\dot{M}^*_-=\SO_{2n_0+1}$ over $\dot{F}$.
    Let $\dot{M}^*\subset\dot{G}^*$ be a Levi subgroup over $\dot{F}$ such that
    \begin{equation*}
        \dot{M}^*\simeq \GL_{N_1}^{e_1}\times\cdots\times \GL_{N_r}^{e_r} \times \dot{M}^*_-.
    \end{equation*}

    For each $i=1,\ldots, r$, let $G^*_i\in\widetilde{\calE}_\simple(N_i)$ over $F$ and $\dot{G}^*_i\in\widetilde{\calE}_\simple(N_i)$ over $\dot{F}$ be as in the proof of Lemma \ref{red}.
    Choose a collection $\phi_{v_1,i}\in\widetilde{\Phi}_{\simple,\bdd}(\dot{G}^*_{i,v_1})$, $(i=1,\ldots,r)$ of pairwise distinct parameters.

    Since $M^*$ is not linear, we have $n_0\geq1$, which implies that $\delta_i=1$ (i.e., $\ell_i$ is odd) for some $i$.
    For such $i$, $\phi_i$ must be symplectic, in particular $N_i\geq2$.

    Suppose that there exists $1\leq s\leq r$ such that $\delta_s=1$, $N_s=2$, and $N_t=1$ for all $t\neq s$.
    In this case $\ell_s$ is odd and $\ell_t$ is even for $t\neq s$.
    We may assume $s=1$.
    If $\phi$ is elliptic, it can be written as
    \begin{equation*}
        \phi=\phi_1\oplus 2\phi_{q+1}\oplus\cdots\oplus 2\phi_r,
    \end{equation*}
    where $\phi_i$ are all symplectic, in particular $N_i\geq 2$ for all $i$.
    This leads to $r=1$ and $\phi=\phi_1$, which contradicts the assumption that $M^*$ is proper.
    If $\phi$ is of type ($\exc1$), it can be written as
    \begin{equation*}
        \phi=\phi_1\oplus 2\phi_2,
    \end{equation*}
    where $\phi_1$ is symplectic, and $\phi_2$ is orthogonal.
    Then the dimension of $\phi$ equals $2+2\times 1=4$, which contradicts the assumption that $(n,\phi)$ is not (2, of type ($\exc1$)).
    If $\phi$ is of type ($\exc2$), it can be written as
    \begin{equation*}
        \phi=3\phi_1,
    \end{equation*}
    where $\phi_1$ is symplectic.
    Then the dimension of $\phi$ equals $3\times 2=6$, which contradicts the assumption that $(n,\phi)$ is not (3, of type ($\exc2$)).

    Therefore, we have that either there exists $1\leq s\leq r$ such that $N_s\geq 3$, or there exists $1\leq s\leq r$ such that $N_s=2$ and $\delta_t=1$ for some $t\neq s$.
    In both cases, by Lemma \ref{4.3.2}, we have $\dot{\phi}_s\in \widetilde{\Phi}_\simple(\dot{G}^*_s)$ such that $\dot{\phi}_{s,u}=\phi_s$, $\dot{\phi}_{s,v_1}=\phi_{v_1,s}$, and $\dot{\phi}_{s,v_2}$ is of the form
    \begin{align*}
        \dot{\phi}_{s,v_2}=\phi_{v_2+} \oplus \dot{\phi}_{v_2+}^\vee \oplus \phi_{v_2-},
    \end{align*}
    where $\phi_{v_2+}\in\Phi_\bdd(\GL_1)$ is a non-self-dual parameter, and $\phi_{v_2-}\in\widetilde{\Phi}_\bdd(N_s-2)$ a simple self-dual parameter.
    For $i\neq s$, choose $\phi_{v_2,i}\in \widetilde{\Phi}_{\simple,\bdd}(\dot{G}^*_i)$ so that $\{\phi_{v_2,i}\}_{i\neq s}$ are mutually distinct and $\phi_{v_2,i}\neq \phi_{v_2-}$ for all $i$.
    Lemma \ref{4.3.1} gives us $\dot{\phi}_i\in\widetilde{\Phi}_\simple(\dot{G}^*_i)$ such that $\dot{\phi}_{i,u}=\phi_i$, $\dot{\phi}_{i,v_1}=\phi_{v_1,i}$, and $\dot{\phi}_{i,v_2}=\phi_{v_2,i}$.

    Put
    \begin{align*}
        \dot{\phi}             & =\ell_1\dot{\phi}_1\boxplus\cdots\boxplus\ell_r\dot{\phi}_r\in\Phi(\dot{G}^*),                    \\
        \dot{\phi}_{\dot{M}^*} & =e_1\dot{\phi}_1\boxplus\cdots\boxplus e_r\dot{\phi}_r \boxplus \dot{\phi}_-\in\Phi_2(\dot{M}^*),
    \end{align*}
    where $\dot{\phi}_-=\boxplus_i\delta_i\dot{\phi}_i$.
    The first and fifth conditions immediately follow from the construction, and hence so do the second and third ones.

    Let us consider the fourth condition.
    We have
    \begin{equation*}
        \dot{\phi}_{v_2}=\ell_s\phi_{v_2-} \oplus \ell_s(\phi_{v_2+}\oplus \phi_{v_2+}^\vee) \oplus \bigoplus_{i\neq s} \ell_i\phi_{v_2,i},
    \end{equation*}
    which is neither elliptic nor exceptional, because $\phi_{v_2+}$ is not self-dual.
    If $N_s\geq 3$, then $\dot{\phi}_{v_2}$ has a symplectic simple component with odd multiplicity because $\phi$ has.
    If $N_s=2$ and $\delta_t=1$ for $t\neq s$, then clearly $\phi_{v_2,t}$ is a symplectic simple component with odd multiplicity.
    This completes the proof.
\end{proof}
\begin{proposition}\label{4.4.4}
    Assume that $M^*\subsetneq G^*$ is not linear.
    Assume also that $(n,\phi)$ is neither $(2, \text{of type }(\exc1))$ nor $(3, \text{of type }(\exc2))$.
    Then there exists a global data
    \begin{equation*}
        (\dot{F},\dot{G}^*,\dot{G}, \dot{\phi}, \dot{M}^*, \dot{\phi}_{\dot{M}^*}, u, v_1, v_2),
    \end{equation*}
    where $\dot{F}$ is a totally real number field, $\dot{G}^*=\SO_{2n+1}$ over $\dot{F}$, $\dot{G}$ an inner form of $\dot{G}^*$, $\dot{\phi}\in\Phi(\dot{G}^*)$ a parameter, $\dot{M}^*\subset\dot{G}^*$ a Levi subgroup, $\dot{\phi}_{\dot{M}^*}\in\Phi_2(\dot{M}^*)$ a parameter whose image in $\Phi(\dot{G}^*)$ is $\dot{\phi}$, and $u, v_1, v_2$ are places of $\dot{F}$, such that $v_1$ and $v_2$ are finite place, and
    \begin{enumerate}
        \item $\dot{F}_u=F$, $\dot{G}^*_u=G^*_u$, $\dot{G}_u=G$, $\dot{\phi}_u=\phi$, $\dot{M}^*_u=M^*$, and $\dot{\phi}_{\dot{M}^*,u}=\phi_{M^*}$;
        \item $\dot{G}_v$ is split over $\dot{F}_v$ unless $v\in\{u,v_2\}$;
        \item if $\phi\in\Phi^2_\el(G^*)$ (resp. $\Phi_{\exc1}(G^*)$, resp. $\Phi_{\exc2}(G^*)$), then $\dot{\phi}\in\Phi^2_\el(\dot{G}^*)$ (resp. $\Phi_{\exc1}(\dot{G}^*)$, resp. $\Phi_{\exc2}(\dot{G}^*)$);
        \item $\dot{\phi}_{v_1}\in\Phi_\bdd(\dot{G}^*_{v_1})$ and $\dot{\phi}_{\dot{M}^*,v_1}\in\Phi_{2,\bdd}(\dot{M}^*_{v_1})$;
        \item $\dot{\phi}_{v_2}\in\Phi^{\el,\exc}_\bdd(\dot{G}^*_{v_2})$ is relevant for $\dot{G}_{v_2}$;
        \item the canonical maps $S_{\dot{\phi}}\to S_{\dot{\phi}_v}$ and $S_{\dot{\phi}_{\dot{M}^*}}\to S_{\dot{\phi}_{\dot{M}^*,v}}$ are isomorphic for $v\in\{u,v_1\}$.
    \end{enumerate}
\end{proposition}
\begin{proof}
    By Lemma \ref{blue}, we obtain $\dot{F}$, $\dot{G}^*$, $\dot{\phi}$, $\dot{M}^*$, $\dot{\phi}_{\dot{M}^*}$, $u$, $v_1$, and $v_2$ satisfying the third, fourth, and sixth conditions.
    By Lemma \ref{4.1.2}, we obtain $\dot{G}$ satisfying the second condition.
    Since $v_2$ is a finite place, the fifth condition follows from the fourth condition of Lemma \ref{blue}, regardless of $\dot{G}$.
    The first condition is now clear.
\end{proof}

Next we shall consider the special cases.
\begin{lemma}\label{emerald}
    Assume that $M^*\subsetneq G^*$ is not linear.
    Assume also that $n=2$ or $3$.
    Let $\heartsuit=(\exc1)$ (resp. $(\exc2)$) if $n=2$ (resp. $3$), and assume that $\phi$ is of type $\heartsuit$.
    Then there exists a global data $(\dot{F}, \dot{G}^*, \dot{M}^*, \dot{\phi}, \dot{\phi}_{\dot{M}^*}, u, v_1, v_2)$, where $\dot{F}$ is a totally real number field, $\dot{G}^*=\SO_{2n+1}$ over $\dot{F}$, $\dot{M}^*\subset\dot{G}^*$ a Levi subgroup over $\dot{F}$, $\dot{\phi}\in\Phi(\dot{G}^*)$, $\dot{\phi}_{\dot{M}^*}\in\Phi_2(\dot{M}^*)$, and $u, v_1, v_2$ are places of $\dot{F}$, such that $v_1$ is a finite place, $v_2$ is a real place, and
    \begin{enumerate}
        \item $\dot{F}_u=F$, $\dot{G}^*_u=G^*$, $\dot{M}^*_u=M^*$, $\dot{\phi}_u=\phi$, and $\dot{\phi}_{\dot{M}^*,u}=\phi_{M^*}$;
        \item $\dot{\phi}\in\Phi_\heartsuit(\dot{G}^*)$;
        \item $\dot{\phi}_{v_1}\in\Phi_\bdd(\dot{G}^*_{v_1})$ and $\dot{\phi}_{\dot{M}^*,v_1}\in\Phi_{2,\bdd}(\dot{M}^*_{v_1})$;
        \item $\dot{\phi}_{v_2}=2\omega_0\oplus \tau_1$ or $3\tau_1$;
        \item the canonical maps $S_{\dot{\phi}}\to S_{\dot{\phi}_v}$ and $S_{\dot{\phi}_{\dot{M}^*}}\to S_{\dot{\phi}_{\dot{M}^*,v}}$ are isomorphic for $v\in\{u,v_1\}$.
    \end{enumerate}
    Here, $\omega_0$ and $\tau_1$ are representations of $W_\R$ defined in the beginning of the subsection \ref{2.9}.
\end{lemma}
\begin{proof}
    A totally real number field $\dot{F}$ and a place $u$ such that $\dot{F}_u=F$ are given by Lemma \ref{art6.2.1}.
    Let $v_1$ be a finite place and $v_2$ a real place of $\dot{F}$.
    Put $\dot{G}^*=\SO_{2n+1}$ over $\dot{F}$.

    Consider first the case when $n=2$ and $\heartsuit=(\exc1)$.
    By the assumption we have $\dot{G}^*=\SO_5$, $\phi=2\phi_1\oplus\phi_2$, $S_\phi\simeq\Sp(2,\C)\times\Or(1,\C)$, and $M^*\simeq\GL_1\times \SO_3$, where $\phi_1$ is 1-dimensional orthogonal and $\phi_2$ is irreducible 2-dimensional symplectic.
    We also have $N_1=1$ and $N_2=2$.
    For $i=1$, $2$, choose $(\dot{G}^*_i, s_i, \eta_i)\in \widetilde{\calE}(N_i)$ so that $\phi_i\in\Phi(\dot{G}^*_i)$ if we regard $\Phi(\dot{G}^*_i)$ as a subset of $\Phi(N_i)$ via $\eta_i$.
    Concretely, $\dot{G}^*_1=\Sp_0$ and $\dot{G}^*_2=\SO_3$.
    Let $\phi_{v_1,i}$ be an element of $\Phi_{\simple,\bdd}(\dot{G}^*_i)$, for $i=1, 2$.
    Put $\phi_{v_2,1}=\omega_0$ and $\phi_{v_2,2}=\tau_1$.
    Then Lemma \ref{4.3.1} gives us a global parameter $\dot{\phi}_i\in\Phi_\simple(\dot{G}^*_i)$ such that $\dot{\phi}_{i,u}=\phi_i$, $\dot{\phi}_{i,v_1}=\phi_{v_1,i}$, and $\dot{\phi}_{i,v_2}=\phi_{v_2,i}$.
    Put
    \begin{align*}
        \dot{\phi}             & =2\dot{\phi}_1 \boxplus \dot{\phi}_2, \\
        \dot{M}^*              & =\GL_1\times \SO_3,                   \\
        \dot{\phi}_{\dot{M}^*} & =\dot{\phi}_1 \boxplus \dot{\phi}_2.
    \end{align*}

    Consider next the case when $n=3$ and $\heartsuit=(\exc2)$.
    By the assumption we have $\dot{G}^*=\SO_7$, $\phi=3\phi_1$, $S_\phi\simeq\Or(3,\C)$, and $M^*\simeq\GL_2\times \SO_3$, where $\phi_1$ is irreducible 2-dimensional symplectic.
    Hence we regard $\phi_1\in\Phi(\SO_3/F)$.
    Let $\phi_{v_1,1}$ be an element of $\Phi_{\simple,\bdd}(\SO_3/\dot{F}_{v_1})$, and put $\phi_{v_2,1}=\tau_1$.
    Then Lemma \ref{4.3.1} gives us a global parameter $\dot{\phi}_1\in\Phi_\simple(\SO_3/\dot{F})$ such that $\dot{\phi}_{1,u}=\phi_1$, $\dot{\phi}_{1,v_1}=\phi_{v_1,1}$, and $\dot{\phi}_{1,v_2}=\phi_{v_2,1}$.
    Put
    \begin{align*}
        \dot{\phi}             & =3\dot{\phi}_1,                      \\
        \dot{M}^*              & =\GL_2\times \SO_3,                  \\
        \dot{\phi}_{\dot{M}^*} & =\dot{\phi}_1 \boxplus \dot{\phi}_1.
    \end{align*}

    In both cases, the conditions follow from the construction.
    This completes the proof.
\end{proof}
\begin{proposition}\label{4.4.6}
    Assume that $M^*\subsetneq G^*$ is not linear.
    Assume also that $n=2$ or $3$.
    Let $\heartsuit=(\exc1)$ (resp. $(\exc2)$) if $n=2$ (resp. $3$).
    Assume that $\phi\in\Phi_{\bdd,\heartsuit}(G^*)$ and that $G$ is not quasi-split.
    Then there exists a global data
    \begin{equation*}
        (\dot{F},\dot{G}^*,\dot{G}, \dot{\phi}, \dot{M}^*, \dot{\phi}_{\dot{M}^*}, u,v_1,v_2),
    \end{equation*}
    where $\dot{F}$ is a totally real number field, $\dot{G}^*=\SO_{2n+1}$ over $\dot{F}$, $\dot{G}$ an inner form of $\dot{G}^*$, $\dot{\phi}\in\Phi(\dot{G}^*)$ a parameter, $\dot{M}^*\subset\dot{G}^*$ a Levi subgroup, $\dot{\phi}_{\dot{M}^*}\in\Phi_2(\dot{M}^*)$ a parameter whose image in $\Phi(\dot{G}^*)$ is $\dot{\phi}$, and $u, v_1, v_2$ are places of $\dot{F}$, such that $v_1$ is a finite place, $v_2$ is a real place, and
    \begin{enumerate}
        \item $\dot{F}_u=F$, $\dot{G}^*_u=G^*_u$, $\dot{G}_u=G$, $\dot{\phi}_u=\phi$, $\dot{M}^*_u=M^*$, and $\dot{\phi}_{\dot{M}^*,u}=\phi_{M^*}$;
        \item $\dot{G}_v$ is split over $\dot{F}_v$ unless $v\in\{u,v_2\}$;
        \item $\dot{\phi}\in\Phi_\heartsuit(\dot{G}^*)$;
        \item $\dot{\phi}_{v_1}\in\Phi_\bdd(\dot{G}^*_{v_1})$ and $\dot{\phi}_{\dot{M}^*,v_1}\in\Phi_{2,\bdd}(\dot{M}^*_{v_1})$;
        \item $\dot{\phi}_{v_2}\in\Phi_{\bdd,\heartsuit}(\dot{G}^*_{v_2})$ is relevant for $\dot{G}_{v_2}$ and satisfies Theorem \ref{2.6.2} relative to $\dot{M}^*_{v_2}$;
        \item the canonical maps $S_{\dot{\phi}}\to S_{\dot{\phi}_v}$ and $S_{\dot{\phi}_{\dot{M}^*}}\to S_{\dot{\phi}_{\dot{M}^*,v}}$ are isomorphic for $v\in\{u,v_1\}$.
    \end{enumerate}
\end{proposition}
\begin{proof}
    By Lemma \ref{emerald}, we obtain $\dot{F}$, $\dot{G}^*$, $\dot{\phi}$, $\dot{M}^*$, $\dot{\phi}_{\dot{M}^*}$, $u$, $v_1$, and $v_2$ satisfying the third, fourth, and sixth conditions.
    By Lemma \ref{4.1.2}, we obtain $\dot{G}$ such that $\dot{G}_u=G$, $\dot{G}_{v_2}\simeq \SO(n-1,n+2)$, and $\dot{G}_v$ is split if $v\notin\{u, v_2\}$.
    Hence the second condition is satisfied.
    The fifth condition follows from \S\ref{2.9}.
    The first condition is now clear.
\end{proof}

If $M^*$ is linear and $G$ is non-quasi-split, then $M^*$ never transfer to $G$ and $\phi$ is not relevant.
The next proposition will be applied to such a case
\begin{lemma}\label{orange}
    Assume that $M^*\subsetneq G^*$ is proper.
    There exists a global data $(\dot{F}, \dot{G}^*, \dot{M}^*, \dot{\phi}, \dot{\phi}_{\dot{M}^*}, u_1, u_2, v_1)$, where $\dot{F}$ is a totally real number field, $\dot{G}^*=\SO_{2n+1}$ over $\dot{F}$, $\dot{M}^*\subset\dot{G}^*$ a Levi subgroup over $\dot{F}$, $\dot{\phi}\in\Phi(\dot{G}^*)$, $\dot{\phi}_{\dot{M}^*}\in\Phi_2(\dot{M}^*)$, and $u_1, u_2, v_1$ are places of $\dot{F}$, such that $v_1$ is a finite place and
    \begin{enumerate}
        \item $\dot{F}_u=F$, $\dot{G}^*_u=G^*$, $\dot{M}^*_u=M^*$, $\dot{\phi}_u=\phi$, and $\dot{\phi}_{\dot{M}^*,u}=\phi_{M^*}$, for $u\in\{u_1, u_2\}$;
        \item if $\phi\in\Phi_2(G^*)$ (resp. $\Phi^2_\el(G^*)$, resp. $\Phi_{\exc1}(G^*)$, resp. $\Phi_{\exc2}(G^*)$), then $\dot{\phi}\in\Phi_2(\dot{G}^*)$ (resp. $\Phi^2_\el(\dot{G}^*)$, resp. $\Phi_{\exc1}(\dot{G}^*)$, resp. $\Phi_{\exc2}(\dot{G}^*)$);
        \item $\dot{\phi}_{v_1}\in\Phi_\bdd(\dot{G}^*_{v_1})$ and $\dot{\phi}_{\dot{M}^*,v_1}\in\Phi_{2,\bdd}(\dot{M}^*_{v_1})$;
        \item the canonical maps $S_{\dot{\phi}}\to S_{\dot{\phi}_v}$ and $S_{\dot{\phi}_{\dot{M}^*}}\to S_{\dot{\phi}_{\dot{M}^*,v}}$ are isomorphic for $v\in\{u_1, u_2, v_1\}$.
    \end{enumerate}
\end{lemma}
\begin{proof}
    A totally real field $\dot{F}$ and places $u_1$ and $u_2$ such that $\dot{F}_{u_1}=\dot{F}_{u_2}=F$ are given by Lemma \ref{4.1.1}.
    Let $v_1$ be a finite place of $\dot{F}$.
    Put $\dot{G}^*=\SO_{2n+1}$ and $\dot{M}^*_-=\SO_{2n_0+1}$ over $\dot{F}$.
    Let $\dot{M}^*\subset\dot{G}^*$ be a Levi subgroup over $\dot{F}$ such that
    \begin{equation*}
        \dot{M}^*\simeq \GL_{N_1}^{e_1}\times\cdots\times \GL_{N_r}^{e_r} \times \dot{M}^*_-.
    \end{equation*}

    For each $i=1,\ldots, r$, let $G^*_i\in\widetilde{\calE}_\simple(N_i)$ be a classical group over $F$ such that $\phi_i\in\widetilde{\Phi}_{\simple,\bdd}(G^*_i)$, and take $\dot{G}^*_i\in\widetilde{\calE}_\simple(N_i)$ to be a simple twisted endoscopic group over $\dot{F}$ so that $\dot{G}^*_{i,u}=G^*_i$ for $u\in\{u_1,u_2\}$.
    Choose a collection $\phi_{v_1,i}\in\widetilde{\Phi}_{\simple,\bdd}(\dot{G}^*_{i,v_1})$ $(i=1,\ldots,r)$ of pairwise distinct parameters.
    Then Lemma \ref{4.3.1} gives us a collection of parameters $\dot{\phi}_i\in\widetilde{\Phi}_\simple(\dot{G}^*_i)$ such that $\dot{\phi}_{i,v_1}=\phi_{v_1,i}$ and $\dot{\phi}_{i,u}=\phi_i$ for $u\in\{u_1,u_2\}$.

    Put
    \begin{align*}
        \dot{\phi}             & =\ell_1\dot{\phi}_1 \boxplus\cdots\boxplus \ell_r\dot{\phi}_r \in \Phi(\dot{G}^*),                  \\
        \dot{\phi}_{\dot{M}^*} & =e_1\dot{\phi}_1\boxplus\cdots\boxplus e_r\dot{\phi}_r \boxplus \dot{\phi}_- \in \Phi_2(\dot{M}^*),
    \end{align*}
    where $\dot{\phi}_-=\boxplus_i\delta_i\dot{\phi}_i$.
    Then by the construction, the first and fourth conditions are satisfied.
    Hence the second and third ones follow.
\end{proof}
\begin{proposition}\label{4.4.7}
    Assume that $M^*\subsetneq G^*$ is proper.
    Then there exists a global data
    \begin{equation*}
        (\dot{F},\dot{G}^*,\dot{G}, \dot{\phi}, \dot{M}^*, \dot{\phi}_{\dot{M}^*}, u_1, u_2, v_1),
    \end{equation*}
    where $\dot{F}$ is a totally real number field, $\dot{G}^*=\SO_{2n+1}$ over $\dot{F}$, $\dot{G}$ an inner form of $\dot{G}^*$, $\dot{\phi}\in\Phi(\dot{G}^*)$ a parameter, $\dot{M}^*\subset\dot{G}^*$ a Levi subgroup, $\dot{\phi}_{\dot{M}^*}\in\Phi_2(\dot{M}^*)$ a parameter whose image in $\Phi(\dot{G}^*)$ is $\dot{\phi}$, and $u_1, u_2, v_1$ are places of $\dot{F}$, such that $v_1$ is a finite place and
    \begin{enumerate}
        \item $\dot{F}_u=F$, $\dot{G}^*_u=G^*_u$, $\dot{G}_u=G$, $\dot{\phi}_u=\phi$, $\dot{M}^*_u=M^*$, and $\dot{\phi}_{\dot{M}^*,u}=\phi_{M^*}$, for $u\in\{u_1, u_2\}$;
        \item $\dot{G}_v$ is split over $\dot{F}_v$ unless $v\in\{u_1, u_2\}$;
        \item if $\phi\in\Phi_2(G^*)$ (resp. $\Phi^2_\el(G^*)$, resp. $\Phi_{\exc1}(G^*)$, resp. $\Phi_{\exc2}(G^*)$), then $\dot{\phi}\in\Phi_2(\dot{G}^*)$ (resp. $\Phi^2_\el(\dot{G}^*)$, resp. $\Phi_{\exc1}(\dot{G}^*)$, resp. $\Phi_{\exc2}(\dot{G}^*)$);
        \item $\dot{\phi}_{v_1}\in\Phi_\bdd(\dot{G}^*_{v_1})$ and $\dot{\phi}_{\dot{M}^*,v_1}\in\Phi_{2,\bdd}(\dot{M}^*_{v_1})$;
        \item the canonical maps $S_{\dot{\phi}}\to S_{\dot{\phi}_v}$ and $S_{\dot{\phi}_{\dot{M}^*}}\to S_{\dot{\phi}_{\dot{M}^*,v}}$ are isomorphic for $v\in\{u_1, u_2, v_1\}$.
    \end{enumerate}
\end{proposition}
\begin{proof}
    By Lemma \ref{orange}, we obtain $\dot{F}$,$\dot{G}^*$, $\dot{\phi}$, $\dot{M}^*$, $\dot{\phi}_{\dot{M}^*}$ $u_1$, $u_2$, and $v_1$ satisfying the first (except the assertion on $\dot{G}_u$), third, fourth, and fifth conditions.
    By Lemma \ref{4.1.2}, we obtain $\dot{G}$ such that $\dot{G}_v=G$ for $v\in\{u_1,u_2\}$ and $\dot{G}_v$ splits over $\dot{F}_v$ otherwise.
    This completes the proof.
\end{proof}

\subsection{On elliptic parameters}\label{4.5}
The following lemma is proved in the same way as \cite[Lemma 4.5.1]{kmsw}.
\begin{lemma}\label{4.5.1}
    Let $\dot{F}$ be a number field, $\dot{G}^*=\SO_{2n+1}$ over $\dot{F}$, $\dot{\phi}\in\Phi^2_\el(\dot{G}^*)$ a parameter, and $\dot{G}$ an inner form of $\dot{G}^*$.
    Let $\dot{M}^* \subset \dot{G}^*$ be a Levi subgroup and $\dot{\phi}_{\dot{M}^*}\in\Phi_2(\dot{M}^*)$ a discrete parameter whose image in $\Phi(\dot{G}^*)$ is $\dot{\phi}$.
    Assume that there is a place $v_1$ of $\dot{F}$ such that
    \begin{itemize}
        \item $\dot{G}_{v_1}$ is split over $\dot{F}$;
        \item $\dot{\phi}_{v_1} \in\Phi_\bdd(\dot{G}^*_{v_1})$ and $\dot{\phi}_{\dot{M}^*_{v_1}} \in \Phi_{2,\bdd}(\dot{M}^*_{v_1})$;
        \item the canonical maps $S_{\dot{\phi}}\to S_{\dot{\phi}_{v_1}}$ and $S_{\dot{\phi}_{\dot{M}^*}}\to S_{\dot{\phi}_{\dot{M}^*,v_1}}$ are isomorphic.
    \end{itemize}
    Then we have
    \begin{equation*}
        \tr R^{\dot{G}}_{\disc,\dot{\phi}}(\dot{f})
        =\sum_{\overline{x}\in\overline{\frakS}_{\dot{\phi},\el}} \left(\dot{f}'_{\dot{G}}(\dot{\phi},\overline{x})-\dot{f}_{\dot{G}}(\dot{\phi},\overline{x}) \right)
        =0,
    \end{equation*}
    for all $\dot{f}\in\calH(\dot{G})$.
\end{lemma}

\subsection{Proof of LIR for $L$-parameters}\label{4.6}
In this subsection we complete the proof of Theorem \ref{2.6.2} for generic parameters (i.e., $L$-parameters).
Let $F$ be a local field, $G^*=\SO_{2n+1}$ over $F$, $(M^*,P^*)$ a standard parabolic pair of $G^*$, and $\xi:G^*\to G$ an inner twist of $G^*$.
If $M^*$ transfers to $G$, then we take $\xi$ so that $M=\xi(M^*)$ is defined over $F$.
Let $\phi_{M^*}\in\Phi_\bdd(M^*)$ be a generic parameter for $M$, and $\phi\in\Phi_\bdd(G^*)$ its image.
If $G$ is split, the theorem is already proven by Arthur \cite{art13}.
Therefore, we may assume that $G$ is non-quasi-split, so in particular $F\neq \C$.
\begin{lemma}\label{4.6.1}
    Assume that $M^*\subsetneq G^*$ is proper, i.e., $\phi$ is not discrete.
    Assume also that $\phi_{M^*}\in\Phi_{2,\bdd}(M^*)$ is discrete and that $\phi$ is elliptic or exceptional.
    Then for any $x\in\overline{\frakS}_{\phi,\el}$, there exists a lift $x\in\frakS_\phi$ of $\overline{x}$ such that
    \begin{equation*}
        f'_G(\phi,x\inv)=e(G)f_G(\phi,x),
    \end{equation*}
    for any $f\in\calH(G)$.
\end{lemma}
\begin{proof}
    The proof is similar to that of \cite[Case $N$ even of Lemma 4.6.1]{kmsw}.
    The difference is that we divide the case into the following cases:
    \begin{itemize}
        \item $M^*$ is not linear and $(n,\phi)$ is neither (2, of type ($\exc1$)) nor (3, of type ($\exc2$));
        \item $M^*$ is not linear and $(n,\phi)$ is either (2, of type ($\exc1$)) or (3, of type ($\exc2$));
        \item $M^*$ is linear,
    \end{itemize}
    instead of the division into the cases
    \begin{itemize}
        \item $M^*$ is not linear and $N\neq 4$;
        \item $M^*$ is not linear and $N=4$;
        \item $M^*$ is linear,
    \end{itemize}
    in loc. cit., and that we appeal to Lemmas \ref{4.5.1}, \ref{3.7.1}, \ref{3.5.10}, and Propositions \ref{4.4.4}, \ref{4.4.6} \ref{4.4.7}, instead of Lemmas 4.5.1, 3.7.1, 3.5.10, and Propositions 4.4.4, 4.4.6, 4.4.7 of loc. cit., respectively.
\end{proof}
\begin{lemma}\label{4.6.2}
    The assertions 2 and 3 of Theorem \ref{2.6.2} hold for generic parameters.
\end{lemma}
\begin{proof}
    Note that the assertion 3 implies the assertion 2.
    Thanks to \S\ref{2.7} and \S\ref{2.8}, we may assume that $\phi_{M^*}$ is discrete and that $\phi\in\Phi^2_\el(G^*)\sqcup\Phi_\exc(G^*)$.
    By the consequence of \S\ref{2.8}, if $\phi$ is elliptic, then it remains to show that
    \begin{description}
        \item $f'_G(\phi,s\inv)=e(G)f_G(\phi,u^\natural)$ for any $u^\natural\in\frakN_\phi$ and $s\in S_{\phi,\semisimple}$ mapping to the same element in $\frakS_{\phi,\el}$,
    \end{description}
    and if $\phi$ is exceptional, then it remains to show that
    \begin{description}
        \item $f'_G(\phi,s\inv)=e(G)f_G(\phi,u^\natural)$ for any $u^\natural\in\frakN_{\phi,\reg}$ and $s\in S_{\phi,\semisimple}$ mapping to the same element in $\frakS_{\phi,\el}$.
    \end{description}
    They can be proved similarly to \cite[Lemma 4.6.2]{kmsw}.
    The difference is that we utilize Lemmas \ref{2.8.6}, \ref{4.6.1}, and \ref{2.6.5} instead of Lemmas 2.8.6, 4.6.1, and 2.6.5 of loc. cit. respectively.
    Note that in our case the equivalence classes of inner forms, inner twists, and pure inner twists are in bijection naturally.
\end{proof}
This completes the proof of the second and third assertions of Theorem \ref{2.6.2} for generic parameters.
The next two lemmas show the first assertion.
The first lemma treats the case of $\phi_{M^*}\in\Phi_{2,\bdd}(M^*)$:
\begin{lemma}\label{4.6.3}
    Assume that $\phi_{M^*}\in\Phi_{2,\bdd}(M^*)$.
    Then the assertion 1 of Theorem \ref{2.6.2} holds for $\phi_{M^*}$.
\end{lemma}
\begin{proof}
    By Lemmas \ref{2.8.7} and \ref{2.8.4}, we may assume that $\phi\in\Phi_\exc(G^*)$.
    The proof is similar to that of \cite[Case $N$ even of Lemma 4.6.3]{kmsw}.
    The difference is that we appeal to Lemmas \ref{4.4.7}, \ref{3.7.1}, \ref{4.6.3} (split case), and \ref{4.6.2} instead of Lemmas 4.4.7, 3.7.1, 4.6.3 (quasi-split case), and 4.6.2 of loc. cit. respectively.
\end{proof}
Recall that in \S\ref{2.7} we reduce only part 2 and 3 of Theorem \ref{2.6.2} to the case of discrete parameters.
The next lemma is the reduction of part 1, and hence it completes the proof of part 1 for all generic parameters.
As the lemmas above, in a similar way to the equation (4.6.3), Lemma 4.6.4, and Lemma 4.6.5 of loc. cit., one can obtain a surjection
\begin{equation}\label{463}
    R_\phi(M,G)=W_\phi(M,G)/W_\phi^\circ(M,G) \onto R_{\pi_M}(M,G):=W_{\pi_M}(M,G)/W_{\pi_M}^\circ(M,G),
\end{equation}
for $\pi_M\in\Pi_{\phi_{M^*}}(M)$, and can show the following lemmas.
\begin{lemma}\label{4.6.4}
    Let $\phi_{M^*}\in\Phi_\bdd(M^*)$.
    Then the assertion 1 of Theorem \ref{2.6.2} holds for $\phi_{M^*}$.
\end{lemma}
\begin{lemma}\label{4.6.5}
    The homomorphism \eqref{463} is bijective if $\phi_{M^*}$ is relevant.
\end{lemma}

Now Theorem \ref{2.6.2} holds for any generic parameters.
This completes the proof of LIR for generic parameters.

\subsection{The construction of $L$-packets}\label{4.7}
Theorem \ref{1.6.1} for generic parameters (i.e., $L$-parameters) can be proven in the same way as Theorem 1.6.1 of \cite{kmsw}.
Let us roughly review the procedure.
See \cite[\S\S4.7-4.9]{kmsw} for detail.

Let $F$ be a local field, $G^*=\SO_{2n+1}$ over $F$, and $(\xi,z):G^*\to G$ a pure inner twist of $G^*$ as in \S\ref{2.4}.
If $G$ is split, the theorem is already proven by Arthur \cite{art13}.
Therefore, we assume that $G$ is non-quasi-split, so in particular $F\neq \C$.
In the archimedean case, the theorem is known by Langlands and Shelstad, so we assume that $F$ is a $p$-adic field.

First consider the non-discrete parameters.
Since we now have Theorem \ref{2.6.2} for generic parameters $\phi$, the local packets $\Pi_\phi(G)$ and the map $\Pi_\phi(G)\to\Irr(\frakS_\phi,\chi_G)$ for $\phi\in\Phi_\bdd(G^*)\setminus\Phi_2(G^*)$ have already constructed at the end of \S\ref{2.6}.
The remaining assertions are that the map is bijective and that the packets are disjoint and exhaust $\Pi_\temp(G)\setminus\Pi_2(G)$.

The general classification \cite{art93} (cf. \cite[\S3.5]{art13}, \cite[\S4.7]{kmsw}) of $\Pi_\temp(G)$ by harmonic analysis says that there is a bijective correspondence
\begin{equation}\label{cltemp}
    (M,\sigma,\mu) \mapsto \pi_\mu \in \Pi_\temp(G),
\end{equation}
from the $G(F)$-orbits of triples consisting of a Levi subgroup $M\subset G$ over $F$, a discrete series representation $\sigma$ of $M(F)$, and an irreducible representation $\mu$ of the representation-theoretic $R$-group $R(\sigma)=R_\sigma(M,G)$.
Here, we do not need an extension $\widetilde{R}(\sigma)$ of the $R$-group, as explained in \cite[p.158]{art13} or \cite[\S4.7]{kmsw}.
The correspondence can be described explicitly as follows.
For a triple $(M,\sigma,\mu)$, let $P$ be a parabolic subgroup of $G$ with Levi factor $M$.
Choose a family $\{R_P(r,\sigma)\}_{r\in R_\sigma(M,G)} \subset \Aut_{G(F)}(\calI_P^G(\sigma))$ of self-intertwining operators so that an assignment $r \mapsto R_P(r,\sigma)$ is a homomorphism from $R_\sigma(M,G)$ to $\Aut_{G(F)}(\calI_P^G(\sigma))$.
Then we have a representation $R_P(r,\sigma) \circ \calI_P^G(\sigma,g)$ of $R_\sigma(M,G)\times G(F)$ on $\calH_P^G(\sigma)$, for which we shall write $\calR$.
The representation $\pi_\mu$ is characterized by a decomposition
\begin{equation*}
    \calR=\bigotimes_{\mu\in\Irr(R_\sigma(M,G))} \mu^\vee \otimes \pi_\mu.
\end{equation*}
Note that $R_\sigma(M,G)$ is abelian in our case.

Then the bijectivity of the map $\Pi_\phi(G)\to\Irr(\frakS_\phi,\chi_G)$ and the disjointness and exhaustion of $L$-packets in $\Pi_\temp(G)$ follows from the same argument as in \cite[\S4.7]{kmsw}, which we omit.
\begin{proposition}\label{4.7.1}
    Theorem \ref{1.6.1} holds for generic parameters $\phi\in\Phi_\bdd(G^*)\setminus\Phi_2(G^*)$.
\end{proposition}
\begin{proof}
    The proof is similar to that of \cite[Proposition 4.7.1]{kmsw}.
    Note that $\Pi_\phi(G,\Xi)$, $\chi_\Xi$, $S_\phi^\natural$, $N_\phi^\natural$, and $S_\phi^{\natural\natural}(M)$ in loc. cit. should be replaced by $\Pi_\phi(G)$, $\chi_G$, $\frakS_\phi$, $\frakN_\phi$, and $\frakS_\phi(M)$ in our case, respectively.
\end{proof}

Next, before considering discrete parameters, we need some preparation.
Put $T_\el(G)$ to be the set of $G(F)$-conjugacy classes of triples $\tau=(M,\sigma,r)$, where $M\subset G$ is a Levi subgroup, $\sigma$ a unitary discrete series representation of $M(F)$, and $r\in R_\sigma(M,G)$ a regular element.
Here, $r\in R_\sigma(M,G)$ is said to be regular if $\fraka_M^{r=1}:=\{\lambda\in\fraka_M \mid r\lambda=\lambda\}$ coincides with $\fraka_G$.
The set $\Pi_{2,\temp}(G)$ is naturally regarded as a subset of $T_\el(G)$ via $\pi\mapsto(G,\pi,1)$.
The complement set will be denoted by $T_\el^2(G)$.
In general, the trace Paley-Wiener theorem tells that "orbital integrals" of cuspidal functions $f\in\calH_\cusp(G)$ are described by the functions on $T_\el(G)$ given by some intertwining operators.

In order to utilize the local intertwining operator defined in \S\ref{2.5} , we need to consider the set $T_\el^\natural(G)$.
Let $T_\el^{2,\natural}(G)$ be the set of $G(F)$-conjugacy classes of triples $\tau^\natural=(M,\sigma,s)$, where $M\subsetneq G$ is a proper Levi subgroup, $\sigma$ a unitary discrete series representation of $M(F)$, and $s\in \frakS_{\phi_\sigma}(G)$ an element whose image under the surjection $\frakS_{\phi_\sigma}(G)\onto R_\sigma(M,G)$ is regular.
Here, we write $\phi_\sigma$ for the $L$-parameter of $\sigma$, and we also write $\phi_\sigma$ for its image in $\Phi_\bdd(G^*)$.
Then put $T_\el^\natural(G):=\Pi_{2,\temp}(G) \sqcup T_\el^{2,\natural}(G)$.
The surjection $\frakS_{\phi_\sigma}(G)\onto R_\sigma(M,G)$ induces the natural surjection $T_\el^\natural(G)\onto T_\el(G)$.
For $f\in\calH(G)$, put
\begin{align*}
    f_G(\tau^\natural)=\begin{dcases*}
                           \tr\left( R_P(u^\natural,\sigma,\phi_\sigma,\psi_F) \circ \calI_P^G(\sigma,f) \right), & for $\tau^\natural=(M,\sigma,s)\in T_\el^{2,\natural}(G)$, \\
                           \tr\left( \pi(f) \right),                                                              & for $\tau^\natural=\pi\in \Pi_{2,\temp}(G)$,
                       \end{dcases*}
\end{align*}
where $u^\natural\in\frakN_{\phi_\sigma}(M,G)$ is a lift of $s$.
It is independent of the choice of $u^\natural$.
If $\tau^\natural_1$ and $\tau^\natural_2$ in $T_\el^{2,\natural}(G)$ maps to a same element in $T_\el(G)$, then they are of the form $\tau^\natural_1=(M,\sigma,s)$ and $\tau^\natural_2=(M,\sigma,ys)$ for some $y\in\frakS_{\phi_\sigma}(M)$.
Then we shall write $\tau^\natural_2=y \tau^\natural_1$.
In this case, by Lemma \ref{2.5.2}, we have
\begin{equation*}
    f_G(\tau^\natural_1)=\an{y,\sigma} f_G(\tau^\natural_2).
\end{equation*}

Let $\phi\in\Phi_{2,\bdd}(G^*)$ be a discrete generic parameter, and $x\in\frakS_\phi$.
Then by Lemma \ref{1.4.1}, one can construct an endoscopic triple $\frake=(G^\frake,s^\frake,\eta^\frake)$ of $G$ and a generic parameter $\phi^\frake$ of $G^\frake$, to obtain a linear form $\calH(G)\ni f \mapsto f'_G(\phi,x):=f^\frake(\phi^\frake)$.
The trace Paley-Wiener theorem implies that there exists a family $\{c_{\phi,x}(\tau^\natural)\}_{\tau^\natural\in T_\el^\natural(G)} \subset \C$ of complex numbers such that for any $f\in \calH_\cusp(G)$, we have
\begin{equation}\label{tpw}
    f'_G(\phi,x)=e(G) \sum_{\tau\in T_\el(G)} c_{\phi,x}(\tau^\natural) f_G(\tau^\natural),
\end{equation}
and
\begin{equation*}
    c_{\phi,x}(\tau^\natural)=\an{y,\tau}c_{\phi,x}(y\tau^\natural),
\end{equation*}
where $\tau^\natural$ is a lift of $\tau$, and $\an{y,\tau}=\an{y,\sigma}$ for $\tau=(M,\sigma,s)\in T_\el^2(G)$.
Note that the product $c_{\phi,x}(\tau^\natural) f_G(\tau^\natural)$ does not depend on the choice of $\tau^\natural$.
Then one can show an orthogonality relation:
\begin{proposition}\label{4.8.3}
    \begin{enumerate}[(a)]
        \item Let $\phi\in\Phi_{2,\bdd}(G^*)$ be a discrete generic parameter, and $x\in\frakS_\phi$. Then $c_{\phi,x}(\tau^\natural)=0$ for all $\tau^\natural\in T_\el^{2,\natural}(G)$.
        \item Let $\phi_1$, $\phi_2\in\Phi_{2,\bdd}(G^*)$ be discrete generic parameters, and $x_i\in\frakS_{\phi_i}$ ($i=1,2$). Then \begin{align*}
                  \sum_{\pi\in\Pi_{2,\temp}(G)} c_{\phi_1,x_1}(\pi) \overline{c_{\phi_2,x_2}(\pi)} = \begin{dcases*}
                                                                                                         \abs{\overline{\frakS}_{\phi_1}}, & if $\phi_1=\phi_2$ and $x_1=x_2$, \\
                                                                                                         0,                                & otherwise.
                                                                                                     \end{dcases*}
              \end{align*}
    \end{enumerate}
\end{proposition}
\begin{proof}
    The proof is similar to that of \cite[Proposition 4.8.3]{kmsw} or \cite[Lemma 6.5.3]{art13}.
\end{proof}

Finally we shall consider discrete generic parameters, after introducing one lemma.
\begin{lemma}\label{4.9.1}
    Let $\xi:\dot{G}^*\to\dot{G}$ be an inner twist of $\dot{G}^*=\SO_{2n+1}$ over a number field $\dot{F}$, and $\dot{\phi}\in \Phi_2(\dot{G}^*)$ a discrete global parameter.
    For any $\dot{f}\in\calH(\dot{G})$, we have
    \begin{equation}\label{smf}
        \sum_{\dot{\pi}} n_{\dot{\phi}}(\dot{\pi}) \dot{f}_{\dot{G}}(\dot{\pi}) = \frac{1}{\abs{\overline{\frakS}_{\dot{\phi}}}} \sum_{\overline{x}\in\overline{\frakS}_{\dot{\phi}}} \dot{f}'_{\dot{G}}(\dot{\phi},\overline{x}),
    \end{equation}
    where $\dot{\pi}$ runs over irreducible representations of $\dot{G}(\A_{\dot{F}})$, and $n_{\dot{\phi}}(\dot{\pi})$ it the multiplicity of $\dot{\pi}$ in $(R^{\dot{G}}_{\disc, \dot{\phi}}, L^2_{\disc, \dot{\phi}}(\dot{G}(\dot{F})\backslash \dot{G}(\A_{\dot{F}})))$.
\end{lemma}
\begin{proof}
    The proof is similar to that of \cite[Lemma 4.9.1]{kmsw} or \cite[(6.6.6)]{art13}.
\end{proof}

Let $\phi\in\Phi_{2,\bdd}(G^*)$.
Then Proposition \ref{4.4.3} and Lemma \ref{4.1.2} gives us a global data $(\dot{F},\dot{G}^*,\dot{G}, \dot{\phi}, u,v_1,v_2)$ such that $\dot{F}$ is totally real, and
\begin{itemize}
    \item $v_1$ is a finite place, and $v_2$ is a real place;
    \item $\dot{F}_u=F$, $\dot{G}^*_u=G^*$, $\dot{G}_u=G$, and $\dot{\phi}_u=\phi$;
    \item $\dot{G}_v$ is split for $v\notin\{u,v_2\}$, and $\chi_G=\chi_{\dot{G}_{v_2}}$;
    \item $\dot{\phi}_v\in\Phi_{2,\bdd}(\dot{G}^*_v)$ for $v\in\{v_1,v_2\}$;
    \item the canonical map $S_{\dot{\phi}}\to S_{\dot{\phi}_v}$ is isomorphic for $v\in\{u,v_1\}$.
\end{itemize}
Choose $\dot{f} \in \calH(\dot{G})$ so that $\dot{f}$ is of the form $\dot{f}=\bigotimes_v \dot{f}_v$, and that $\dot{f}_u\in\calH_\cusp(G)$.
Then we have the equation \eqref{smf} due to Lemma \ref{4.9.1}.
For $v\notin\{u,v_2\}$, the group $\dot{G}_v$ is split and hence we have ECR \eqref{ecr}.
In the case $v=v_2$, since $\dot{F}_{v_2}=\R$, ECR is known by the work of Shelstad.
At the place $u$, we have the equation \eqref{tpw} since $\dot{f}_u$ is cuspidal.
Substituting them, we obtain the following equation from \eqref{smf}:
\begin{equation*}
    \sum_{\dot{\pi} \in \Pi_\unit(\dot{G}(\A_{\dot{F}}))} n_{\dot{\phi}}(\dot{\pi}) \dot{f}_{\dot{G}}(\dot{\pi})
    =\frac{1}{\abs{\overline{\frakS}_{\dot{\phi}}}} \sum_{\dot{\pi}^u \in \Pi_{\dot{\phi}^u}} \sum_{\pi \in \Pi_{2,\temp}(G)} \sum_{\overline{x}\in \overline{\frakS}_{\dot{\phi}}} \an{\dot{x}^u,\dot{\pi}^u} c_{\phi, \dot{x}_u}(\pi) \dot{f}_{\dot{G}}(\dot{\pi}^u \otimes \pi),
\end{equation*}
where $\Pi_\unit(\dot{G}(\A_{\dot{F}}))$ denotes the set of isomorphism classes of irreducible unitary representations of $\dot{G}(\A_{\dot{F}})$, $\dot{x}\in\frakS_{\dot{\phi}}$ is a lift of $\overline{x}\in \overline{\frakS}_{\dot{\phi}}$, $\dot{x}_v\in\frakS_{\dot{\phi}_v}$ is the image of $\dot{x}$, and
\begin{align*}
    \Pi_{\dot{\phi}^u}         & :=\Set{\dot{\pi}^u=\bigotimes_{v\neq u}\dot{\pi}_v | \dot{\pi}_v\in\Pi_{\dot{\phi}_v},\ \an{-,\dot{\pi}_v}=1 \ \text{for almost all } v}, \\
    \dot{x}^u                  & :=\left( \dot{x}_v \right)_{v\neq u} \in \prod_{v\neq u} \frakS_{\dot{\phi}_v},                                                           \\
    \an{\dot{x}^u,\dot{\pi}^u} & :=\prod_{v\neq u} \an{\dot{x}_v,\dot{\pi}_v}.
\end{align*}

For each place $v$ other than $u$, $v_1$ and $v_2$, fix $\dot{\pi}_v\in \Pi_{\dot{\phi}_v}$ such that $\an{-,\dot{\pi}_v}$ is trivial.
For a real place $v_2$, fix $\dot{\pi}_{v_2}\in \Pi_{\dot{\phi}_{v_2}}$ arbitrarily.
They exist because $\dot{G}_v$ splits if $v\notin\{u,v_1,v_2\}$ and $\dot{F}_{v_2}=\R$.

For any $\mu\in \Irr(\frakS_\phi,\chi_G)$, choose $\dot{\pi}_{v_1} \in \Pi_{\dot{\phi}_{v_1}}$ so that $\mu(\dot{x}_u)\an{\dot{x}^u,\dot{\pi}^u}=1$ for all $\dot{x}\in\frakS_{\dot{\phi}}$, where $\dot{\pi}^u=\bigotimes_{v\neq u} \dot{\pi}_v$.
Such $\dot{\pi}_{v_1}$ does exist because we have $\chi_G=\chi_{\dot{G}_{v_2}}$, $v_1$ is a finite place, $\dot{G}_{v_1}$ is split, and $S_{\dot{\phi}}\to S_{\dot{\phi}_v}$ is isomorphic for $v\in\{u,v_1\}$.
Then for $\pi\in \Pi_{2,\temp}(G)$, we set
\begin{equation*}
    n_{\phi}(\mu,\pi):=n_{\dot{\phi}}\left( \dot{\pi}^u \otimes \pi \right).
\end{equation*}

Then by the same argument of the proof of \cite[Proposition 4.9.2]{kmsw}, we have
\begin{equation}\label{492}
    n_\phi(\mu,\pi)=\frac{1}{\abs{\overline{\frakS}_\phi}} \sum_{\overline{x}\in \overline{\frakS}_\phi} \mu(x)\inv c_{\phi,x}(\pi),
\end{equation}
for any $\pi\in \Pi_{2,\temp}(G)$, where $x\in \frakS_\phi$ denotes for an arbitrary representative of $\overline{x}$.
Moreover, the orthogonality relation (Proposition \ref{4.8.3}) and the formula \eqref{492} implies the following formula:
\begin{proposition}\label{4.9.3}
    Let $\phi, \phi'\in \Phi_{2,\bdd}(G^*)$ be discrete generic parameters, and $\mu\in \Irr(\frakS_\phi,\chi_G)$ and $\mu'\in \Irr(\frakS_{\phi'},\chi_G)$.
    Then we have
    \begin{align*}
        \sum_{\pi\in\Pi_{2,\temp}(G)} n_\phi(\mu,\pi) n_{\phi'}(\mu',\pi)=\begin{dcases*}
                                                                              1, & if  $(\phi,\mu)=(\phi',\mu')$, \\
                                                                              0, & otherwise.
                                                                          \end{dcases*}
    \end{align*}
\end{proposition}
\begin{proof}
    The proof is the similar to that of \cite[Proposition 4.9.3]{kmsw}.
\end{proof}

Since $n_\phi(\mu,\pi)$ is a non-negative integer, Proposition \ref{4.9.3} tells us that for any $\phi\in \Phi_{2,\bdd}(G^*)$ and $\mu\in \Irr(\frakS_\phi, \chi_G)$, there exists a unique $\pi\in \Pi_{2,\temp}(G)$ such that $n_\phi(\mu,\pi)=1$.
We shall write $\pi(\phi,\mu)$ for this representation $\pi$.
Then an assignment $(\phi,\mu)\mapsto \pi(\phi,\mu)$ gives an injective map.
The packet of $\phi\in\Phi_{2,\bdd}(G^*)$ is defined by
\begin{equation*}
    \Pi_\phi:=\Set{\pi(\phi,\mu) | \mu\in \Irr(\frakS,\chi_G)},
\end{equation*}
and the associated character of the component group is defined by
\begin{equation*}
    \an{-,\pi(\phi,\mu)}:=\mu.
\end{equation*}
One can now easily see that the packets are disjoint and the map $\pi_\phi\to\Irr(\frakS_\phi,\chi_G)$ is bijective.

Let $\phi\in\Phi_{2,\bdd}(G^*)$ and $x\in\frakS_\phi$.
Combining the formula \eqref{492} with Proposition \ref{4.9.3}, we have
\begin{align*}
    c_{\phi,x}(\pi)=\begin{dcases*}
                        \mu(x)=\an{x,\pi}, & if $\pi=\pi(\phi,\mu)$, \\
                        0,                 & otherwise.
                    \end{dcases*}
\end{align*}
Then the formula \eqref{tpw} is none other than the endoscopic character relation for cuspidal function $f\in\calH_\cusp(G)$.
\begin{proposition}\label{4.9.4}
    For general $f\in\calH(G)$, we have ECR:
    \begin{equation*}
        f'(\phi,x)=e(G) \sum_{\pi\in\Pi_\phi} \an{x,\pi} f_G(\pi).
    \end{equation*}
\end{proposition}
\begin{proof}
    The proof is similar to that of \cite[Corollary 6.7.4]{art13}.
\end{proof}

Proposition \ref{4.9.4} characterizes the packet $\Pi_\phi$ and the bijection $\Pi_\phi\to\Irr(\frakS_\phi,\chi_G)$.
Thus they does not depend on the choice of $\dot{\phi}$ or $\dot{\pi}^u$, but only on $\phi$.
By the argument between Theorems \ref{1.6.1} and \ref{3.12}, we now have the packet $\Pi_\phi$ and the bijection $\Pi_\phi\to\Irr(\frakS_\phi,\chi_G)$ for all generic parameters $\phi\in\Phi(G^*)$.

It remains to show that the packets exhaust $\Pi_{2,\temp}(G)$.
Let $\pi\in\Pi_{2,\temp}(G)$.
Take a globalization $\dot{\pi}$ of $\pi$ as in Lemma \ref{4.2.1}.
By the decomposition \eqref{gap}, we have the $A$-parameter $\dot{\phi}$ of $\dot{\pi}$.
In particular, we have $n_{\dot{\phi}}(\dot{\pi})\neq 0$.
Since $\dot{\pi}_{v'}$ has sufficiently regular infinitesimal character for some real place $v'$, the parameter $\dot{\phi}$ is generic.
Thus Theorem \ref{2.6.2} and Hypothesis \ref{3.6.3} hold for $\dot{\phi}$, so does Theorem \ref{5.0.5} for $\dot{\phi}$.
(Note that \S \ref{5} is independent from the exhaustion.)
Hence we have $\pi\in\Pi_{\dot{\phi}_u}$.
It finally remains to show that the parameter $\dot{\phi}_u$ is bounded and discrete. Suppose first that $\dot{\phi}_u$ is not bounded. The packet $\Pi_{\dot{\phi}_u}$ has already been defined after Theorem \ref{1.6.1}. The Langlands classification of irreducible admissible representations of connected reductive groups over $p$-adic groups implies that the packet does not contain any tempered representation. This contradicts that $\pi\in \Pi_{\dot{\phi}_u}$. Hence $\dot{\phi}_u$ is bounded. Suppose next that $\dot{\phi}_u$ is not discrete. The packet $\Pi_{\dot{\phi}_u}$ has already been defined at the end of \S\ref{2.6}. The general classification of irreducible tempered representations \eqref{cltemp} implies that the packet contains only discrete series representations. This contradicts that $\pi\in \Pi_{\dot{\phi}_u}$. Hence $\dot{\phi}_u$ is discrete.
Therefore, we have $\dot{\phi}_u\in\Phi_{2,\bdd}(G^*)$.
This completes the proof of the local classification theorem for generic parameters.

\section{The proof of the global theorem}\label{5}
Assuming the existence of local $A$-packets and LIR, Theorem \ref{1.7.1} can be proven in the same way as Theorem 1.7.1 of \cite{kmsw}.
In particular, since the assumptions hold for generic parameters, generic part of the theorem holds true.
In this section we record the statements.
See \cite[\S5]{kmsw} for the proof.

Let $F$ be a number field, $G^*=\SO_{2n+1}$ over $F$, and $G$ an inner form of $G^*$.
If $G$ is split, the theorem is already proven by Arthur \cite{art13}.
Therefore, we assume that $G$ is non-quasi-split.
Recall from \S\ref{3.1} \eqref{gap} the decomposition
\begin{equation*}
    \begin{gathered}
        L^2_{\disc}(G(F)\backslash G(\A_F))=\bigoplus_{\psi\in \Psi(G^*)} L^2_{\disc,\psi}(G(F)\backslash G(\A_F)),\\
        \tr R_{\disc}(f)=\sum_{\psi\in\Psi(G^*)} R_{\disc,\psi}(f), \qquad f\in\calH(G).
    \end{gathered}
\end{equation*}
Let us introduce a hypothesis on local $A$-packets:
\begin{hypothesis}\label{3.6.3}
    Let $\psi\in\Psi(G^*)$ be a global parameter for $G$.
    We have the local packet $\Pi_{\psi_v}=\Pi_{\psi_v}(G_v)$ and the map $\Pi_{\psi_v} \to \Irr(\frakS_{\psi_v},\chi_{G_v})$ satisfying ECR \eqref{ecr} for all places $v$ of $F$.
\end{hypothesis}
\begin{theorem}\label{5.0.5}
    Let $\psi\in\Psi(G^*)$ be a global parameter for $G$.
    Assume that Theorem \ref{2.6.2} and Hypothesis \ref{3.6.3} hold for $\psi$.
    \begin{enumerate}
        \item If $\psi\notin \Psi_2(G^*)$, then \begin{equation*}
                  L^2_{\disc,\psi}(G(F)\backslash G(\A_F))=0.
              \end{equation*}
        \item If $\psi\in \Psi_2(G^*)$, then \begin{equation*}
                  L^2_{\disc,\psi}(G(F)\backslash G(\A_F))=\bigoplus_{\pi\in\Pi_\psi(G,\varepsilon_\psi)} \pi.
              \end{equation*}
    \end{enumerate}
\end{theorem}
\begin{proof}
    The proof is similar to that of \cite[Theorem 5.0.5]{kmsw}.
\end{proof}
Since Theorem \ref{2.6.2} and Hypothesis \ref{3.6.3} hold true if $\psi=\phi\in\Phi(G^*)$ is generic, the global classification theorem is now established for the generic part.

\end{document}